\documentclass[a4, 12pt]{amsart}
\usepackage[mathscr]{eucal}
\usepackage{amssymb}
\usepackage{latexsym}
\usepackage{amsthm}
\usepackage{color}
\theoremstyle{plain}
\newtheorem{theorem}{Theorem}[section]

\newtheorem{example}{Example}[section]

\newtheorem{lemma}{Lemma}[section]
\newtheorem{proposition}{Proposition}[section]

\makeatletter
\@addtoreset{equation}{section}

\title[Complete Lagrangian self-shrinkers  in $\mathbf R^4$]
{Complete Lagrangian self-shrinkers  in $\mathbf R^4$}
\author [Q. -M. Cheng, H.  Hori  and G. Wei]{Qing-Ming Cheng*, Hiroaki Hori and Guoxin Wei*}
\dedicatory{Dedicated to Professor Yuan-Long Xin for his 75th birthday}
\address{Qing-Ming Cheng \\  \newline \indent Department of Applied Mathematics, Faculty of Sciences,
\newline \indent Fukuoka University, Fukuoka  814-0180, Japan.  \newline \indent cheng@fukuoka-u.ac.jp}
\address{Hori Hiroaki \\  \newline \indent Department of Applied Mathematics, Graduate School of Sciences,
\newline \indent Fukuoka University, Fukuoka  814-0180, Japan.  \newline \indent sd150501@cis.fukuoka-u.ac.jp}
\address{Guoxin Wei \\  School of Mathematical Sciences, South China Normal University,
\newline \indent 510631, Guangzhou,  China, weiguoxin@tsinghua.org.cn}

\begin{document}
\maketitle

\begin{abstract}
The purpose of this paper is to study complete  self-shrinkers of mean curvature flow in Euclidean spaces. In the paper,
we give a complete classification
for 2-dimensional complete  Lagrangian self-shrinkers in Euclidean space $\mathbb R^4$  with constant squared norm of the
second fundamental form.

\end{abstract}


\footnotetext{The first author was partially  supported by JSPS Grant-in-Aid for Scientific Research (B):  No.16H03937.
The third  author was partly supported by grant No. 11371150 of NSFC.}

\section{introduction}
\vskip2mm
\noindent
Let  $X: M\to \mathbb{R}^{n+p}$  be an $n$-dimensional submanifold in the ($n+p$)-dimensional
Euclidean space $\mathbb{R}^{n+p}$. A family of $n$-dimensional submanifolds
$X(\cdot, t):M\to  \mathbb{R}^{n+p}$ is called a mean curvature flow if they
satisfy $X(\cdot, 0)=X(\cdot)$ and
\begin{equation}
\dfrac{\partial X(p,t)}{\partial t}=\vec H(p,t), \quad (p,t)\in M\times [0,T),
\end{equation}
where $\vec H(p,t)$ denotes the mean curvature vector  of submanifold  $M_t=X(M,t)$ at point $X(p,t)$.
The mean curvature flow has been used to model various things  in material sciences and physics
such as cell, bubble growth and so on.
The study of the mean curvature  flow from the perspective of partial differential
equations commenced with Huisken's paper  \cite{H1}  on the flow of convex
hypersurfaces.
One of the most important problems in the  mean curvature flow is to understand
the possible singularities that the flow goes through.  A key starting point
for singularity analysis is Huisken's monotonicity formula, the monotonicity
implies that the flow is asymptotically self-similar near a given type I singularity. Thus, it
 is modeled  by  self-shrinking solutions of the flow.

\noindent
An  $n$-dimensional  submanifold  $X: M\rightarrow \mathbb{R}^{n+p}$  in the $(n+p)$-dimensional
Euclidean space $\mathbb{R}^{n+p}$  is called a self-shrinker if it satisfies
\begin{equation*}
\vec H+ X^{\perp}=0,
\end{equation*}
where  $X^{\perp}$ denotes the normal part of the position vector $X$.
It is known that self-shrinkers play an important role in the study of the mean curvature flow because
they describe all possible  blow-ups at a given singularity of the  mean curvature flow.

\noindent
For complete
self-shrinkers with co-dimension $1$,  Abresch and Langer \cite{AL}  classified
closed self-shrinker curves in $\mathbb{R}^2$ and showed that the  round circle is the only embedded self-shrinker.
Huisken \cite{H2, H3},  Colding and Minicozzi \cite{CM1}  have proved that  if $X: M\rightarrow \mathbb{R}^{n+1}$
 is  an $n$-dimensional  complete  embedded self-shrinker
in $\mathbb{R}^{n+1}$ with  mean curvature $H\geq 0$ and  with polynomial volume growth,
then $X: M\rightarrow \mathbb{R}^{n+1}$  is  isometric to  $\mathbb{R}^{n}$,  or the  round sphere $S^{n}(\sqrt{n})$, or a cylinder $S^m (\sqrt{m})\times \mathbb{R}^{n-m}$, $1\leq m\leq n-1$.   Halldorsson in  \cite{H}
proved that there exist complete self-shrinking curves $\Gamma$ in $\mathbb R^2$,
which is contained in an annulus around the origin  and whose image is dense in the  annulus. Furthermore,
Ding and Xin \cite{DX1}, X. Cheng and Zhou \cite{CZ}  proved that
a  complete self-shrinker has  polynomial volume growth if and only if it is proper.
Thus, the  condition on polynomial volume growth in \cite{H3} and \cite{CM1}
is essential since  these complete self-shrinking curves $\Gamma$  of Halldorsson  \cite{H}  are not proper and
for any integer $n>0$,
$\Gamma \times \mathbb R^{n-1}$ is  a complete self-shrinker without polynomial volume growth in $\mathbb R^{n+1}$.

\vskip1mm
\noindent
As for the study on the rigidity of complete self-shrinkers, many important works have been done
(cf. \cite{CL}, \cite{CP}, \cite{CW}, \cite{CO}, \cite{DX1, DX2},  \cite{LW}  and so on).
 In particular, Cheng and Peng in \cite{CP}  proved that for an $n$-dimensional complete self-shrinker
 $X:M^n\rightarrow \mathbb{R}^{n+1} $  with   $\inf H^2>0$,
if  the squared norm $S$ of the second fundamental form is constant, then $M^n$ is  isometric to one of the following:
\begin{enumerate}
\item $S^n(\sqrt{n})$,
\item $S^m(\sqrt m)\times\mathbb{R}^{n-m}\subset \mathbb{R}^{n+1}$.
\end{enumerate}
Furthermore, Ding and Xin \cite{DX2}  studied 2-dimensional complete self-shrinkers with polynomial volume growth
and with constant  squared norm $S$ of the second fundamental form. They have proved
that a $2$-dimensional complete self-shrinker  $X: M\rightarrow \mathbb{R}^{3}$   with polynomial volume growth
is isometric to one of the following:
\begin{enumerate}
\item $\mathbb{R}^{2}$,
\item
 $S^1 (1)\times \mathbb{R}$
\item  $S^{2}(\sqrt{2})$,
\end{enumerate}
if $S$  is constant.
Recently,  Cheng and Ogata \cite{CO}
have removed both the assumption on polynomial volume growth in the above theorem of Ding and Xin \cite{DX2}
and the assumption $\inf H^2>0$ in  the theorem of Cheng and Peng \cite{CP} for $n=2$.

\vskip2mm
\noindent
It is natural to ask  the following problems:

\vskip1mm
\noindent
{\bf Problem 1}. To classify
2-dimensional complete self-shrinkers in $\mathbb R^4$
if  the squared norm $S$ of the second fundamental form is constant.

\vskip2mm
\noindent
It is well-known that the unit sphere $S^2(1)$, the Clifford torus $S^1(1)\times S^1(1)$ , the  Euclidean plane $\mathbb R^2$ and  the cylinder $S^1(1)\times \mathbb{R}^{1}$ are the canonical  self-shrinkers in $\mathbb R^4$.
Besides the standard examples, there are many examples of complete  self-shrinkers in $\mathbb R^4$. For examples, compact minimal
surfaces in the sphere $S^3(2)$ are compact self-shrinkers in $\mathbb R^4$. Further, Anciaux \cite{A}, Lee and Wang \cite{LeW}, Castro and  Lerma \cite{CLe} constructed many compact self-shrinkers in  $\mathbb R^4$ (cf.  Section 3).  Except  the canonical  self-shrinkers in $\mathbb R^4$,
the known examples of complete self-shrinkers in $\mathbb R^4$ do not have the constant squared norm $S$ of the second fundamental form.

\noindent
Since the above problem is very difficult,  one may consider the special case of complete Lagrangian self-shrinkers in   $\mathbb R^4$ first. Here we have  identified $\mathbb R^{2n}$ with $\mathbb C^n$ and let us recall the definition of Lagrangian submanifolds.  A submanifold $X: M\rightarrow\mathbb{R}^{2n}$ is called a Lagrangian submanifold if  $J(T_pM)=T^{\perp}_pM$, for any $p\in M$, where $J$ is the complex structure of $\mathbb{R}^{2n}$,  $T_pM$
and $T_p^{\perp}M$ denote the tangent space and the normal space at $p$.

\vskip2mm
\noindent
It is known that the mean curvature flow preserves the Lagrangian property, which  means that, if the initial submanifold $X: M\to \mathbb{R}^{2n}$ is Lagrangian, then the mean curvature flow $X(\cdot, t):M\to  \mathbb{R}^{2n}$ is also Lagrangian. Lagrangian submanifolds are a class
of important submanifolds in geometry of submanifolds and they also have many  applications in  many other fields of differential geometry. For instance, the existence of special Lagrangian submanifolds in Calabi-Yau manifolds attracts a lot of  attention since it
plays a critical role in the $T$-duality formulation of Mirror symmetry of Strominger-Yau-Zaslow \cite{SYZ}. In particular, recently, the study on complete Lagrangian self-shrinkers of mean curvature flow has attracted much attention. Many
important examples of compact Lagrangian self-shrinkers are constructed (see Section 3 and cf. \cite{A, CLe, LeW}). It was proved by
Smoczyk \cite{S} that there are no Lagrangian self-shrinkers, which are  topological spheres, in $\mathbb{R}^{2n}$. In \cite{CLe1}, Castro
and Lerma gave a classification of  Hamiltonian stationary Lagrangian self-shrinkers in $\mathbb{R}^{4}$ and in \cite{CLe}, they proved
that Clifford torus  $S^1(1)\times S^1(1)$ is the only compact Lagrangian self-shrinker with $S\leq 2$ in  $\mathbb{R}^{4}$ if
the Gaussian curvature does not change sign. Here, it is noticeable that compactness is important  since
 the Gauss-Bonnet theorem is the key in their proof. In fact,
Since $X: M^2\to \mathbb{R}^{4}$  is compact, according to the Gauss-Bonnet theorem, we have
$$
8\pi (1-g)=2\int_MKdA=\int_M(H^2-S)dA=\int_M(2-S)dA.
$$
Hence, $X: M^2\to \mathbb{R}^{4}$ is a torus and $K\equiv 0$, $S\equiv 2$.
Recently,  Li and Wang \cite{LW}  have removed the condition on Gaussian curvature. They proved that
Clifford torus  $S^1(1)\times S^1(1)$ is the only compact Lagrangian self-shrinker with $S\leq 2$ in  $\mathbb{R}^{4}$.
Furthermore, they proved that
Clifford torus  $S^1(1)\times S^1(1)$ is the only compact Lagrangian self-shrinker with constant squared norm $S$
of the second fundamental form in  $\mathbb{R}^{4}$. The Gauss-Bonnet theorem is still the key in their proof.
Since  the Euclidean plane $\mathbb R^2$ and  the cylinder $S^1(1)\times \mathbb{R}^{1}$ are complete and non-compact
Lagrangian self-shrinkers with $S=$ constant in  $\mathbb{R}^{4}$, we may ask the following problem:
\vskip2mm
\noindent
{\bf Problem 2}.
Let $X: M^2\to \mathbb{R}^{4}$ be a
2-dimensional complete  Lagrangian self-shrinker  in $\mathbb R^4$.
If  the squared norm $S$ of the second fundamental form is constant,   is $X: M^2\to \mathbb{R}^{4}$
isometric to one of the following
\begin{enumerate}
\item $\mathbb R^2$,
\item $S^1(1)\times \mathbb{R}^{1}$,
\item $S^1(1)\times S^1(1)$?
\end{enumerate}

\vskip2mm
\noindent
It is our motivation to solve  the above problem.  In this paper, we solve the problem 2. In fact, we prove the following:
\begin{theorem}
 Let $X: M^2\to \mathbb{R}^{4}$ be a
2-dimensional complete  Lagrangian self-shrinker in $\mathbb R^4$.
If  the squared norm $S$ of the second fundamental form is constant, then
  $X: M^2\to \mathbb{R}^{4}$ is
isometric to one of
\begin{enumerate}
\item $\mathbb R^2$,
\item $S^1(1)\times \mathbb{R}^{1}$,
\item $S^1(1)\times S^1(1)$.
\end{enumerate}
\end{theorem}

\vskip5mm
\section {Preliminaries}
\vskip2mm

\noindent
Let $X: M\rightarrow\mathbb{R}^{2n}$ be an
n-dimensional connected submanifold of the $2n$-dimensional Euclidean space
$\mathbb{R}^{2n}$. We choose a local orthonormal frame field
$\{e_A\}_{A=1}^{2n}$ in $\mathbb{R}^{2n}$ with dual coframe field
$\{\omega_A\}_{A=1}^{2n}$, such that, restricted to $M$,
$e_1,\cdots, e_n$ are tangent to $M^n$. 
%It is known that
%$X: M\rightarrow\mathbb{R}^{2n}$ is called a Lagrangian submanifold if
%$J(T_pM)=T^{\perp}_pM$, for any $p\in M$, where $J$ is the complex structure of $\mathbb{R}^{2n}$,  $T_pM$
%and $T_p^{\perp}M$ denote the tangent space and the normal space at $p$.
 Here we have  identified $\mathbb R^{2n}$ with $\mathbb C^n$.
\noindent
For a  Lagrangian submanifold $X: M\rightarrow\mathbb{R}^{2n}$,  we choose an adapted Lagrangian  frame field
$$
e_1, e_2, \cdots, e_n \ \ \text{\rm and } \ e_{1^*}=Je_1, e_{2^*}=Je_2, \cdots, e_{n^*}=Je_n.
$$
From now on,  we use the following conventions on the ranges of indices:
$$
 1\leq i,j,k,l\leq n, \quad 1\leq
\alpha,\beta,\gamma\leq n
$$
and $\sum_{i}$ means taking  summation from $1$ to $n$ for $i$.
Then we have
\begin{equation*}
dX=\sum_i\limits \omega_i e_i,
\end{equation*}
\begin{equation*}
de_i=\sum_j\limits \omega_{ij}e_j+\sum_{\alpha}\limits
\omega_{i\alpha^{\ast}}e_{\alpha^{\ast}},
\end{equation*}
\begin{equation*}
de_{\alpha^{\ast}}=\sum_i\limits\omega_{\alpha^{\ast} i}e_i+\sum_\beta\limits
\omega_{\alpha^{\ast}\beta^{\ast}}e_{\beta^{\ast}} ,
\end{equation*}
where $\omega_{ij}$ is the Levi-Civita connection of $M$,
$\omega_{\alpha^{\ast}\beta^{\ast}}$ is the normal connection of
$T^{\perp}M$.\newline
By  restricting  these forms to $M$,  we have
\begin{equation}\label{2.1-1}
\omega_{\alpha^{\ast}}=0 \quad \text{for}\quad  1\leq\alpha\leq n
\end{equation}
and the induced Riemannian metric of $M$ is written as
$ds^2_M=\sum_i\limits\omega^2_i$.
Taking exterior derivatives of \eqref{2.1-1}, we have
\begin{equation*}
0=d\omega_{\alpha^{\ast}}=\sum_i \omega_{\alpha^{\ast} i}\wedge\omega_i.
\end{equation*}
By Cartan's lemma, we have
\begin{equation*}\label{2.1-2}
\omega_{i\alpha^{\ast}}=\sum_j h^{\alpha^{\ast}}_{ij}\omega_j,\quad
h^{\alpha^{\ast}}_{ij}=h^{\alpha^{\ast}}_{ji}.
\end{equation*}
Since  $X: M\rightarrow\mathbb{R}^{2n}$ is a Lagrangian submanifold, we have
\begin{equation}\label{2.1-2}
h^{p^{\ast}}_{ij}=h^{p^{\ast}}_{ji}=h^{i^{\ast}}_{pj}, \ \ \text{\rm for any } \  i, j, p.
\end{equation}
$$
h=\sum_{i,j,p}h^{p^{\ast}}_{ij}\omega_i\otimes\omega_j\otimes e_{p^{\ast}}
$$
and
$$
\vec{H}=\sum_{p}\limits H^{p^{\ast}} e_{p^{\ast}}=\sum_{p}\limits \sum_i\limits h^{p^{\ast}}_{ii}e_{p^{\ast}}
$$
are called  the second fundamental form and the mean curvature vector field of $X: M\rightarrow\mathbb{R}^{2n}$, respectively.
Let $S=\sum_{i,j,p}\limits (h^{p^{\ast}}_{ij})^2$ be  the squared norm
of the second fundamental form and $H=|\vec{H}|$ denote the mean
curvature of $X: M\rightarrow\mathbb{R}^{2n}$.
The induced structure equations of $M$ are given by
\begin{equation*}
d\omega_{i}=\sum_j \omega_{ij}\wedge\omega_j, \quad  \omega_{ij}=-\omega_{ji},
\end{equation*}
\begin{equation*}
d\omega_{ij}=\sum_k \omega_{ik}\wedge\omega_{kj}-\frac12\sum_{k,l}
R_{ijkl} \omega_{k}\wedge\omega_{l},
\end{equation*}
where $R_{ijkl}$ denotes components of the curvature tensor of $M$.
Hence,
the Gauss equations are given by
\begin{equation}\label{eq:2.1-3}
R_{ijkl}=\sum_{p}\left(h^{p^{\ast}}_{ik}h^{p^{\ast}}_{jl}-h^{p^{\ast}}_{il}h^{p^{\ast}}_{jk}\right),\ \ R_{ik}=\sum_{p} H^{p^{\ast}}
h^{p^{\ast}}_{ik}-\sum_{p,j}h^{p^{\ast}}_{ij}h^{p^{\ast}}_{jk}.
\end{equation}
%\begin{equation*}
%R_{ik}=\sum_{p} H^{p^{\ast}}
%h^{p^{\ast}}_{ik}-\sum_{p,j}h^{p^{\ast}}_{ij}h^{p^{\ast}}_{jk}.
%\end{equation*}
Letting $R_{p^{\ast}q^{\ast} ij}$ denote the curvature tensor
of the normal connection $\omega_{p^{\ast}q^{\ast}}$ in the normal bundle
of $X:M\rightarrow \mathbb{R}^{2n}$, then Ricci equations are given
by
\begin{equation}\label{eq:16-5}
R_{p^{\ast}q^{\ast}kl}=\sum_i\left(h^{p^{\ast}}_{ik}h^{q^{\ast}}_{il}-h^{p^{\ast}}_{il}h^{q^{\ast}}_{ik}\right).
\end{equation}
Defining the
covariant derivative of $h^{p^{\ast}}_{ij}$ by
\begin{equation}\label{2.1-6}
\sum_{k}h^{p^{\ast}}_{ijk}\omega_k=dh^{p^{\ast}}_{ij}+\sum_kh^{p^{\ast}}_{ik}\omega_{kj}
+\sum_k h^{p^{\ast}}_{kj}\omega_{ki}+\sum_{q}
h^{q^{\ast}}_{ij}\omega_{p^{\ast}q^{\ast}},
\end{equation}
we obtain the Codazzi equations
\begin{equation}\label{eq:16-5}
h_{ijk}^{p^{\ast}}=h_{ikj}^{p^{\ast}}=h^{i^{\ast}}_{pjk}.
\end{equation}
By taking exterior differentiation of \eqref{2.1-6}, and
defining
\begin{equation}\label{eq:16-3}
\sum_lh^{p^{\ast}}_{ijkl}\omega_l=dh^{p^{\ast}}_{ijk}+\sum_lh^{p^{\ast}}_{ljk}\omega_{li}
+\sum_lh^{p^{\ast}}_{ilk}\omega_{lj}+\sum_l h^{p^{\ast}}_{ijl}\omega_{lk}
+\sum_{q} h^{q^{\ast}}_{ijk}\omega_{{q^{\ast}p^{\ast}}},
\end{equation}
we have the following Ricci identities:
\begin{equation}\label{eq:16-4}
h^{p^{\ast}}_{ijkl}-h^{p^{\ast}}_{ijlk}=\sum_m
h^{p^{\ast}}_{mj}R_{mikl}+\sum_m h^{p^{\ast}}_{im}R_{mjkl}+\sum_{q}
h^{q^{\ast}}_{ij}R_{{q^{\ast}p^{\ast}}kl}.
\end{equation}
Defining
\begin{equation}\label{eq:16-6}
\begin{aligned}
\sum_mh_{ijklm}^{p^{\ast}}\omega_m&=dh_{ijkl}^{p^{\ast}}+\sum_mh_{mjkl}^{p^{\ast}}\omega_{mi}
+\sum_mh_{imkl}^{p^{\ast}}\omega_{mj}+\sum_mh_{ijml}^{p^{\ast}}\omega_{mk}\\
&\ \ +\sum_mh_{ijkm}^{p^{\ast}}\omega_{ml}
+\sum_{m}h_{ijkl}^{m^{\ast}}\omega_{m^{\ast}p^{\ast}}
\end{aligned}
\end{equation}
and taking exterior differentiation of  \eqref{eq:16-3}, we get
\begin{equation}
\begin{aligned}
h_{ijkln}^{p^{\ast}}-h_{ijknl}^{p^{\ast}}&=\sum_{m} h_{mjk}^{p^{\ast}}R_{miln}
+ \sum_{m}h_{imk}^{p^{\ast}}R_{mjln}+ \sum_{m}h_{ijm}^{p^{\ast}}R_{mkln}\\
&\ \ +\sum_{m}h_{ijk}^{m^{\ast}}R_{m^{\ast}p^{\ast}ln}.
\end{aligned}
\end{equation}
For the  mean curvature vector field $\vec{H}=\sum_{p}
H^{p^{\ast}} e_{p^{\ast}}$, we define
\begin{equation}\label{2.1-14}
\sum_i H^{p^{\ast}}_{,i}\omega_i=dH^{p^{\ast}}+\sum_{q}
H^{q^{\ast}}\omega_{q^{\ast}p^{\ast}},
\end{equation}
\begin{equation}\label{2.1-15}
\sum_j H^{p^{\ast}}_{,ij}\omega_j=dH^{p^{\ast}}_{,i}+\sum_j
H^{p^{\ast}}_{,j}\omega_{ji}+\sum_{q}
H^{q^{\ast}}_{,i}\omega_{q^{\ast}p^{\ast}},
\end{equation}
\begin{equation}\label{2.1-15}
|\nabla^{\perp}\vec H|^2=\sum_{i,p }(H^{p^{\ast}}_{,i})^2,\ \ \ \  \Delta^\perp H^{p^{\ast}}=\sum_i H^{p^{\ast}}_{,ii}.
\end{equation}
%\begin{equation*}
%\Delta^\perp H^{p^{\ast}}=\sum_i H^{p^{\ast}}_{,ii}.
%\end{equation*}
For a smooth function $f$, the $\mathcal{L}$-operator is defined by
\begin{equation}
\mathcal{L}f=\Delta f-\langle X,\nabla f\rangle,
\end{equation}
where $\Delta$ and $\nabla$ denote the Laplacian and the gradient
operator, respectively.
\vskip2mm
\noindent
Formulas in the following Lemma 2.1 may be found in several papers, for examples, \cite{CL, CP, LW, LWe}.
Since many calculations in their proof are used in this paper, we also provide the proofs for reader's convenience.
\vskip1mm
\noindent
If $X:M^2\rightarrow \mathbb{R}^4$ is a  self-shrinker, then we have
\begin{equation}\label{eq:16-8}
H^{p^{\ast}}=-\langle X,e_{p^{\ast}}\rangle, \ \ p=1, 2.
\end{equation}
From \eqref{eq:16-8}, we can get
\begin{equation}
H^{p^\ast}_{,i}=\nabla_{i}H^{p^{\ast}}=-\nabla_{i}\langle X,e_{p^{\ast}}\rangle
=\sum_{j}h_{ij}^{p^{\ast}}\langle X,e_{j}\rangle.
\end{equation}
Since
\begin{equation}
\nabla_{i}|X|^{2}=2\langle X,e_{i}\rangle,
\end{equation}
we have the following equations from \eqref{eq:16-8}
\begin{equation}
\aligned
\nabla_{j}\nabla_{i}|X|^{2}
=&2\langle e_{i},e_{j}\rangle+2\langle X,X_{ij}\rangle \\
=&2\delta_{ij}+2\langle X,\sum_{p}h_{ij}^{p^{\ast}}e_{p^{\ast}}\rangle\\
=&2\delta_{ij}-2 \sum_{p}h_{ij}^{p^{\ast}}H^{p^{\ast}},
\endaligned
\end{equation}

\begin{equation}\label{eq:16-9}
\aligned
\nabla_{j}\nabla_{i}H^{p^{\ast}}
 =&\nabla_{j}(\sum_{k}h_{ik}^{p^{\ast}}\langle X,e_{k}\rangle)\\
=&\sum_{k}h_{ikj}^{p^{\ast}}\langle X,e_{k}\rangle+h_{ij}^{p^{\ast}}
+\sum_{k}h_{ik}^{p^{\ast}}\sum_{q}h_{jk}^{q^{\ast}}\langle X,e_{q^{\ast}}\rangle\\
=&\sum_{k}h_{ikj}^{p^{\ast}}\langle X,e_{k}\rangle+h_{ij}^{p^{\ast}}-\sum_{k,q}h_{ik}^{p^{\ast}}h_{jk}^{q^{\ast}}H^{q^{\ast}}.
\endaligned
\end{equation}
By a direct calculation, from \eqref{eq:16-8} and \eqref{eq:16-9}, we have
\begin{equation}\label{eq:16-13}
\mathcal{L}H^{p^{\ast}}=\sum_{k}H_{,kk}^{p^{\ast}}-\langle X,\sum_{k}H_{,k}^{p^{\ast}}e_{k}\rangle=H^{p^{\ast}}
-\sum_{i,j,q}h_{ij}^{p^{\ast}}h_{ij}^{q^{\ast}}H^{q^{\ast}}.
\end{equation}
From the definition of the self-shrinker, we get
\begin{equation}
\aligned
\frac{1}{2}\mathcal{L}
|X|^{2}=2-H^{2}-\langle X,\sum_{i}\langle X,e_{i}\rangle e_{i}\rangle=2-H^{2}-|X^{\top}|^{2}=2-|X|^{2}.
\endaligned
\end{equation}
Since $X:M^2\rightarrow \mathbb{R}^4$ is a 2-dimensional Lagrangian self-shrinker,
we know
\begin{equation}\label{eq:16-11}
R_{ijkl}=K(\delta_{ik}\delta_{jl}-\delta_{il}\delta_{jk})=R_{i^{\ast}j^{\ast}kl},
\end{equation}
where $K=\dfrac12(H^2-S)$ is the Gaussian  curvature of $X:M^2\rightarrow \mathbb{R}^4$. \newline
According to  \eqref{eq:2.1-3}, \eqref{eq:16-5}, \eqref{eq:16-4}, \eqref{eq:16-11},  we have
\begin{equation}\label{eq:16-14}
\aligned
\mathcal{L} h_{ij}^{p^{\ast}}
=&\sum_{k}h_{ijkk}^{p^{\ast}}-\langle X,\sum_{k}\nabla_{k}h_{ij}^{p^{\ast}}e_{k}\rangle \\
=&\sum_{k}h_{ikjk}^{p^{\ast}}-\sum_{k}h_{ijk}^{p^{\ast}}\langle X,e_{k}\rangle\\
=&\sum_{m,k}h_{mk}^{p^{\ast}}R_{mijk}+ \sum_{m,k}h_{im}^{p^{\ast}}R_{mkjk}\\
&+ \sum_{q,k}h_{ik}^{q^{\ast}}R_{q^{\ast}p^{\ast}jk}
+\sum_{k}h_{ikkj}^{p^{\ast}}-\sum_{k}h_{ijk}^{p^{\ast}}\langle X,e_{k}\rangle\\
=&K \sum_{m,k}h_{mk}^{p^{\ast}}(\delta_{mj}\delta_{ik}-\delta_{mk}\delta_{ij})
+ K \sum_{m,k}h_{im}^{p^{\ast}}(\delta_{mj}\delta_{kk}-\delta_{mk}\delta_{kj})\\
&+K\sum_{q,k} h_{ik}^{q^{\ast}}(\delta_{qj}\delta_{pk}-\delta_{qk}\delta_{pj})
+H_{,ij}^{p^{\ast}}-\sum_{k}h_{ijk}^{p^{\ast}}\langle X,e_{k}\rangle\\
=&K(h_{ij}^{p^{\ast}}-H^{p^{\ast}}\delta_{ij})+K(2h_{ij}^{p^{\ast}}-h_{ij}^{p^{\ast}})
+K(h_{ij}^{p^{\ast}}-\sum_{k}h_{kk}^{i^{\ast}}\delta_{pj})\\
&+\sum_{k}h_{ijk}^{p^{\ast}}\langle X,e_{k}\rangle
+h_{ij}^{p^{\ast}}-\sum_{q,k}h_{ik}^{p^{\ast}}h_{jk}^{q^{\ast}}H^{q^{\ast}}-\sum_{k}h_{ijk}^{p^{\ast}}\langle X,e_{k}\rangle \\
=&(3K+1)h_{ij}^{p^{\ast}}-K(H^{p^{\ast}}\delta_{ij}+H^{i^{\ast}}\delta_{pj})-\sum_{q,k}h_{ik}^{p^{\ast}}h_{jk}^{q^{\ast}}H^{q^{\ast}}.
\endaligned
\end{equation}
Hence, we get
\begin{equation}
\aligned
\frac{1}{2}\mathcal{L}S
=&\frac{1}{2}\sum_{k}\nabla_{k}\nabla_{k}\sum_{i,j,p}(h_{ij}^{p^{\ast}})^{2}
-\frac{1}{2}\langle X,\sum_{k}\nabla_{k}S e_{k}\rangle \\
=&\sum_{k}\nabla_{k}(\sum_{i,j,p} h_{ijk}^{p^{\ast}}h_{ij}^{p^{\ast}})
-\langle X,\sum_{i,j,k,p}h_{ijk}^{p^{\ast}}h_{ij}^{p^{\ast}}e_{k}\rangle\\
=&\sum_{i,j,p} h_{ij}^{p^{\ast}}
\mathcal{L} h_{ij}^{p^{\ast}}+\sum_{i,j,k,p} (h_{ijk}^{p^{\ast}})^{2} \\
=&\sum_{i,j,k,p} (h_{ijk}^{p^{\ast}})^{2}\\
&+\sum_{i,j,p} h_{ij}^{p^{\ast}}\bigl[(3K+1)h_{ij}^{p^{\ast}}
-K(H^{p^{\ast}}\delta_{ij}+H^{i^{\ast}}\delta_{pj})-\sum_{k,q} h_{ik}^{p^{\ast}}h_{jk}^{q^{\ast}} H^{q^{\ast}}\bigl]\\
=&\sum_{i,j,k,p}  (h_{ijk}^{p^{\ast}})^{2}+(3K+1)S-K(H^{2}+H^{2})
-\sum_{i,j,k,p,q} h_{ik}^{p^{\ast}}h_{ij}^{p^{\ast}}h_{jk}^{q^{\ast}} H^{q^{\ast}}  \\
=&\sum_{i,j,k,p}  (h_{ijk}^{p^{\ast}})^{2}+S(1-\frac{3}{2}S)+\frac{5}{2}H^{2}S-H^{4}
-\sum_{i,j,k,p,q}h_{ik}^{p^{\ast}}h_{ij}^{p^{\ast}}h_{jk}^{q^{\ast}} H^{q^{\ast}}.
\endaligned
\end{equation}
From \eqref{eq:16-5} and \eqref{eq:16-11}, we get
\begin{equation*}
\sum_{p,i} h_{ik}^{p^{\ast}}h_{ji}^{p^{\ast}}-\sum_{p} H^{p^{\ast}}h_{jk}^{p^{\ast}} =K(\delta_{kj}-2\delta_{jk})
\end{equation*}
and
\begin{equation*}
\sum_{p,i} h_{ik}^{p^{\ast}}h_{ji}^{p^{\ast}}=-K\delta_{jk}+\sum_{p}H^{p^{\ast}}h_{jk}^{p^{\ast}}.
\end{equation*}
Since
\begin{equation*}
\sum_{j,k}(K\delta_{jk}-\sum_{p}H^{p^{\ast}}h_{jk}^{p^{\ast}})\sum_{q}h_{jk}^{q^{\ast}} H^{q^{\ast}}=
KH^{2}-\sum_{j,k}\sum_{p}(H^{p^{\ast}}h_{jk}^{p^{\ast}})\sum_{q}(H^{q^{\ast}}h_{jk}^{q^{\ast}}),
\end{equation*}
we obtain from (2.24)
\begin{equation*}
\aligned
\frac{1}{2}\mathcal{L}S
=&\sum_{i,j,k,p} (h_{ijk}^{p^{\ast}})^{2}+S(1-\frac{3}{2}S)+\frac{5}{2}H^{2}S-H^{4}\\
&+\sum_ {j,k}(K\delta_{jk}-\sum_{p}H^{p^{\ast}}h_{jk}^{p^{\ast}})\sum_{q}h_{jk}^{q^{\ast}} H^{q^{\ast}}\\
=&\sum_{i,j,k,p} (h_{ijk}^{p^{\ast}})^{2}+S(1-\frac{3}{2}S)+2H^{2}S-\frac{1}{2}H^{4}
-\sum_{j,k,p,q} H^{p^{\ast}}h_{jk}^{p^{\ast}}H^{q^{\ast}}h_{jk}^{q^{\ast}}.
\endaligned
\end{equation*}
From \eqref{eq:16-13}, we have
\begin{equation*}
\aligned
\frac{1}{2}\mathcal{L}H^{2}=&\frac{1}{2}\mathcal{L}\sum_{p}(H^{p^{\ast}})^{2}
=\sum_{i,p}(\nabla_{i}H^{p^{\ast}})^{2}+\sum_{p}H^{p^{\ast}}\mathcal{L}H^{p^{\ast}}\\
=&|\nabla^{\perp}{\vec H}|^{2}+H^{2}-\sum_{i,j,p,q}H^{p^{\ast}}h_{ij}^{p^{\ast}}H^{q^{\ast}}h_{ij}^{q^{\ast}}.
\endaligned
\end{equation*}
Thus, we conclude the following lemma
\begin{lemma}
Let $X:M^2\rightarrow \mathbb{R}^4$ is a 2-dimensional Lagrangian self-shrinker in $\mathbb R^4$. We have
\begin{equation}\label{eq:18-3}
\aligned
\frac{1}{2}\mathcal{L}S
=&\sum_{i,j,k,p} (h_{ijk}^{p^{\ast}})^{2}+S(1-\frac{3}{2}S)+2H^{2}S-\frac{1}{2}H^{4}
-\sum_{j,k,p,q}H^{p^{\ast}}h_{jk}^{p^{\ast}}H^{q^{\ast}}h_{jk}^{q^{\ast}}.
\endaligned
\end{equation}
\begin{equation}
\aligned
\frac{1}{2}\mathcal{L}H^{2}=&|\nabla^{\perp}{\vec H}|^{2}
+H^{2}-\sum_{i,j,p,q}H^{p^{\ast}}h_{ij}^{p^{\ast}}\cdot H^{q^{\ast}}h_{ij}^{q^{\ast}}.
\endaligned
\end{equation}
\end{lemma}

\noindent
Next, we  will prove the following lemma,  by making use of a long calculation:
\begin{lemma}
Let $X:M^2\rightarrow \mathbb{R}^4$ is a 2-dimensional Lagrangian self-shrinker in $\mathbb R^4$. Then 
\begin{equation}\label{eq:18-2}
\aligned
&\frac{1}{2}\mathcal{L}\sum_{i, j,k,p}(h_{ijk}^{p^{\ast}})^{2}\\
=&\sum_{i,j,k,l,p}(h_{ijkl}^{p^{\ast}})^{2}+(10K+2)\sum_{i,j,k,p}(h_{ijk}^{p^{\ast}})^{2}
-5K|\nabla^{\perp} {\vec H}|^{2}+3\langle\nabla K,\nabla S\rangle\\
&-\frac{3K}{4}\langle\nabla S,\nabla|X|^{2}\rangle-\langle\nabla K,\nabla H^{2}\rangle
-3\sum_{j,l,p}K_{,l}h_{lj}^{p^{\ast}}H_{,j}^{p^{\ast}}\\
&-2\sum_{i,j,k,l,p,q}h_{ijk}^{p^{\ast}}h_{ijl}^{p^{\ast}}h_{kl}^{q^{\ast}}H^{q^{\ast}}
-\sum_{i,j,k,l,p,q}h_{il}^{p^{\ast}}h_{ijk}^{p^{\ast}}h_{jlk}^{q^{\ast}}H^{q^{\ast}}\\
&-\sum_{i,j,k,l,p,q}h_{il}^{p^{\ast}}h_{ijk}^{p^{\ast}}h_{jl}^{q^{\ast}}H_{,k}^{q^{\ast}}
\endaligned
\end{equation}
holds.
\end{lemma}

\begin{proof}
We have the following equation from the Ricci identities (2.10).
\begin{equation}
\aligned
&\mathcal{L}h_{ijk}^{p^{\ast}}=\sum_{l}h_{ijkll}^{p^{\ast}}-\langle X,\sum_{l}\nabla_{l}h_{ijk}^{p^{\ast}}e_{l}\rangle\\
=&\sum_{l}h_{ijlkl}^{p^{\ast}}+\sum_{l,m}(h_{mj}^{p^{\ast}}R_{mikl}
+h_{im}^{p^{\ast}}R_{mjkl}+h_{ij}^{m^{\ast}}R_{m^{\ast}p^{\ast}kl})_{,l}-\langle X,\sum_{l}\nabla_{l}h_{ijk}^{p^{\ast}}e_{l}\rangle \\
=&\sum_{l}h_{ijllk}^{p^{\ast}}+\sum_{l,m}h_{mjl}^{p^{\ast}}R_{mikl}
+\sum_{l,m}h_{iml}^{p^{\ast}}R_{mjkl}+\sum_{l,m}h_{ijm}^{p^{\ast}}R_{mlkl}\\
&+\sum_{l,m}h_{ijl}^{m^{\ast}}R_{m^{\ast}p^{\ast}kl}
+\sum_{l,m}h_{mjl}^{p^{\ast}}R_{mikl}\\
&+\sum_{l,m}h_{iml}^{p^{\ast}}R_{mjkl}+\sum_{l,m}h_{ijl}^{m^{\ast}}R_{m^{\ast}p^{\ast}kl}
+\sum_{l,m}h_{mj}^{p^{\ast}}R_{mikl,l}\\
&+\sum_{l,m}h_{im}^{p^{\ast}}R_{mjkl,l}
+\sum_{l,m}h_{ij}^{m^{\ast}}R_{m^{\ast}p^{\ast}kl,l}
-\langle X,\sum_{l}\nabla_{l}h_{ijk}^{p^{\ast}}e_{l}\rangle.
\endaligned
\end{equation}
From \eqref{eq:16-14}, we have
\begin{equation}\label{eq:18-1}
\aligned
&\sum_{l}h_{ijllk}^{p^{\ast}}=\bigl[(3K+1)h_{ij}^{p^{\ast}}-K(H^{p^{\ast}}\delta_{ij}+H^{i^{\ast}}\delta_{pj})\\
&-\sum_{l,q}h_{il}^{p^{\ast}}h_{jl}^{q^{\ast}}H^{q^{\ast}}
+\langle X,\sum_{l}h_{ijl}^{p^{\ast}}e_{l}\rangle\bigl]_{,k} \\
=&3K_{,k}h_{ij}^{p^{\ast}}+(3K+1)h_{ijk}^{p^{\ast}}-K_{,k}(H^{p^{\ast}}\delta_{ij}+H^{i^{\ast}}\delta_{pj})
-K(H_{,k}^{p^{\ast}}\delta_{ij}+H_{,k}^{i^{\ast}}\delta_{pj})\\
&-\sum_{l,q}h_{ilk}^{p^{\ast}}h_{jl}^{q^{\ast}}H^{q^{\ast}}
-\sum_{l,q}h_{il}^{p^{\ast}}h_{jlk}^{q^{\ast}}H^{q^{\ast}}-\sum_{l,q}h_{il}^{p^{\ast}}h_{jl}^{q^{\ast}}H_{,k}^{q^{\ast}}\\
&+h_{ijk}^{p^{\ast}}+\langle X,\sum_{l,q}h_{ijl}^{p^{\ast}}h_{lk}^{q^{\ast}}e_{q^{\ast}}\rangle
+\langle X,\sum_{l}h_{ijlk}^{p^{\ast}}e_{l}\rangle.
\endaligned
\end{equation}
From \eqref{eq:16-11}, we obtain
\begin{equation*}
\aligned
&(a)^{p^{\ast}}_{ijk}:=\sum_{l,m}h_{mjl}^{p^{\ast}}R_{mikl}+\sum_{l,m}h_{iml}^{p^{\ast}}R_{mjkl}
+\sum_{l,m}h_{ijm}^{p^{\ast}}R_{mlkl}+\sum_{l,m}h_{ijl}^{m^{\ast}}R_{m^{\ast}p^{\ast}kl}\\
&+\sum_{l,m}h_{mjl}^{p^{\ast}}R_{mikl}+\sum_{l,m}h_{iml}^{p^{\ast}}R_{mjkl}+\sum_{l,m}h_{ijl}^{m^{\ast}}R_{m^{\ast}p^{\ast}kl}\\
=&2K\sum_{l,m}\bigl[h_{mjl}^{p^{\ast}}(\delta_{mk}\delta_{il}-\delta_{ml}\delta_{ik})+h_{iml}^{p^{\ast}}(\delta_{mk}\delta_{jl}-\delta_{ml}\delta_{jk})\\
&+h_{ijl}^{m^{\ast}}(\delta_{mk}\delta_{pl}-\delta_{ml}\delta_{pk})\bigl]+\sum_{l,m}Kh_{ijm}^{p^{\ast}}(\delta_{mk}\delta_{ll}
-\delta_{ml}\delta_{lk})\\
=&2K\bigl[h_{ijk}^{p^{\ast}}-H_{,j}^{p^{\ast}}\delta_{ik}
+h_{ijk}^{p^{\ast}}-H_{,i}^{p^{\ast}}\delta_{jk}+h_{ijk}^{p^{\ast}}-H_{,j}^{i^{\ast}}\delta_{pk}
+h_{ijk}^{p^{\ast}}-\frac{1}{2}h_{ijk}^{p^{\ast}}\bigl]\\
=&2K\bigl[\frac{7}{2}h_{ijk}^{p^{\ast}}-H_{,j}^{p^{\ast}}\delta_{ik}-H_{,i}^{p^{\ast}}\delta_{jk}-H_{,j}^{i^{\ast}}\delta_{pk}\bigl],
\endaligned
\end{equation*}

\begin{equation*}
\aligned
&(b)^{p^{\ast}}_{ijk}:=\sum_{l,m}h_{mj}^{p^{\ast}}R_{mikl,l}+\sum_{l,m}h_{im}^{p^{\ast}}R_{mjkl,l}+\sum_{l,m}h_{ij}^{m^{\ast}}R_{m^{\ast}p^{\ast}kl,l}\\
=&\sum_{l,m}K_{,l}\bigl[h_{mj}^{p^{\ast}}(\delta_{mk}\delta_{il}-\delta_{ml}\delta_{ik})
+h_{im}^{p^{\ast}}(\delta_{mk}\delta_{jl}-\delta_{ml}\delta_{jk})\\
&+
h_{ij}^{m^{\ast}}(\delta_{mk}\delta_{pl}-\delta_{ml}\delta_{pk})\bigl]\\
=&K_{,i}h_{jk}^{p^{\ast}}-\sum_{l}K_{,l}h_{lj}^{p^{\ast}}\delta_{ik}+K_{,j}h_{ik}^{p^{\ast}}-\sum_{l}K_{,l}h_{il}^{p^{\ast}}\delta_{jk}+
K_{,p}h_{ij}^{k^{\ast}}-\sum_{l}K_{,l}h_{ij}^{l^{\ast}}\delta_{pk}
\endaligned
\end{equation*}
and
\begin{equation*}
\aligned
&(c)^{p^{\ast}}_{ijk}:=\sum_{l}\langle X,e_{l}\rangle h_{ijlk}^{p^{\ast}}-\sum_{l}\langle X,e_{l}\rangle h_{ijkl}^{p^{\ast}}\\
=&\sum_{l,m}\langle X,e_{l}\rangle\bigl[h_{mj}^{p^{\ast}}R_{milk}+h_{im}^{p^{\ast}}R_{mjlk}
+h_{ij}^{m^{\ast}}R_{m^{\ast}p^{\ast}lk}\bigl] \\
=&K\sum_{l,m}\langle X,e_{l}\rangle\bigl[h_{mj}^{p^{\ast}}(\delta_{ml}\delta_{ik}
-\delta_{mk}\delta_{il})+h_{im}^{p^{\ast}}(\delta_{ml}\delta_{jk}-\delta_{mk}\delta_{jl})\\
&+h_{ij}^{m^{\ast}}(\delta_{ml}\delta_{pk}-\delta_{mk}\delta_{pl})\bigl]\\
\endaligned
\end{equation*}
\begin{equation*}
\aligned
=&K\sum_{l}\langle X,e_{l}\rangle\bigl[h_{lj}^{p^{\ast}}\delta_{ik}-h_{kj}^{p^{\ast}}\delta_{il}
+h_{il}^{p^{\ast}}\delta_{jk}-h_{ik}^{p^{\ast}}\delta_{jl}
+h_{ij}^{l^{\ast}}\delta_{pk}-h_{ij}^{k^{\ast}}\delta_{pl}\bigl] \\
=&K\bigl[\sum_{l}\langle X,e_{l}\rangle h_{lj}^{p^{\ast}}\delta_{ik}-\langle X,e_{i}\rangle h_{kj}^{p^{\ast}}+\sum_{l}\langle X,e_{l}\rangle h_{il}^{p^{\ast}}\delta_{jk}-\langle X,e_{j}\rangle h_{ik}^{p^{\ast}}\\
&+\sum_{l}\langle X,e_{l}\rangle h_{ij}^{l^{\ast}}\delta_{pk}-\langle X,e_{p}\rangle h_{ij}^{k^{\ast}}\bigl].
\endaligned
\end{equation*}
We conclude
\begin{equation}\aligned
&\sum_{i,j,k,p}h_{ijk}^{p^{\ast}}\cdot((a)^{p^{\ast}}_{ijk}+(b)^{p^{\ast}}_{ijk})\\
=&2K\bigl[\frac{7}{2}\sum_{i,j,k,p}(h_{ijk}^{p^{\ast}})^{2}-\sum_{j,p}H_{,j}^{p^{\ast}}H_{,j}^{p^{\ast}}
-\sum_{i,p}H_{,i}^{p^{\ast}}H_{,i}^{p^{\ast}}-\sum_{j,i}H_{,j}^{i^{\ast}}H_{,j}^{i^{\ast}}\bigl]\\
&+\sum_{i,j,k,p}K_{,i}h_{jk}^{p^{\ast}}h_{ijk}^{p^{\ast}}-\sum_{l,j,p}K_{,l}h_{lj}^{p^{\ast}}H_{,j}^{p^{\ast}}
+\sum_{i,j,k,p}K_{,j}h_{ik}^{p^{\ast}}h_{ijk}^{p^{\ast}}\\
&-\sum_{l,i,p}K_{,l}h_{il}^{p^{\ast}}H_{,i}^{p^{\ast}}
+\sum_{i,j,k,p}K_{,p}h_{ij}^{k^{\ast}}h_{ijk}^{p^{\ast}}-\sum_{l,i,j}K_{,l}h_{ij}^{l^{\ast}}H_{,j}^{i^{\ast}}\\
=&2K\bigl[\frac{7}{2}\sum_{i,j,k,p}(h_{ijk}^{p^{\ast}})^{2}-3|\nabla^{\perp} {\vec H}|^{2}]
+ \frac{3}{2}\langle\nabla K,\nabla S\rangle-3\sum_{l,j,p}K_{,l}h_{lj}^{p^{\ast}}H_{,j}^{p^{\ast}}
\endaligned\end{equation}
and
\begin{equation}\aligned
&\sum_{i,j,k,p}h_{ijk}^{p^{\ast}}\cdot(c)^{p^{\ast}}_{ijk}\\
=&K\bigl[\sum_{l,j,p}\langle X,e_{l}\rangle h_{lj}^{p^{\ast}}H_{,j}^{p^{\ast}}
+\sum_{l,i,p}\langle X,e_{l}\rangle h_{li}^{p^{\ast}}H_{,i}^{p^{\ast}}
+\sum_{l,i,j}\langle X,e_{l}\rangle h_{ij}^{l^{\ast}}H_{,j}^{i^{\ast}}\\
&-\frac{1}{2}\sum_{i}\langle X,e_{i}\rangle\nabla_{i}S
-\frac{1}{2}\sum_{i}\langle X,e_{i}\rangle\nabla_{i}S
-\frac{1}{2}\sum_{i}\langle X,e_{i}\rangle\nabla_{i}S\bigl]\\
=&K\bigl(3\sum_{l,j,p}\langle X,e_{l}\rangle h_{lj}^{p^{\ast}}H_{,j}^{p^{\ast}}
-\frac{3}{4}\langle\nabla S,\nabla\langle X,X\rangle\rangle\bigl)\\
=&3K(|\nabla^{\perp} {\vec H}|^{2}-\frac{1}{4}\langle\nabla S,\nabla|X|^{2}\rangle).
\endaligned\end{equation}
From \eqref{eq:18-1}, we have
\begin{equation}
\aligned
&\sum_{i,j,k,l,p}h_{ijk}^{p^{\ast}}(h_{ijll,k}^{p^{\ast}}-\langle X,h_{ijlk}^{p^{\ast}}e_{l}\rangle)\\
=&3\sum_{i,j,k,p}K_{,k}h_{ij}^{p^{\ast}}h_{ijk}^{p^{\ast}}
+(3K+1)\sum_{i,j,k,p}(h_{ijk}^{p^{\ast}})^{2}-\sum_{k,p}K_{,k}(H^{p^{\ast}}H_{,k}^{p^{\ast}}+H^{p^{\ast}}H_{,k}^{p^{\ast}})\\
&-\sum_{k,p}K(H_{,k}^{p^{\ast}}H_{,k}^{p^{\ast}}+H_{,k}^{p^{\ast}}H_{,k}^{p^{\ast}})
-\sum_{i,j,k,l,p,q}h_{ijk}^{p^{\ast}}h_{ilk}^{p^{\ast}}h_{jl}^{q^{\ast}}H^{q^{\ast}}\\
&-\sum_{i,j,k,l,p,q}h_{il}^{p^{\ast}}h_{ijk}^{p^{\ast}}h_{jlk}^{q^{\ast}}H^{q^{\ast}}
-\sum_{i,j,k,l,p,q}h_{il}^{p^{\ast}}h_{ijk}^{p^{\ast}}h_{jl}^{q^{\ast}}H_{,k}^{q^{\ast}}\\
\endaligned
\end{equation}
\begin{equation*}
\aligned
&+\sum_{i,j,k,p}(h_{ijk}^{p^{\ast}})^{2}
-\sum_{i,j,k,l,p,q}h_{ijk}^{p^{\ast}}h_{ijl}^{p^{\ast}}h_{lk}^{q^{\ast}}H^{q^{\ast}}\\
=&\frac{3}{2}\langle\nabla K,\nabla S\rangle+(3K+2)\sum_{i,j,k,p}(h_{ijk}^{p^{\ast}})^{2}
-\langle\nabla K,\nabla H^{2}\rangle-
2K|\nabla^{\perp} {\vec H}|^{2}\\
&-2\sum_{i,j,k,l,p,q}h_{ijk}^{p^{\ast}}h_{ijl}^{p^{\ast}}h_{kl}^{q^{\ast}}H^{q^{\ast}}
-\sum_{i,j,k,l,p,q}h_{il}^{p^{\ast}}h_{ijk}^{p^{\ast}}h_{jlk}^{q^{\ast}}H^{q^{\ast}}
-\sum_{i,j,k,l,p,q}h_{il}^{p^{\ast}}h_{ijk}^{p^{\ast}}h_{jl}^{q^{\ast}}H_{,k}^{q^{\ast}}.
\endaligned
\end{equation*}
From the above equations, we get
\begin{equation}\label{eq:18-2}
\aligned
&\frac{1}{2}\mathcal{L}\sum_{i,j,k,p}(h_{ijk}^{p^{\ast}})^{2}
=\sum_{i,j,k,p}h_{ijk}^{p^{\ast}}\mathcal{L}h_{ijk}^{p^{\ast}}+\sum_{i,j,k,l,p}(h_{ijkl}^{p^{\ast}})^{2}\\
=&\sum_{i,j,k,l,p}(h_{ijkl}^{p^{\ast}})^{2}+2K[\frac{7}{2}\sum_{i,j,k,p}(h_{ijk}^{p^{\ast}})^{2}
-3|\nabla^{\perp} {\vec H}|^{2}]+\frac{3}{2}\langle\nabla K,\nabla S\rangle\\
&-3\sum_{l,j,p}K_{,l}h_{lj}^{p^{\ast}}H_{,j}^{p^{\ast}}
+3K\bigl[|\nabla^{\perp} {\vec H}|^{2}-\frac{1}{4}\langle\nabla S,\nabla|X|^{2}\rangle\bigl]\\
&+\frac{3}{2}\langle\nabla K,\nabla S\rangle
+(3K+2)\sum_{i,j,k,p}(h_{ijk}^{p^{\ast}})^{2}-\langle\nabla K,\nabla H^{2}\rangle-2K|\nabla^{\perp} {\vec H}|^{2}\\
&-2\sum_{i,j,k,l,p,q}h_{ijk}^{p^{\ast}}h_{ijl}^{p^{\ast}}h_{kl}^{q^{\ast}}H^{q^{\ast}}
-\sum_{i,j,k,l,p,q}h_{il}^{p^{\ast}}h_{ijk}^{p^{\ast}}h_{jlk}^{q^{\ast}}H^{q^{\ast}}
-\sum_{i,j,k,l,p,q}h_{il}^{p^{\ast}}h_{ijk}^{p^{\ast}}h_{jl}^{q^{\ast}}H_{,k}^{q^{\ast}}\\
=&\sum_{i,j,k,l,p}(h_{ijkl}^{p^{\ast}})^{2}+(10K+2)\sum_{i,j,k,p}(h_{ijk}^{p^{\ast}})^{2}
-5K|\nabla^{\perp} {\vec H}|^{2}\\
&+3\langle\nabla K,\nabla S\rangle-\frac{3K}{4}\langle\nabla S,\nabla|X|^{2}\rangle
-\langle\nabla K,\nabla H^{2}\rangle\\
&-3\sum_{l,j,p}K_{,l}h_{lj}^{p^{\ast}}H_{,j}^{p^{\ast}}
-2\sum_{i,j,k,l,p,q}h_{ijk}^{p^{\ast}}h_{ijl}^{p^{\ast}}h_{kl}^{q^{\ast}}H^{q^{\ast}}\\
&-\sum_{i,j,k,l,p,q}h_{il}^{p^{\ast}}h_{ijk}^{p^{\ast}}h_{jlk}^{q^{\ast}}H^{q^{\ast}}
-\sum_{i,j,k,l,p,q}h_{il}^{p^{\ast}}h_{ijk}^{p^{\ast}}h_{jl}^{q^{\ast}}H_{,k}^{q^{\ast}}.
\endaligned
\end{equation}
It completes the proof of the lemma.
\end{proof}
\begin{lemma}
Let $X:M^2\rightarrow \mathbb{R}^4$ be  a 2-dimensional Lagrangian self-shrinker in $\mathbb R^4$. If $S$ is constant,   we have
\begin{equation}
\aligned
&\frac{1}{2}\mathcal{L}\sum_{i,j,k,p}(h_{ijk}^{p^{\ast}})^{2}\\
=&\sum_{i,j,k,l,p}(h_{ijkl}^{p^{\ast}})^{2}+(10K+2)\sum_{i,j,k,p}(h_{ijk}^{p^{\ast}})^{2}-5K|\nabla^{\perp} {\vec H}|^{2}
\\
&-\langle\nabla K,\nabla H^{2}\rangle
-3\sum_{l,j,p}K_{,l}h_{lj}^{p^{\ast}}H_{,j}^{p^{\ast}}
-2\sum_{i,j,k,l,p,q}h_{ijk}^{p^{\ast}}h_{ijl}^{p^{\ast}}h_{kl}^{q^{\ast}}H^{q^{\ast}}\\
&-\sum_{i,j,k,l,p,q}h_{il}^{p^{\ast}}h_{ijk}^{p^{\ast}}h_{jlk}^{q^{\ast}}H^{q^{\ast}}
-\sum_{i,j,k,l,p,q}h_{il}^{p^{\ast}}h_{ijk}^{p^{\ast}}h_{jl}^{q^{\ast}}H_{,k}^{q^{\ast}}
\endaligned
\end{equation}
and
\begin{equation}
\aligned
&\frac{1}{2}\mathcal{L}\sum_{i,j,k,p}(h_{ijk}^{p^{\ast}})^{2}\\
=&(H^{2}-2S)(|\nabla^{\perp} {\vec H}|^{2}+H^{2})+\dfrac12|\nabla H^2|^{2}\\
&+(3K+2-H^{2}+2S)\sum_{i,j,p,q}H^{p^{\ast}}h_{ij}^{p^{\ast}} H^{q^{\ast}}h_{ij}^{q^{\ast}}\\
&-K(H^{4}+\sum_{j,k,p}H^{k^{\ast}}H^{j^{\ast}}H^{p^{\ast}}h_{jk}^{p^{\ast}})
-\sum_{i,j,k,l,p,q,r}H^{r^{\ast}}H^{q^{\ast}}h_{jk}^{r^{\ast}}h_{jk}^{p^{\ast}}h_{il}^{p^{\ast}}h_{il}^{q^{\ast}}\\
&+2\sum_{i,j,k,p,q}H_{,i}^{p^{\ast}}H^{q^{\ast}}h_{jk}^{q^{\ast}}h_{ijk}^{p^{\ast}}
-\sum_{i,j,k,p,q,r}H^{p^{\ast}}H^{q^{\ast}}H^{r^{\ast}}h_{ik}^{p^{\ast}}h_{ji}^{q^{\ast}}h_{jk}^{r^{\ast}}\\
&+\sum_{i,j,k}\bigl(\sum_{q}(H_{,i}^{q^{\ast}}h_{jk}^{q^{\ast}}+ H^{q^{\ast}}h_{ijk}^{q^{\ast}})\bigl)\cdot\bigl(\sum_{p}(H_{,i}^{p^{\ast}}h_{jk}^{p^{\ast}}+H^{p^{\ast}}h_{ijk}^{p^{\ast}})\bigl).
\endaligned
\end{equation}

\end{lemma}
\begin{proof}
Since $S$ is constant,   we have the following equation from \eqref{eq:18-2}
\begin{equation*}
\aligned
&\frac{1}{2}\mathcal{L}\sum_{i,j,k,p}(h_{ijk}^{p^{\ast}})^{2}\\
=&\sum_{i,j,k,l,p}(h_{ijkl}^{p^{\ast}})^{2}+(10K+2)\sum_{i,j,k,p}(h_{ijk}^{p^{\ast}})^{2}
-5K|\nabla^{\perp} {\vec H}|^{2}-\langle\nabla K,\nabla H^{2}\rangle\\
&-3\sum_{j,l,p}K_{,l}h_{lj}^{p^{\ast}}H_{,j}^{p^{\ast}}-2\sum_{i,j,k,l,p,q}h_{ijk}^{p^{\ast}}h_{ijl}^{p^{\ast}}h_{kl}^{q^{\ast}}H^{q^{\ast}}\\
&-\sum_{i,j,k,l,p,q}h_{il}^{p^{\ast}}h_{ijk}^{p^{\ast}}h_{jlk}^{q^{\ast}}H^{q^{\ast}}
-\sum_{i,j,k,l,p,q}h_{il}^{p^{\ast}}h_{ijk}^{p^{\ast}}h_{jl}^{q^{\ast}}H_{,k}^{q^{\ast}}.
\endaligned
\end{equation*}
Now, we prove the formula (2.35).
From \eqref{eq:18-3} in Lemma 2.1, we obtain
\begin{equation}
0=\frac{1}{2}\mathcal{L}S
=\sum_{i,j,k,p} (h_{ijk}^{p^{\ast}})^{2}+S(1-\frac{3}{2}S)+2H^{2}S-\frac{1}{2}H^{4}
-\sum_{j,k,p,q} H^{p^{\ast}}h_{jk}^{p^{\ast}}H^{q^{\ast}}h_{jk}^{q^{\ast}}.
\end{equation}
Hence, we have
\begin{equation}
\aligned
&\frac{1}{2}\mathcal{L}\sum_{i,j,k,p}(h_{ijk}^{p^{\ast}})^{2}
=-S\mathcal{L}H^{2}+\frac{1}{4}\mathcal{L}H^{4}
+\frac{1}{2}\mathcal{L}\sum_{j,k,p,q}H^{p^{\ast}}h_{jk}^{p^{\ast}}H^{q^{\ast}}h_{jk}^{q^{\ast}}\\
=&-2S\Bigg(|\nabla^{\perp} {\vec H}|^{2}+H^{2}-\sum_{i,j,p,q}H^{p^{\ast}}h_{ij}^{p^{\ast}} H^{q^{\ast}}h_{ij}^{q^{\ast}}\Bigg)
+\frac{1}{2}H^{2}\mathcal{L}H^{2}\\
&+\frac{1}{2}|\nabla H^{2}|^{2}+\sum_{j,k}(\sum_{q} H^{q^{\ast}}h_{jk}^{q^{\ast}})\mathcal{L}(\sum_{p} H^{p^{\ast}}h_{jk}^{p^{\ast}})\\
&+\sum_{i,j,k}\nabla_{i}(\sum_{q} H^{q^{\ast}}h_{jk}^{q^{\ast}})\cdot \nabla_{i}(\sum_{p} H^{p^{\ast}}h_{jk}^{p^{\ast}})\\
\endaligned
\end{equation}
\begin{equation*}
\aligned
=&(H^{2}-2S)\Bigg(|\nabla^{\perp} {\vec H}|^{2}+H^{2}-\sum_{i,j,p,q} H^{p^{\ast}}h_{ij}^{p^{\ast}} H^{q^{\ast}}h_{ij}^{q^{\ast}}\Bigg)\\
&+\dfrac12|\nabla H^2|^{2}+\sum_{j,k}(\sum_{q} H^{q^{\ast}}h_{jk}^{q^{\ast}})\mathcal{L}(\sum_{p} H^{p^{\ast}}h_{jk}^{p^{\ast}})\\
&+\sum_{i,j,k,p,q}(H_{,i}^{q^{\ast}}h_{jk}^{q^{\ast}}
+ H^{q^{\ast}}h_{ijk}^{q^{\ast}})\cdot(H_{,i}^{p^{\ast}}h_{jk}^{p^{\ast}}+H^{p^{\ast}}h_{ijk}^{p^{\ast}}).
\endaligned
\end{equation*}
Since
\begin{equation}
\aligned
&\mathcal{L}\sum_{p}(H^{p^{\ast}}h_{jk}^{p^{\ast}})\\
=&\sum_{i,p}\nabla_{i}\nabla_{i}(H^{p^{\ast}}h_{jk}^{p^{\ast}})
-\langle X,\sum_{i,p}\nabla_{i}(H^{p^{\ast}}h_{jk}^{p^{\ast}})e_{i}\rangle\\
=&\sum_{i,p}\nabla_{i}(H_{,i}^{p^{\ast}}h_{jk}^{p^{\ast}}
+H^{p^{\ast}}h_{ijk}^{p^{\ast}})
-\langle X,\sum_{i,p}(\nabla_{i}H^{p^{\ast}}h_{jk}^{p^{\ast}}+H^{p^{\ast}}\nabla_{i}h_{jk}^{p^{\ast}})e_{i}\rangle\\
=&\sum_{p}h_{jk}^{p^{\ast}}\mathcal{L}H^{p^{\ast}}+\sum_{p}H^{p^{\ast}}\mathcal{L}h_{jk}^{p^{\ast}}
+2\sum_{i,p}H_{,i}^{p^{\ast}}h_{ijk}^{p^{\ast}}
\endaligned
\end{equation}
and from (2.20) and (2.23),
\begin{equation}
\sum_{p}h_{jk}^{p^{\ast}}\mathcal{L}H^{p^{\ast}}=\sum_{p}h_{jk}^{p^{\ast}}H^{p^{\ast}}-\sum_{i,l,p,q}h_{jk}^{p^{\ast}}h_{il}^{p^{\ast}}h_{il}^{q^{\ast}}H^{q^{\ast}},
\end{equation}
\begin{equation}
\aligned
&\sum_{p}H^{p^{\ast}}\mathcal{L} h_{jk}^{p^{\ast}}\\
=&(3K+1)\sum_{p}H^{p^{\ast}}h_{jk}^{p^{\ast}}-K\sum_{p}((H^{p^{\ast}})^{2}\delta_{kj}+H^{p^{\ast}}H^{k^{\ast}}\delta_{pj})\\
&-\sum_{i,p,q}h_{ik}^{p^{\ast}}H^{p^{\ast}}h_{ji}^{q^{\ast}}H^{q^{\ast}}\\
=&(3K+1)\sum_{p}H^{p^{\ast}}h_{jk}^{p^{\ast}}-K(H^{2}\delta_{kj}+H^{k^{\ast}}H^{j^{\ast}})
-\sum_{i,p,q}h_{ik}^{p^{\ast}}H^{p^{\ast}}h_{ji}^{q^{\ast}}H^{q^{\ast}},
\endaligned
\end{equation}
we get
\begin{equation}
\aligned
&\mathcal{L}\sum_{p}(H^{p^{\ast}}h_{jk}^{p^{\ast}})\\
=&\sum_{p}h_{jk}^{p^{\ast}}H^{p^{\ast}}-\sum_{i,l,p,q}h_{jk}^{p^{\ast}}h_{il}^{p^{\ast}}h_{il}^{q^{\ast}}H^{q^{\ast}}
+2\sum_{i,p}H_{,i}^{p^{\ast}}h_{ijk}^{p^{\ast}}\\
&+(3K+1)\sum_{p}H^{p^{\ast}}h_{jk}^{p^{\ast}}-K(H^{2}\delta_{kj}+H^{k^{\ast}}H^{j^{\ast}})
-\sum_{i,p,q}h_{ik}^{p^{\ast}}H^{p^{\ast}}h_{ji}^{q^{\ast}}H^{q^{\ast}}
\endaligned
\end{equation}
and
\begin{equation}
\aligned
&\sum_{j,k}(\sum_{q} H^{q^{\ast}}h_{jk}^{q^{\ast}})\mathcal{L}(\sum_{p} H^{p^{\ast}}h_{jk}^{p^{\ast}})\\
=&-\sum_{i,j,k,l,p,q,r}H^{r^{\ast}}h_{jk}^{r^{\ast}}h_{jk}^{p^{\ast}}h_{il}^{p^{\ast}}h_{il}^{q^{\ast}}H^{q^{\ast}}
+2\sum_{i,j,k,p,q}H_{,i}^{p^{\ast}}h_{ijk}^{p^{\ast}}H^{q^{\ast}}h_{jk}^{q^{\ast}}\\
&+(3K+2)\sum_{j,k}\bigl(\sum_q H^{q^{\ast}}h_{jk}^{q^{\ast}}\bigl)\bigl(\sum_p H^{p^{\ast}}h_{jk}^{p^{\ast}}\bigl)\\
&-K(H^{4}+\sum_{j,k,p}H^{k^{\ast}}H^{j^{\ast}}H^{p^{\ast}}h_{jk}^{p^{\ast}})
-\sum_{i,j,k,p,q,r}H^{p^{\ast}}H^{q^{\ast}}H^{r^{\ast}}h_{ik}^{p^{\ast}}h_{ji}^{q^{\ast}}h_{jk}^{r^{\ast}}.
\endaligned
\end{equation}
From the above equations, we conclude
\begin{equation*}
\aligned
&\frac{1}{2}\mathcal{L}\sum_{i,j,k,p}(h_{ijk}^{p^{\ast}})^{2}\\
=&(H^{2}-2S)(|\nabla^{\perp} {\vec H}|^{2}+H^{2})+\dfrac12|\nabla H^2|^{2}\\
&+(3K+2-H^{2}+2S)\sum_{i,j,p,q}H^{p^{\ast}}h_{ij}^{p^{\ast}} H^{q^{\ast}}h_{ij}^{q^{\ast}}\\
&-K(H^{4}+\sum_{j,k,p}H^{k^{\ast}}H^{j^{\ast}}H^{p^{\ast}}h_{jk}^{p^{\ast}})
-\sum_{i,j,k,l,p,q,r}H^{r^{\ast}}H^{q^{\ast}}h_{jk}^{r^{\ast}}h_{jk}^{p^{\ast}}h_{il}^{p^{\ast}}h_{il}^{q^{\ast}}\\
&+2\sum_{i,j,k,p,q}H_{,i}^{p^{\ast}}H^{q^{\ast}}h_{jk}^{q^{\ast}}h_{ijk}^{p^{\ast}}
-\sum_{i,j,k,p,q,r}H^{p^{\ast}}H^{q^{\ast}}H^{r^{\ast}}h_{ik}^{p^{\ast}}h_{ji}^{q^{\ast}}h_{jk}^{r^{\ast}}\\
&+\sum_{i,j,k}\bigl(\sum_{q}(H_{,i}^{q^{\ast}}h_{jk}^{q^{\ast}}+ H^{q^{\ast}}h_{ijk}^{q^{\ast}})\bigl)\cdot\bigl(\sum_{p}(H_{,i}^{p^{\ast}}h_{jk}^{p^{\ast}}+H^{p^{\ast}}h_{ijk}^{p^{\ast}})\bigl).
\endaligned
\end{equation*}
\end{proof}
\noindent
\begin{lemma}
Let $X:M^2\rightarrow \mathbb{R}^4$ be  a 2-dimensional Lagrangian self-shrinker in $\mathbb R^4$. Then we have
\begin{equation}
\begin{aligned}
&\sum_{i,j,k,l,p}(h_{ijkl}^{p^{\ast}})^{2}\\
=&4(h_{1122}^{1^{\ast}})^{2}+6(h_{2211}^{1^{\ast}})^{2}+6(h_{1122}^{2^{\ast}})^{2}+4(h_{2211}^{2^{\ast}})^{2}
+4(h_{2222}^{1^{\ast}})^{2}+4(h_{1111}^{2^{\ast}})^{2}\\
&+(h_{1111}^{1^{\ast}})^{2}+(h_{2222}^{2^{\ast}})^{2}+(h_{1112}^{1^{\ast}})^{2}+(h_{2221}^{2^{\ast}})^{2}\\
\geq&2(h_{1122}^{1^{\ast}}-h_{2211}^{1^{\ast}})^{2}+2(h_{1122}^{2^{\ast}}-h_{2211}^{2^{\ast}})^{2}
+\frac{1}{2}(h_{2222}^{1^{\ast}}-h_{2221}^{2^{\ast}})^{2}. \\
\end{aligned}
\end{equation}

\end{lemma}
\begin{proof}
Since
\begin{equation}
\sum_{i,j,k,l,p}(h_{ijkl}^{p^{\ast}})^{2}=\sum_{i,j,k,l}(h_{ijkl}^{1^{\ast}})^{2}+\sum_{i,j,k,l}(h_{ijkl}^{2^{\ast}})^{2},
\end{equation}

\begin{equation}
\aligned
\sum_{i,j,k,l}(h_{ijkl}^{1^{\ast}})^{2}
=&(h_{1111}^{1^{\ast}})^{2}+(h_{1112}^{1^{\ast}})^{2}+(h_{2221}^{1^{\ast}})^{2}+(h_{2222}^{1^{\ast}})^{2}\\
&+3(h_{1121}^{1^{\ast}})^{2}+3(h_{1122}^{1^{\ast}})^{2}
+3(h_{2211}^{1^{\ast}})^{2}+3(h_{2212}^{1^{\ast}})^{2} \\
=&(h_{1111}^{1^{\ast}})^{2}+3\bigl[(h_{1122}^{1^{\ast}})^{2}+(h_{2211}^{1^{\ast}})^{2}\bigl]+3(h_{1122}^{2^{\ast}})^{2}\\
&+(h_{1112}^{1^{\ast}})^{2}+(h_{2211}^{2^{\ast}})^{2}+3(h_{1111}^{2^{\ast}})^{2}+(h_{2222}^{1^{\ast}})^{2},
\endaligned
\end{equation}

\begin{equation}
\aligned
\sum_{i,j,k,l}(h_{ijkl}^{2^{\ast}})^{2}
=&(h_{1111}^{2^{\ast}})^{2}+(h_{1112}^{2^{\ast}})^{2}+(h_{2222}^{2^{\ast}})^{2}+(h_{2221}^{2^{\ast}})^{2}\\
&+3(h_{1121}^{2^{\ast}})^{2}+3(h_{1122}^{2^{\ast}})^{2}
+3(h_{2211}^{2^{\ast}})^{2}+3(h_{2212}^{2^{\ast}})^{2} \\
=&(h_{1111}^{2^{\ast}})^{2}+(h_{2222}^{2^{\ast}})^{2}+(h_{2221}^{2^{\ast}})^{2}+(h_{1122}^{1^{\ast}})^{2}\\
&+3(h_{2211}^{1^{\ast}})^{2}+3(h_{1122}^{2^{\ast}})^{2}
+3(h_{2211}^{2^{\ast}})^{2}+3(h_{2222}^{1^{\ast}})^{2},
\endaligned
\end{equation}
we get
\begin{equation}
\aligned
&\sum_{i,j,k,l,p}(h_{ijkl}^{p^{\ast}})^{2}\\
=&4(h_{1122}^{1^{\ast}})^{2}+6(h_{2211}^{1^{\ast}})^{2}+6(h_{1122}^{2^{\ast}})^{2}+(h_{1111}^{1^{\ast}})^{2}
+(h_{1112}^{1^{\ast}})^{2}+3(h_{1111}^{2^{\ast}})^{2}\\
&+(h_{2222}^{1^{\ast}})^{2}+(h_{1111}^{2^{\ast}})^{2}+(h_{2222}^{2^{\ast}})^{2}+(h_{2221}^{2^{\ast}})^{2}+
3(h_{2222}^{1^{\ast}})^{2}+4(h_{2211}^{2^{\ast}})^{2}\\
=&4(h_{1122}^{1^{\ast}})^{2}+6(h_{2211}^{1^{\ast}})^{2}+6(h_{1122}^{2^{\ast}})^{2}+4(h_{2211}^{2^{\ast}})^{2}
+4(h_{2222}^{1^{\ast}})^{2}+4(h_{1111}^{2^{\ast}})^{2}\\
&+(h_{1111}^{1^{\ast}})^{2}+(h_{2222}^{2^{\ast}})^{2}+(h_{1112}^{1^{\ast}})^{2}+(h_{2221}^{2^{\ast}})^{2}\\
\geq&2(h_{1122}^{1^{\ast}}-h_{2211}^{1^{\ast}})^{2}+2(h_{1122}^{2^{\ast}}-h_{2211}^{2^{\ast}})^{2}
+\frac{1}{2}(h_{2222}^{1^{\ast}}-h_{2221}^{2^{\ast}})^{2} \\
&+2(h_{1122}^{1^{\ast}}+h_{2211}^{1^{\ast}})^{2}+2(h_{1122}^{2^{\ast}}+h_{2211}^{2^{\ast}})^{2}
+\frac{1}{2}(h_{2222}^{1^{\ast}}+h_{2221}^{2^{\ast}})^{2} \\
\geq&2(h_{1122}^{1^{\ast}}-h_{2211}^{1^{\ast}})^{2}+2(h_{1122}^{2^{\ast}}-h_{2211}^{2^{\ast}})^{2}
+\frac{1}{2}(h_{2222}^{1^{\ast}}-h_{2221}^{2^{\ast}})^{2}. \\
\endaligned
\end{equation}
\end{proof}
\noindent
If ${\vec H}\neq0$ at  $p$, we can choose a local orthogonal frame $\{e_{1}, e_{2} \}$
 such that
 \begin{equation}
e_{1^{*}}=\frac{{\vec H}}{|{\vec H}|}, \ \ H^{1^{*}}=H=|{\vec H}|,\ \ H^{2^{*}}=h_{11}^{2^{*}}+h_{22}^{2^{*}}=0.
\end{equation}
Defining $\lambda=h_{12}^{1^{*}}$, $\lambda_1=h_{11}^{1^{*}}$ and $\lambda_2=h_{22}^{1^{*}}$, we have
$h_{22}^{2^{*}}=-\lambda$.

\begin{lemma}
Let $X:M^2\rightarrow \mathbb{R}^4$ be  a 2-dimensional Lagrangian self-shrinker in $\mathbb R^4$.
If $S$ is constant,   ${\vec H}(p)\neq0$ and $\sum_{i,j,k,p}(h_{ijk}^{p^{\ast}})^{2}(p)=0$,
then we  have, at  $p$,
\begin{equation}
\aligned
\frac{1}{2}\mathcal{L}\sum_{i,j,k,p}(h_{ijk}^{p^{\ast}})^{2}=\sum_{i,j,k,l,p}(h_{ijkl}^{p^{\ast}})^{2}
\endaligned
\end{equation}
and
\begin{equation}
\aligned
&\frac{1}{2}\mathcal{L}\sum_{i,j,k,p}(h_{ijk}^{p^{\ast}})^{2}\\
=&H^{2}\bigl[H^{2}-2S+\frac{1}{2}H^{4}-H^2\lambda^2-KH^{2}-K^{2}\bigl]\\
&+H^{2}(S+2-\frac{3}{2}H^{2}-\lambda_{1}^{2}
-\lambda_{2}^{2}-2\lambda^2)(\lambda_{1}^{2}+\lambda_{2}^{2}+2\lambda^2).
\endaligned
\end{equation}

\end{lemma}
\begin{proof}
Since ${\vec H}\neq0$ at  $p$, we can choose  a local orthogonal frame $\{e_{1}, e_{2} \}$  such that
\begin{equation}
e_{1^{*}}=\frac{{\vec H}}{|{\vec H}|}, \ \ H^{1^{*}}=H=|{\vec H}|,\ \ H^{2^{*}}=h_{11}^{2^{*}}+h_{22}^{2^{*}}=0.
\end{equation}
By the definition of  $\lambda=h_{12}^{1^{*}}$, $\lambda_1=h_{11}^{1^{*}}$ and $\lambda_2=h_{22}^{1^{*}}$,
we have
$h_{22}^{2^{*}}=-\lambda$.
Since $S$ is constant and  $h_{ijk}^{p^{\ast}}=0$ at $p$,  we obtain from (2.34) of Lemma 2.3,
\begin{equation}
\aligned
\frac{1}{2}\mathcal{L}\sum_{i,j,k,p}(h_{ijk}^{p^{\ast}})^{2}
=\sum_{i,j,k,l,p}(h_{ijkl}^{p^{\ast}})^{2}.
\endaligned
\end{equation}
Furthermore, by making use of
\begin{equation}
S=\lambda_{1}^{2}+3\lambda_{2}^{2}+4\lambda^2, \ \ H\lambda_{1}=K+\lambda_{1}^{2}+\lambda_{2}^{2}+2\lambda^2,
\end{equation}
from (2.35) in Lemma 2.3, we have the following equations,  at $p$,
\begin{equation}
\aligned
&\frac{1}{2}\mathcal{L}\sum_{i,j,k,p}(h_{ijk}^{p^{\ast}})^{2}\\
=&(H^{2}-2S)H^{2}+(\frac{1}{2}H^{2}+\frac{1}{2}S+2)H^{2}(\lambda_{1}^{2}+\lambda_{2}^{2}+2\lambda^2)\\
&-K(H^{4}+H^{3}\lambda_{1})-H^{2}\bigl\{(\lambda_{1}^{2}+\lambda_{2}^{2}+2\lambda^2)^{2}+\lambda^2H^2\bigl\}\\
&-H^{3}(\lambda_{1}^{3}+\lambda_{2}^{3}+3H\lambda^2)\\
=&H^{2}\biggl[H^{2}-2S+(\frac{1}{2}H^{2}
+\frac{1}{2}S+2-\lambda_{1}^{2}-\lambda_{2}^{2}-2\lambda^2)(\lambda_{1}^{2}+\lambda_{2}^{2}+2\lambda^2)\\
&-K(H^{2}+K+\lambda_{1}^{2}+\lambda_{2}^{2}+2\lambda^2)-H^{2}(\lambda_{1}^{2}+\lambda_{2}^{2}-\lambda_{1}\lambda_{2}+4\lambda^2)\biggl]\\
=&H^{2}\biggl[H^{2}-2S+(S+2-\lambda_{1}^{2}-\lambda_{2}^{2}-2\lambda^2)(\lambda_{1}^{2}+\lambda_{2}^{2}+2\lambda^2)
-K(H^{2}+K)\\
&-H^{2}(\lambda_{1}^{2}+\lambda_{2}^{2}+2\lambda^2)
+\frac{1}{2}H^{4}-H^2\lambda^2-\frac{1}{2}H^{2}(\lambda_{1}^{2}+\lambda_{2}^{2}+2\lambda^2)\biggl]\\
=&H^{2}\bigl[H^{2}-2S+\frac{1}{2}H^{4}-H^2\lambda^2-KH^{2}-K^{2}\bigl]\\
&+H^{2}(S+2-\frac{3}{2}H^{2}-\lambda_{1}^{2}
-\lambda_{2}^{2}-2\lambda^2)(\lambda_{1}^{2}+\lambda_{2}^{2}+2\lambda^2).
\endaligned
\end{equation}
This finishes  the proof.
\end{proof}

\vskip2mm
\noindent
In order to  prove  our results,  we need the following important
generalized maximum principle for $\mathcal{L}$-operator on self-shrinkers  which was proved
by Cheng and Peng in \cite{CP}:
\begin{lemma} {\rm(}Generalized maximum principle for $\mathcal{L}$-operator {\rm)}
Let $X: M^n\to \mathbb{R}^{n+p}$ {\rm (}$p\geq 1${\rm)} be a complete self-shrinker with Ricci
curvature bounded from below. Let $f$ be any $C^2$-function bounded
from above on this self-shrinker. Then, there exists a sequence of points
$\{p_m\}\subset M^n$, such that
\begin{equation*}
\lim_{m\rightarrow\infty} f(X(p_m))=\sup f,\quad
\lim_{m\rightarrow\infty} |\nabla f|(X(p_m))=0,\quad
\limsup_{m\rightarrow\infty}\mathcal{L} f(X(p_m))\leq 0.
\end{equation*}
\end{lemma}
\noindent

\section{Examples of Lagrangian self-shrinkers in $\mathbb R^4$}

\vskip2mm
\noindent
It is known that the  Euclidean plane $\mathbb R^2$, the cylinder $S^1(1)\times \mathbb{R}^{1}$
and the Clifford torus $S^1(1)\times S^1(1)$ are the canonical Lagrangian self-shrinkers in $\mathbb R^4$.
Apart from the standard examples, there are many  other examples of complete Lagrangian self-shrinkers in $\mathbb R^4$.

\begin{example}
Let $\Gamma_1(s)=(x_1(s), y_1(s))^T$, $0\leq s<L_1$ and $\Gamma_2(t)=(x_2(t), y_2(t))^T$, $0\leq t<L_2$ be  two self-shrinker curves in $\mathbb{R}^2$
with arc length as parameters,  respectively.
We consider Riemannian product $\Gamma_1(s)\times \Gamma_2(t)$ of  $\Gamma_1(s)$ and $\Gamma_2(t)$ defined
by
$$
X(s,t)=\begin{pmatrix}x_1(s)\\ x_2(t)\\ y_1(s)\\ y_2(t)
\end{pmatrix}.
$$
\end{example}

\noindent We can prove $\Gamma_1(s)\times \Gamma_2(t)$  is a Lagrangian self-shrinker in $\mathbb R^4$
and the Gaussian curvature $K$ of $\Gamma_1(s)\times \Gamma_2(t)$  satisfies $K\equiv 0$.

\vskip1mm
\noindent
In \cite{AL}, Abresch and Langer  classfied
closed self-shrinking curves. For two closed self-shrinking curves $\Gamma_1(s)$ and $\Gamma_2(t)$ of  Abresch and Langer in
$\mathbb R^2$,  $\Gamma_1(s)\times \Gamma_2(t)$  is   a compact Lagrangian self-shrinker  in $\mathbb R^4$, which is called Abresch-Langer torus. It is known that complete and non-compact self-shrinking curves exist in $\mathbb R^2$ \rm(see \cite{H}\rm ).
Consequently, there
are many complete and non-compact Lagrangian self-shrinkers with zero  Gaussian curvature in $\mathbb R^4$.

\begin{example} For  a closed curve  $\gamma(t)=(x_1(t), x_2(t))^T$, $t\in I$,
such that its curvature $\kappa_{\gamma}$  satisfy
$$
\kappa_{\gamma}=E\dfrac{e^{\frac{|\gamma|^2}2}}{|\gamma|^2}(|\gamma|^2-1),
\ \ E^2= |\gamma|^4(1-(\dfrac{d|\gamma|}{dt})^2)e^{|\gamma|^2},
$$
where $E$ is a positive  constant.  In \cite{A}, Anciaux proved that
\begin{equation*}
\begin{aligned}
&X(s,t)=\begin{pmatrix}x_1(t)\cos s\\ x_1(t)\sin s\\ x_2(t)\cos s\\ x_2(t)\sin s\end{pmatrix}
\end{aligned}
\end{equation*}
defines a compact Lagrangian self-shrinker in $\mathbb R^4$, which is called  Anciaux torus,
and the squared norm $S$
of the second fundamental form  satisfies
$$
 S=E^2\dfrac{e^{\frac{|\gamma|^6}2}}{|\gamma|^2}(|\gamma|^4-2|\gamma|^2+4).
$$

\end{example}

\begin{example} For positive integers $m, n$ with  $(m,n)=1$,
define $X^{m,n}: \mathbb R^2\to \mathbb R^4$ by
\begin{equation*}
\begin{aligned}
X^{m,n}(s,t)=\sqrt{m+n}\begin{pmatrix}\dfrac{\cos s}{\sqrt n}\cos\sqrt{\frac nm}t\\  \dfrac{\sin s}{\sqrt m}\cos\sqrt{\frac mn}t\\
\dfrac{\cos s}{\sqrt n}\sin\sqrt{\frac nm}t\\  \dfrac{\sin s}{\sqrt m}\sin\sqrt{\frac mn}\end{pmatrix}.
\end{aligned}
\end{equation*}
\end{example}
\noindent Lee and Wang \cite{LeW} proved $X^{m,n}: \mathbb R^2\to \mathbb R^4$  is a Lagrangian self-shrinker in $\mathbb R^4$.
It is not difficult to prove  that the squared norm $S$ and the Gauss curvature $K$ of $X^{m,n}: \mathbb R^2\to \mathbb R^4$,
for  $m\leq n$,  satisfy
$$
  \dfrac{3m^2+n^2}{n(m+n)}\leq S \leq \dfrac{m^2+3n^2}{m(n+m)},
 $$
$$
 - \dfrac{n(n-m)}{m(m+n)}\leq K \leq \dfrac{m(n-m)}{n(n+m)}.
  $$

 \vskip10mm
\section{Proofs of the main results}

\vskip5mm
\noindent
First of all, we prove the following:

\vskip2mm
\noindent
\begin{theorem} Let $X: M^2\to \mathbb{R}^{4}$ be a
2-dimensional complete  Lagrangian self-shrinker   in $\mathbb R^4$.
If  the squared norm $S$ of the second fundamental form is constant, then  $S\leq 2$.
\end{theorem}

\begin{proof}
Since $S$ is constant, from the Gauss equations, we know that the Ricci curvature of $X: M^2\to \mathbb{R}^{4}$ is bounded from below.
We can apply the generalized maximum principle for $\mathcal L$-operator
to the function $-|X|^2$. Thus, there exists a sequence $\{p_m\}$ in $M^2$ such that
\begin{equation*}
\lim_{m\rightarrow\infty} |X|^2(p_m)=\inf |X|^2,\quad
\lim_{m\rightarrow\infty} |\nabla |X|^2(p_m)|=0,\quad
\liminf_{m\rightarrow\infty}\mathcal{L} |X|^2(p_m)\geq 0.
\end{equation*}
Since $|\nabla |X|^2|^2=4\sum_{i=1}^2\langle X, e_i\rangle^2$ and
\begin{equation*}
\dfrac12\mathcal L|X|^2=2-|X|^2,
\end{equation*}
we  have
\begin{equation}
\lim_{m\to\infty}\sum_j\langle X, e_j\rangle^2(p_m)=0 \ \ \text{\rm and} \ \ 2-\inf |X|^2\geq 0.
\end{equation}
Since $X: M^2\to \mathbb{R}^{4}$ is a self-shrinker, we know
\begin{equation}
H^{p^{\ast}}_{,i}=\sum_kh^{p^{\ast}}_{ik} \langle X,e_k \rangle, \quad  i, p=1,  2.
\end{equation}
 From the definition of the self-shrinker, (4.1)  and (4.2), we get
\begin{equation}
\inf|X|^2=\lim_{m\rightarrow\infty}H^2(p_m)\leq 2, \quad \lim_{m\rightarrow\infty}|\nabla^{\perp} \vec H|^2(p_m)=0.
\end{equation}
Since $S=\sum_{i, j, p^{\ast}}(h^{p^{\ast}}_{ij})^2$ is constant, from (2.25) in Lemma 2.1, we know
$\{h^{p^{\ast}}_{ij}(p_m)\}$ and $\{h^{p^{\ast}}_{ijl}(p_m)\}$ are bounded sequences for any $ i, j, l, p$.
Thus, we can assume
$$
\lim_{m\rightarrow\infty}h^{p^{\ast}}_{ijl}(p_m)=\bar h^{p^{\ast}}_{ijl},
\quad \lim_{m\rightarrow\infty}h^{p^{\ast}}_{ij}(p_m)=\bar h^{p^{\ast}}_{ij},
$$
for $i, j, l, p=1, 2$.

\noindent
Therefore, we have
\begin{equation}
\bar h^{p^{\ast}}_{11j}+\bar h^{p^{\ast}}_{22j}=0, \ \ \text{\rm for} \  j, p=1, 2
\end{equation}
and
$$
\sum_{i,j,p}h^{p^{\ast}}_{ij}h^{p^{\ast}}_{ijk}=0, \ \ \text{for } \ k=1, 2,
$$
because of $S$ constant.
Since $X: M^2\to \mathbb{R}^{4}$ is a Lagrangian self-shrinker,
$$
h^{1^{\ast}}_{11}h^{1^{\ast}}_{11k}+3h^{1^{\ast}}_{12}h^{1^{\ast}}_{12k}
+3h^{2^{\ast}}_{12}h^{2^{\ast}}_{12k}+h^{2^{\ast}}_{22}h^{2^{\ast}}_{22k}=0, \ \ \text{for } \ k=1, 2
$$
holds. Thus, we conclude
\begin{equation}
\bar h^{1^{\ast}}_{11}\bar h^{1^{\ast}}_{11k}+3\bar h^{1^{\ast}}_{12}\bar h^{1^{\ast}}_{12k}
+3\bar h^{2^{\ast}}_{12}\bar h^{2^{\ast}}_{12k}+\bar h^{2^{\ast}}_{22}\bar h^{2^{\ast}}_{22k}=0, \ \ \text{for } \ k=1, 2.
\end{equation}
If $\lim_{m\rightarrow\infty}H^2(p_m)=0$,  we get
$$
\bar h^{1^{\ast}}_{11}+\bar h^{1^{\ast}}_{22}=0, \   \ \ \bar h^{2^{\ast}}_{11}+\bar h^{2^{\ast}}_{22}=0.
$$
Consequently, from (4.4) and (4.5), we have the following equations for $ k=1, 2$,
\begin{equation*}
\begin{cases}
&\bar h^{1^*}_{11k}+\bar h^{1^*}_{22k}=0,\\
&\bar h^{2^*}_{11k}+\bar h^{2^*}_{22k}=0,\\
&4\bar h^{1^{\ast}}_{11}\bar h^{1^*}_{11k}+4\bar h^{2^{\ast}}_{11}\bar h^{2^*}_{11k}=0.
\end{cases}
\end{equation*}
Hence, we obtain $S=0$ or  $\bar h^{p^{\ast}}_{ijk}=0$  for any $i, j, k$ and  $p$.
According to (2.25) in Lemma 2.1, we have
$S=0$ or  $S=\dfrac23$.

\noindent
If $\lim_{m\rightarrow\infty}H^2(p_m)=\bar H^2\neq 0$, without loss of the generality, at each point $p_m$, we choose $e_1$, $e_2$
 such that
$$
\vec H=H^{1^{\ast}}e_{1^{\ast}}.
$$
Then we have
$$
\bar h^{1^*}_{11}+ \bar h^{1^*}_{22}=\bar H,  \ \ \bar h^{2^{\ast}}_{11}+\bar h^{2^{\ast}}_{22}=0
$$
and
\begin{equation}
\begin{cases}
&\bar h^{1^*}_{11k}+\bar h^{1^*}_{22k}=0,\\
&\bar h^{2^*}_{11k}+\bar h^{2^*}_{22k}=0,\\
&(\bar h^{1^*}_{11}-3\bar h^{1^*}_{22})\bar h^{1^*}_{11k}+4\bar h^{2^{\ast}}_{11}\bar h^{2^*}_{11k}=0.
\end{cases}
\end{equation}
If $\bar h^{1^{\ast}}_{11}=3\bar h^{1^{\ast}}_{22}$ and $\bar h^{2^{\ast}}_{11}=0$, we know
$$
\lim_{m\rightarrow\infty}H^2(p_m)=(\bar \lambda_1+\bar \lambda_2)^2=16\bar \lambda_2^2\leq 2 \ \text{\rm and } \
S=12\bar \lambda_2^2\leq \dfrac32.
$$
If $\bar h^{1^{\ast}}_{11}\neq 3\bar h^{1^{\ast}}_{22}$ or  $\bar h^{2^{\ast}}_{11}\neq 0$, we have
$\bar h^{p^{\ast}}_{ijk}=0$  for any $i, j, k, p$ from (4.6).  Thus, from (2.25) in Lemma 2.1, we get

\begin{equation*}
\begin{aligned}
&0=S(1-\dfrac32S)+2\bar H^2S-\dfrac12\bar H^4-\bar H^2(\bar \lambda_1^2+\bar\lambda_2^2+2\bar\lambda^2 )\\
&=S(1-\dfrac12S)-(S-\bar H^2)^2-\dfrac12\bar H^2(\bar \lambda_1-\bar\lambda_2)^2-2\bar H^2\bar\lambda^2.
\end{aligned}
\end{equation*}
Then we conclude
$$
S\leq 2.
$$
This completes the proof of Theorem 1.1.
\end{proof}

\vskip2mm
\noindent
Since $S$ is constant, from the result of Cheng and Peng in \cite{CP}, we know that  $S=0$ or $S=1$ if $S\leq 1$ .
Thus, we  only need to prove the  following
\begin{theorem}
There are no 2-dimensional complete   Lagrangian self-shrinkers  $X: M^2\to \mathbb{R}^{4}$
with constant squared norm $S$ of the second fundamental form and $1<S< 2$.
\end{theorem}

\noindent
The following lemma is key in this paper.
\begin{lemma} If $X: M^2\to \mathbb{R}^{4}$ is a 2-dimensional  complete  Lagrangian self-shrinker
in $\mathbb R^4$ with $S=$constant
and  $1\leq S\leq 2$, there exists a sequence $\{p_m\}$ in $M$  such that
$$
\lim_{m\rightarrow\infty} H^2(p_m)=\sup H^2, \quad
\lim_{m\rightarrow\infty}h^{p^{\ast}}_{ijl}(p_m)=\bar h^{p^{\ast}}_{ijl},
\quad \lim_{m\rightarrow\infty}h^{p^{\ast}}_{ij}(p_m)=\bar h^{p^{\ast}}_{ij},
$$
for $i, j, l,p=1, 2$,  and one  can choose an orthonormal frame $e_1, e_2$ at $p_m$
such that
$\bar \lambda=\bar h^{1^{\ast}}_{12}=0$.
\end{lemma}
\begin{proof}
\noindent
From (2.25) and (2.26) in Lemma 2.1, we have
\begin{equation*}
\begin{aligned}
\dfrac12\mathcal{L}H^2
&=|\nabla^{\perp} \vec H|^2+H^2-\sum_{i,j,k,p}(h^{p^{\ast}}_{ijk})^2-S(1-\dfrac32S) -2H^2S+\dfrac12H^4\\
&=|\nabla^{\perp} \vec H|^2-\sum_{i,j,k,p}(h^{p^{\ast}}_{ijk})^2+\dfrac12(H^2-S)(H^2-3S+2).
\end{aligned}
\end{equation*}
If,  at  $p\in M$, $H=0$, we  have $H^2<S$. If $H\neq 0$ at $p\in M$, we choose $e_1$, $e_2$ such that
$$
\vec H=H^{1^{\ast}}e_{1^{\ast}}.
$$
From $2ab\leq \epsilon a^2+\dfrac1{\epsilon}b^2$, we obtain
$$
S=\lambda_1^2 +3\lambda_2^2+4\lambda^2, \ \ H^2=(\lambda_1+\lambda_2)^2\leq \dfrac43(\lambda_1^2 +3\lambda_2^2)\leq \dfrac43S,
$$
where we denote  $ \lambda_1= h^{1^{\ast}}_{11}$, $ \lambda_2= h^{1^{\ast}}_{22}$
and $\lambda=h^{1^{\ast}}_{12}$.
Hence, we have on $M$
$$
H^2\leq \dfrac43S
$$
and the equality holds if and only if $\lambda_1=3\lambda_2$
and  $\lambda=0$.
Thus, by applying  the generalized maximum principle of Cheng and Peng \cite{CP} to $H^2$,
there exists a sequence $\{p_m\}$ in $M^2$ such that
\begin{equation*}
\lim_{m\rightarrow\infty} H^2(p_m)=\sup H^2,\quad
\lim_{m\rightarrow\infty} |\nabla H^2(p_m)|=0,\quad
\limsup_{m\rightarrow\infty}\mathcal{L}H^2(p_m)\leq 0.
\end{equation*}
Since $X: M^2\to \mathbb{R}^{4}$ is a self-shrinker, we have
\begin{equation}
H^{p^{\ast}}_{,i}=\sum_k h^{p^{\ast}}_{ik} \langle X,e_k \rangle, \quad i, p=1, 2.
\end{equation}
According to  $1\leq  S\leq 2$, we know $\sup H^2>0$.  Hence, without loss of the generality,
at each point $p_m$, we can assume $H(p_m)\neq 0$  and  choose $e_1$, $e_2$  such that
$$
\vec H=H^{1^{\ast}}e_{1^{\ast}}.
$$
From (2.25) in Lemma 2.1, Lemma 2.3  and the definition of $S$, we know that
$\{h^{p^{\ast}}_{ij}(p_m)\}$,  $\{h^{p^{\ast}}_{ijl}(p_m)\}$ and $\{h^{p^{\ast}}_{ijkl}(p_m)\}$, for any $ i, j, k, l, p$, are bounded sequences.
We can assume
$$
\lim_{m\rightarrow\infty}h^{p^{\ast}}_{ijl}(p_m)=\bar h^{p^{\ast}}_{ijl}, \quad \lim_{m\rightarrow\infty}h^{p^{\ast}}_{ij}(p_m)=\bar h^{p^{\ast}}_{ij},
\quad \lim_{m\rightarrow\infty}h^{p^{\ast}}_{ijkl}(p_m)=\bar h^{p^{\ast}}_{ijkl},
$$
for $i, j, k, l,p=1, 2$
\vskip1mm
\noindent
and get
\begin{equation}
\begin{cases}
\begin{aligned}
&\lim_{m\rightarrow\infty} H^2(p_m)=\sup H^2=\bar H^2,\quad
\lim_{m\rightarrow\infty} |\nabla H^2(p_m)|=0,\\
&0\geq
\lim_{m\rightarrow\infty} |\nabla^{\perp} \vec H|^2(p_m)-\sum_{i,j,k,p}(\bar h^{p^{\ast}}_{ijk})^2+\dfrac12(\bar H^2-S)(\bar H^2-3S+2).
\end{aligned}
\end{cases}
\end{equation}
From $\lim_{k\rightarrow\infty} |\nabla H^2(p_m)|=0$ and $|\nabla H^2|^2=4\sum_i(\sum_{p^*}H^{p^{\ast}}H^{p^{\ast}}_{,i})^2$, we have
\begin{equation}
\bar H^{1^*}_{,k}=0.
\end{equation}
From (4.7), we obtain
\begin{equation}
\begin{cases}
\begin{aligned}
&\bar \lambda_1\lim_{m\rightarrow\infty} \langle X,e_1\rangle(p_m)+\bar \lambda \lim_{m\rightarrow\infty} \langle X,e_2 \rangle(p_m)=0,\\
&\bar \lambda\lim_{m\rightarrow\infty} \langle X,e_1\rangle(p_m)
+\bar \lambda_2 \lim_{m\rightarrow\infty} \langle X,e_2 \rangle(p_m)=0.
\end{aligned}
\end{cases}
\end{equation}
We will then prove $\bar \lambda=0$.

\noindent
Let us assume $\bar \lambda \neq 0$, we will get a contradiction. The proof is  divided into three cases.

\vskip2mm
\noindent
{\bf Case 1: $ \bar \lambda_2=0$}.

\noindent Since $\bar H^2\neq 0$, we have $\bar \lambda_1\neq 0$. From (4.10), we get
$$
\lim_{m\rightarrow\infty} \langle X,e_1\rangle(p_m)=
\lim_{m\rightarrow\infty} \langle X,e_2 \rangle(p_m)=0.
$$
Thus, for  $k=1, 2$, from (4.5) and (4.7),
\begin{equation}
\begin{cases}
\begin{aligned}
&\bar h^{1^*}_{11k}+\bar h^{1^*}_{22k}=0, \\
&\bar h^{2^*}_{11k}+\bar h^{2^*}_{22k}=0,\\
&\bar \lambda_1\bar h^{1^{\ast}}_{11k}+4\bar \lambda \bar h^{2^{\ast}}_{11k}=0.
\end{aligned}
\end{cases}
\end{equation}
We  can draw a conclusion,  for any $i, j, k, p$,
$$
\bar h^{p^*}_{ijk}=0.
$$
From (4.8), we know $S\leq \bar H^2$, which is in contradiction to $S=\bar H^2+4\bar \lambda^2>\bar H^2$ .

\vskip2mm
\noindent
{\bf Case 2: $ \bar \lambda_1=0$}.

\noindent In this case, we have
$$
 \bar \lambda_2\neq 0,  \ \ \bar H^2=\bar \lambda_2^2, \ \
 S=3\bar \lambda_2^2+4\bar \lambda^2=3\bar H^2+4\bar \lambda^2.
 $$
From (4.10), we obtain
$$
\lim_{m\rightarrow\infty} \langle X,e_1\rangle(p_m)=
\lim_{m\rightarrow\infty} \langle X,e_2 \rangle(p_m)=0.
$$
Therefore, we infer
\begin{equation}
\begin{cases}
\begin{aligned}
&\bar h^{1^*}_{11k}+\bar h^{1^*}_{22k}=0, \\
&\bar h^{2^*}_{11k}+\bar h^{2^*}_{22k}=0,\\
&3\bar \lambda_2\bar h^{1^{\ast}}_{22k}+4\bar \lambda \bar h^{2^{\ast}}_{11k}=0.
\end{aligned}
\end{cases}
\end{equation}
By solving the above system of equations, we have  for any $i, j, k, p$,
$$
\bar h^{p^*}_{ijk}=0.
$$
From (4.8), we know
$$
 (\bar H^2-S)(\bar H^2-3S+2)=(2\bar H^2+ 4\bar \lambda^2)(2S-2+2\bar H^2+ 4\bar \lambda^2)\leq 0,
$$
it is impossible since $S\geq1$.
\vskip2mm
\noindent
{\bf Case 3: $ \bar \lambda_1\bar \lambda_2\neq 0$}.

\noindent From (4.10), we have
\begin{equation}
(\bar \lambda_1\bar\lambda_2 -\bar \lambda^2)\lim_{m\rightarrow\infty} \langle X,e_2 \rangle(p_m)=0.
\end{equation}
If $\bar \lambda_1\bar\lambda_2 = \bar \lambda^2$, we get, for $k=1, 2$,  in view of  (4.5) and (4.9),
\begin{equation*}
\begin{cases}
\begin{aligned}
&\bar h^{1^*}_{11k}+\bar h^{1^*}_{22k}=0, \\
&(\bar \lambda_1-3\bar \lambda_2)\bar h^{1^{\ast}}_{11k}
+3\bar \lambda \bar h^{2^{\ast}}_{11k}-\bar \lambda \bar h^{2^{\ast}}_{22k}=0.
\end{aligned}
\end{cases}
\end{equation*}
By solving the above system of equations, we have
$$
(\bar \lambda_1+3\bar \lambda_2)^2\bar h^{1^*}_{111}=-4\bar \lambda^2\bar h^{2^*}_{222}.
$$
Hence, we obtain
\begin{equation*}
\begin{aligned}
&\bar h^{1^*}_{111}=\dfrac{-4\bar \lambda^2}{(\bar \lambda_1+3\bar \lambda_2)^2}\bar h^{2^*}_{222},\\
&\bar h^{1^*}_{221}=-\bar h^{1^*}_{111},\\
&\bar h^{1^*}_{222}=-\bar h^{2^*}_{111}={\color{red}-}
\dfrac{\bar \lambda(\bar \lambda_1-3\bar \lambda_2)}{(\bar \lambda_1+3\bar \lambda_2)^2}\bar h^{2^*}_{222}.
\end{aligned}
\end{equation*}
Since
$$
\lim_{m\rightarrow\infty} |\nabla^{\perp} \vec H|^2(p_m)=(\bar h^{2^*}_{112}+\bar h^{2^*}_{222})^2=
\dfrac{(10 \bar \lambda^2+\bar \lambda_1^2+9\bar \lambda_2^2)^2}{(\bar \lambda_1+3\bar \lambda_2)^4}(\bar h^{2^*}_{222})^2
$$
and
$$
\sum_{i,j,k,p}(\bar h^{p^{\ast}}_{ijk})^2=7(\bar h^{1^*}_{111})^2+8(\bar h^{2^*}_{111})^2+(\bar h^{2^*}_{222})^2=
\dfrac{(10 \bar \lambda^2+\bar \lambda_1^2+9\bar \lambda_2^2)^2}{(\bar \lambda_1+3\bar \lambda_2)^4}(\bar h^{2^*}_{222})^2,
$$
 we get the following inequality from (4.8)
$$
(\bar H^2-S)(\bar H^2-3 S+2)\leq 0,
$$
that is,
$$
S\leq \bar H^2\leq 3S-2.
$$
It is impossible because of  $S=\bar \lambda_1^2+3\bar \lambda_2^2+4\bar \lambda^2>\bar \lambda_1^2+ \bar \lambda_2^2+2\bar \lambda^2=\bar H^2$.
Hence, we obtain  $\bar \lambda_1\bar\lambda_2 \neq  \bar \lambda^2$.

\noindent From (4.10) and (4.13), we have
$$
\lim_{m\rightarrow\infty} \langle X,e_2 \rangle(p_m)=\lim_{m\rightarrow\infty} \langle X,e_1 \rangle(p_m)=0.
$$
Thus, we know from (4.7)
$$
\bar H^{p^*}_{,k}=0,
$$
for any $k, p=1, 2$. Hence we infer
\begin{equation*}
\begin{cases}
&\bar h^{p^*}_{11k}+\bar h^{p^*}_{22k}=0,\\
&(\bar \lambda_1-3\bar \lambda_2)\bar h^{1^*}_{11k}+4\bar \lambda\bar h^{2^*}_{11k}=0.
\end{cases}
\end{equation*}
Through the above system, we have
$$
 \sum_{i,j,k,p}(\bar h^{p^{\ast}}_{ijk})^2=0.
$$
From  (4.8) and (2.25) in Lemma 2.1, we get

\begin{equation}
\begin{aligned}
&S\leq \bar H^2\leq 3S-2,\\
&S(1-\dfrac12S)-(S-\bar H^2)^2+\dfrac12\bar  H^4-\bar H^2(\bar \lambda_1^2+\bar\lambda_2^2+2\bar \lambda^2)=0.
\end{aligned}
\end{equation}

%From (2.19) and taking limit, we know, for $\ i, j, p=1, 2$
%\begin{equation*}
%\bar h_{11ij}^{p^{*}}+\bar h_{22ij}^{p^{*}}=\bar h_{ij}^{p^{*}}-\sum_k\bar h_{ik}^{p^{*}}\bar h_{jk}^{1^{*}}\bar H.
%\end{equation*}
%Hence, we have
%\begin{equation*}
%\begin{aligned}
%& \bar h_{1122}^{2^{*}}+\bar h_{2222}^{2^{*}}=-\bar \lambda.
%\end{aligned}
%\end{equation*}
%Because of  $\bar\lambda\neq 0$, we know $(\bar h_{1122}^{2^{*}})^2+(\bar h_{2222}^{2^{*}})^2>0$.
%Therefore, we obtain,

\noindent From Lemma 2.5 and taking limit,
\begin{equation*}
\aligned
0&\leq\sum_{i,j,k,l,p}(\bar h_{ijkl}^{p^{\ast}})^{2}
=\frac{1}{2}\lim_{m\rightarrow\infty} \mathcal{L}\sum_{i,j,k,p}( h_{ijk}^{p^{\ast}})^{2}(p_m)\\
&=\bar H^{2}\bigl[\bar H^{2}-2S+\frac{1}{2}\bar H^{4}-\bar K\bar H^{2}-\bar K^{2}\bigl]-\bar \lambda^2\bar H^{4}\\
&\ \ +\bar H^{2}(S+2-\frac{3}{2}\bar H^{2}-\bar \lambda_{1}^{2}
-\bar \lambda_{2}^{2}-2\bar \lambda^2)(\bar \lambda_{1}^{2}+\bar \lambda_{2}^{2}+2\bar\lambda^2)\\
&<\bar H^{2}\bigl[\bar H^{2}-2S+\frac{1}{2}\bar H^{4}-\bar K\bar H^{2}-\bar K^{2}\bigl]\\
&\ \ +\bar H^{2}(S+2-\frac{3}{2}\bar H^{2}-\bar \lambda_{1}^{2}
-\bar \lambda_{2}^{2}-2\bar \lambda^2)(\bar \lambda_{1}^{2}+\bar \lambda_{2}^{2}+2\bar\lambda^2).
\endaligned
\end{equation*}
According to (4.14), we have
\begin{equation*}
\aligned
&0<\bar H^{2}\bigl[\bar H^{2}-2S+\frac{1}{2}\bar H^{4}-\bar K\bar H^{2}-\bar K^{2}\bigl]\\
&\ \ + \biggl(S+2-\frac{3}{2}\bar H^{2}-\dfrac{1}{\bar  H^2}\bigl(S(1-\dfrac12S)-(S-\bar H^2)^2+\dfrac12\bar  H^4\bigl)\biggl)\\
&\ \ \times\biggl(S(1-\dfrac12S)-(S-\bar H^2)^2+\dfrac12\bar  H^4\biggl)\\
&=\dfrac{1}{4\bar  H^2}\biggl(\bar  H^8-2S\bar H^6-6S(S-1)\bar  H^4+2S(2-3S)^2\bar  H^2-(2-3S)^2S^2\biggl)\leq 0.
\endaligned
\end{equation*}
This  is a contradiction. In fact, we consider a function $f(t)$ defined by
\begin{equation}
\aligned
&f(t)=t^4-2St^3-6S(S-1)t^2+2S(2-3S)^2t-(2-3S)^2S^2,
\endaligned
\end{equation}
for $S\leq t\leq 3S-2$. Thus, we have
\begin{equation}
f^{'}(t)=4t^3-6St^2-12S(S-1)t+2S(2-3S)^2,\ \
f^{''}(t)=12(t^2-St-S(S-1)),
\end{equation}
$f^{''}(t)<0$ for  $t\in (S,\frac{S+\sqrt{S^2+4S(S-1)}}{2})$,  $f^{''}(t)>0$
for  $t\in (\frac{S+\sqrt{S^2+4S(S-1)}}{2}, 3S-2)$.  Hence, $f^{'}(t)$ is a decreasing function
for  $t\in (S,\frac{S+\sqrt{S^2+4S(S-1)}}{2})$ and $f^{'}(t)$ is an increasing function for  $t\in (\frac{S+\sqrt{S^2+4S(S-1)}}{2}, 3S-2)$.
According to
\begin{equation}
f(S)=2(S-1)(S-2)S^2\leq 0,\ \ f(3S-2)=2(3S-2)^2(S-1)(S-2)\leq 0,
\end{equation}
we conclude $f(t)\leq 0$ for $t\in (S, 3S-2)$  because of $f^{'}(S)=4S(S-1)(S-2)\leq 0$. Hence,
we obtain  $\bar \lambda=0$  and
$$
\bar h^{1^{*}}_{ij}=\bar \lambda_{i}\delta_{ij}, \ \ \bar H=\bar \lambda_1+\bar \lambda_2, \ \ S=\bar \lambda_1^2+3\bar \lambda_2^2.
$$
\end{proof}

\noindent
Since the proof of Theorem 4.2  is very long, we will divide the proof into three steps.

\noindent
In the first step, we prove the following:
\begin{proposition} If $X: M^2\to \mathbb{R}^{4}$ is a 2-dimensional  complete   Lagrangian self-shrinker
in $\mathbb R^4$ with $S=$constant
and  $1<S<2$, there exists a sequence $\{p_m\}$ in $M$ and at $p_m$, we can choose an orthonormal $e_1, e_2$ such that
$$
\lim_{m\rightarrow\infty}h^{p^{\ast}}_{ijl}(p_m)=\bar h^{p^{\ast}}_{ijl},
\quad \lim_{m\rightarrow\infty}h^{p^{\ast}}_{ij}(p_m)=\bar h^{p^{\ast}}_{ij},
$$
for $i, j, l,p=1, 2$,   $\bar \lambda =0$  and
the following holds, either
\begin{equation}
\begin{cases}
\begin{aligned}
& \sum_{i,j,k,p}(\bar h^{p^{\ast}}_{ijk})^2=0, \ \    \ \ \bar\lambda_1\bar\lambda_2\neq 0, \\
& S<\sup H^2=\bar H^2 \leq 3S-2 \ \ \text{\rm and } \ \  S<\sup H^2<\dfrac43S,
\end{aligned}
\end{cases}
\end{equation}
or
\begin{equation}
\begin{cases}
\begin{aligned}
&\bar \lambda_1=3\bar \lambda_2, \ \ \bar\lambda_1\bar\lambda_2\neq 0,
\ \ \bar h^{p^{\ast}}_{11k}+\bar h^{p^{\ast}}_{22k}=0\\
&\ \sup H^2=\dfrac43S, \ \  S{\color{red}\geq}\dfrac65,
\end{aligned}
\end{cases}
\end{equation}
for $k, p=1, 2$, where  we denote  $\bar \lambda_1=\bar h^{1^{\ast}}_{11}$, $\bar \lambda_2=\bar h^{1^{\ast}}_{22}$
and $\bar \lambda=\bar h^{1^{\ast}}_{12}$.
\end{proposition}

\vskip2mm
\noindent
\begin{proof}
By making use of the same assertion as in the proof of Lemma 4.1,
there exists a sequence $\{p_m\}$ in $M^2$ such that
$$
\lim_{m\rightarrow\infty}h^{p^{\ast}}_{ijl}(p_m)=\bar h^{p^{\ast}}_{ijl}, \quad \lim_{m\rightarrow\infty}h^{p^{\ast}}_{ij}(p_m)=\bar h^{p^{\ast}}_{ij},
\quad \lim_{m\rightarrow\infty}h^{p^{\ast}}_{ijkl}(p_m)=\bar h^{p^{\ast}}_{ijkl},
$$
for $i, j, k, l,p=1, 2$ and
\begin{equation}
\begin{cases}
\begin{aligned}
&\lim_{m\rightarrow\infty} H^2(p_m)=\sup H^2=\bar H^2,\quad
\lim_{m\rightarrow\infty} |\nabla H^2(p_m)|=0,\\
&0\geq
\lim_{m\rightarrow\infty} |\nabla^{\perp} \vec H|^2(p_m)-\sum_{i,j,k,p}(\bar h^{p^{\ast}}_{ijk})^2+\dfrac12(\bar H^2-S)(\bar H^2-3S+2),
\end{aligned}
\end{cases}
\end{equation}
with $\bar \lambda=0$.
From $\lim_{k\rightarrow\infty} |\nabla H^2(p_m)|=0$ and $|\nabla H^2|^2=4\sum_i(\sum_{p^*}H^{p^{\ast}}H^{p^{\ast}}_{,i})^2$, we have
\begin{equation}
\bar H^{1^*}_{,k}=0.
\end{equation}
From (4.7) and $\bar \lambda=0$, we have
\begin{equation*}
\begin{aligned}
&\bar \lambda_1\lim_{m\rightarrow\infty} \langle X,e_1\rangle(p_m)=0, \ \
\bar \lambda_2 \lim_{m\rightarrow\infty} \langle X,e_2 \rangle(p_m)=0,
\end{aligned}
\end{equation*}
it means that,
$$
\bar \lambda_i\lim_{m\rightarrow\infty} \langle X,e_i \rangle(p_m)=0.
$$
According to
$S=\bar \lambda_1^2 +3\bar \lambda_2^2>1$  and  $\sup H^2=(\bar \lambda_1+\bar \lambda_2)^2$, if $\bar \lambda_2=0$,
we have
$$
\bar \lambda_1\neq 0,  \ \ S=\sup H^2, \ \
\bar H^{2^*}_{,k}=0
$$
because of  $H^{p^{\ast}}_{,i}=\sum_k h^{p^{\ast}}_{ik} \langle X,e_k \rangle$.
Hence, by using the same
calculations as in (4.6), we have
\begin{equation}
\begin{cases}
&\bar h^{1^*}_{11k}+\bar h^{1^*}_{22k}=0,\\
&\bar h^{2^*}_{11k}+\bar h^{2^*}_{22k}=0,\\
&\bar \lambda_1\bar h^{1^*}_{11k}=0.
\end{cases}
\end{equation}
Then we obtain
$$
 \sum_{i,j,k,p}(\bar h^{p^{\ast}}_{ijk})^2=0.
$$
From $ S=\sup H^2$  and (2.25), we get
$S=1$ or $S=0$. It is impossible.
If $\bar \lambda_1=0$,  we have
$$
\bar \lambda_2\neq 0,  \ \ S=3\sup H^2, \ \
\bar H^{2*}_{,1}=0.
$$
In this way, by using the same calculations as in (4.6), we get
\begin{equation*}
\begin{cases}
&\bar h^{1^*}_{11k}+\bar h^{1^*}_{22k}=0,\\
&\bar h^{2^*}_{111}+\bar h^{2^*}_{221}=0,\\
&3\bar \lambda_2\bar h^{1^*}_{22k}=0.
\end{cases}
\end{equation*}
So, we know
$$
\bar h^{p^{\ast}}_{ijk}=0, \ \ \text{\rm except} \ \ i=j=k=p^{\ast}=2
$$
and
$$
|\nabla^{\perp} \vec H|^2=\sum_{i,j,k,p}(\bar h^{p^{\ast}}_{ijk})^2 \ \ \text{\rm and}\ \  \dfrac12(H^2-S)(H^2-3S+2)\leq 0.
$$
Hence
$S\leq \dfrac34$. This is also impossible.
\vskip1mm

\noindent
We get  $\bar \lambda_1\bar \lambda_2\neq 0$.

\noindent Because of
$$
 H^{1^{\ast}}_{,i}=\sum_k h^{1^{\ast}}_{ik} \langle X,e_k \rangle, \ \ H^{2^{\ast}}_{,i}=\sum_k h^{2^{\ast}}_{ik} \langle X,e_k \rangle,
 $$
for $i=1, 2 $,
 we obtain $\lim_{m\rightarrow\infty} \langle X,e_i \rangle(p_m)=0$ from $\bar H^{1^*}_{,i}=0$ and
 $\bar \lambda_1\bar \lambda_2\neq 0$. Thus, we have $\bar H^{2^*}_{,i}=0$,
  % Hence, $\bar H^{1^*}_{,i}=0$, $\bar H^{2^*}_{,i}=0$ and
 then we get from (4.5),  for $i=1, 2 $,
\begin{equation}
\begin{cases}
&\bar h^{1^*}_{11i}+\bar h^{1^*}_{22i}=0,\\
&\bar h^{2^*}_{11i}+\bar h^{2^*}_{22i}=0,\\
&(\bar \lambda_1-3\bar \lambda_2)\bar h^{1^*}_{11i}=0.
\end{cases}
\end{equation}
If $\bar \lambda_1 \neq 3\bar \lambda_2$, we have
$$
 \sum_{i,j,k,p}(\bar h^{p^{\ast}}_{ijk})^2=0.
$$
Therefore, from (4.8) and (2.25), we get
$$
S < \sup H^2 \leq 3S-2.
$$
If $\bar \lambda_1 = 3\bar \lambda_2$, we have $\sup H^2=\dfrac43S$ and $\lim_{m\rightarrow\infty} |\nabla^{\perp} \vec H|^2(p_m)=0$.
From (2.25), we know
\begin{equation}
\begin{aligned}
\sum_{i,j,k,p}(\bar h^{p^{\ast}}_{ijk})^2
&=-S(1-\dfrac12S)+(S-\sup H^2)^2
-\dfrac12(\sup H^2)^2+\bar H^2(\bar \lambda_1^2+\bar\lambda_2^2)\\
&=(\frac56S-1)S.
\end{aligned}
\end{equation}
Hence, we get
$S\geq \dfrac65$.  When  $S= \dfrac65$,  from (4.24), we have
$\sum_{i,j,k,p}(\bar h^{p^{\ast}}_{ijk})^2=0$. It  finishes  the proof of Proposition 4.1.
\end{proof}

\vskip2mm
\noindent
In the step 2, we prove the following:
\begin{proposition} Under the assumptions of Proposition {\rm 4.1},
the formula {\rm (4.21)} in Proposition {\rm 4.1}  does not occur.
\end{proposition}
\begin{proof}
If  the formula (4.19) holds, we have
$$
\bar \lambda=0, \ \  \bar \lambda_1 = 3\bar \lambda_2, \ \ S{\color{red}\geq}\dfrac65,
\ \  \sup H^2=\bar H^2=\dfrac43S
$$
and
\begin{equation}
\bar H_{,k}^{p^{\ast}}=\bar h^{p^{\ast}}_{11k}+\bar h^{p^{\ast}}_{22k}=0, \ \text{\rm for}\ k, p=1, 2.
\end{equation}
From (2.6) and (4.25), we have
\begin{equation*}
\begin{aligned}
 &\sum_{i,j}(\bar h^{1^{\ast}}_{ij1})^2=(\bar h^{1^{\ast}}_{111})^2+(\bar h^{1^{\ast}}_{221})^2+2(\bar h^{1^{\ast}}_{121})^2
 =2(\bar h^{1^{\ast}}_{111})^2+2(\bar h^{1^{\ast}}_{112})^2\\
 &\sum_{i,j}(\bar h^{2^{\ast}}_{ij1})^2=(\bar h^{2^{\ast}}_{111})^2+(\bar h^{2^{\ast}}_{221})^2+2(\bar h^{2^{\ast}}_{121})^2
 =2(\bar h^{1^{\ast}}_{111})^2+2(\bar h^{1^{\ast}}_{112})^2\\
 &\sum_{i,j,p}(\bar h^{p^{\ast}}_{ij1})^2
 =\sum_{p}\bigl[(\bar h^{p^{\ast}}_{111})^2+(\bar h^{p^{\ast}}_{221})^2+2(\bar h^{p^{\ast}}_{121})^2\bigl]
 =\sum_{p}\bigl[2(\bar h^{p^{\ast}}_{111})^2+2(\bar h^{p^{\ast}}_{112})^2\bigl]\\
 &\sum_{i,j,p}(\bar h^{p^{\ast}}_{ij2})^2
 =\sum_{p}\bigl[(\bar h^{p^{\ast}}_{112})^2+(\bar h^{p^{\ast}}_{222})^2+2(\bar h^{p^{\ast}}_{122})^2\bigl]
 =\sum_{p}\bigl[2(\bar h^{p^{\ast}}_{111})^2+2(\bar h^{p^{\ast}}_{112})^2\bigl]\\
 &\sum_{i,j,p}(\bar h^{p^{\ast}}_{ijk})^2
 =\sum_{i,j,p}(\bar h^{p^{\ast}}_{ij1})^2+\sum_{i,j,p}(\bar h^{p^{\ast}}_{ij2})^2
 =8(\bar h^{1*}_{111})^2+8(\bar h^{1*}_{112})^2.
\end{aligned}
\end{equation*}
From (2.25) in Lemma 2.1 and $\bar H^2=\dfrac43S$, we get
$$
\sum_{i,j,p^{\ast}}(\bar h^{p^{\ast}}_{ijk})^2=(\frac56S-1)S.
$$
Thus, we obtain
$$
\sum_{i,j}(\bar h^{1^{\ast}}_{ij1})^2=\sum_{i,j}(\bar h^{2^{\ast}}_{ij1})^2=\dfrac14(\frac56S-1)S
$$
and
$$
\sum_{i,j,p^{\ast}}(\bar h^{p^{\ast}}_{ij1})^2=\sum_{i,j,p^{\ast}}(\bar h^{p^{\ast}}_{ij2})^2=\dfrac12(\frac56S-1)S.
$$
Since
$$
\bar H\bar \lambda_1=\frac{\bar H^2-S}2+\bar\lambda_1^2+\bar\lambda_2^2=\dfrac{\bar H^2}2+\dfrac{S}3=S,
$$
$$
  \bar H(\bar\lambda_1^3+ \bar\lambda_2^3)=\bar H^2(\bar\lambda_1^2
  -\bar\lambda_1\bar\lambda_2+\bar\lambda_2^2)=\dfrac{7S^2}{9},
$$
according to (2.35) in Lemma 2.3, we get

\begin{equation}
\begin{aligned}
&\dfrac12\lim_{m\to \infty}\mathcal L\sum_{i,j,k,p^{\ast}}(h^{p^{\ast}}_{ijk})^2(p_m)\\
&=(\bar H^2-2S)\bar H^2+(\dfrac12\bar H^2+\dfrac S2+2)\bar H^2(\bar \lambda_1^2+\bar \lambda_2^2)\\
&\ \ -\dfrac{\bar H^2-S}2(\bar H^4+\bar H^3\bar \lambda_1)-\bar H^2(\bar \lambda_1^2+\bar \lambda_2^2)^2
-\bar H^3(\bar \lambda_1^3+\bar \lambda_2^3)+\bar H^2\sum_{i,j,k}(\bar h^{1^{\ast}}_{ijk})^2\\
\end{aligned}
\end{equation}
\begin{equation*}
\begin{aligned}&=(\bar H^2-2S)\bar H^2+(\dfrac12\bar H^2+\dfrac S2+2)\bar H^2\dfrac{5S}6\\
&\ \ -\dfrac{\bar H^2-S}2\bigl(\bar H^4+\bar H^2S\bigl)
-\bar H^2(\dfrac{5S}6)^2-\bar H^2\dfrac{7S^2}{9}+\bar H^2\sum_{i,j,k}(\bar h^{1^{\ast}}_{ijk})^2\\
&=S^2(\dfrac23-\dfrac{17}{27}S)<0.
\end{aligned}
\end{equation*}
On the other hand, from (2.34) in Lemma 2.3, we have
\begin{equation}
\begin{aligned}
&\dfrac12\lim_{m\to \infty}\mathcal L\sum_{i,j,k,p^{\ast}} (h^{p^{\ast}}_{ijk})^2(p_m)\\
&=\sum_{i,j,k,l,p^{\ast}}(\bar h^{p^{\ast}}_{ijkl})^2+(5\bar H^2-5S+2)\sum_{i,j,k,p^{\ast}}(\bar h^{p^{\ast}}_{ijk})^2\\
&\ \ -2\bar H\sum_{i,j,k,p^{\ast}}(\bar h^{p^{\ast}}_{ijk})^2\bar h_{kk}^{1^{\ast}}
-\bar H\sum_{i,j,k}(\bar h^{1*}_{ijk})^2\bar h_{kk}^{1^{\ast}}\\
&=\sum_{i,j,k,l,p^{\ast}}(\bar h^{p^{\ast}}_{ijkl})^2+2\sum_{i,j,k,p^{\ast}}(\bar h^{p^{\ast}}_{ijk})^2{\color{red}\geq} 0.
\end{aligned}
\end{equation}
Hence, we conclude that (4.26) is in contradiction to  (4.27). It  completes the proof of Proposition 4.2.
\end{proof}

\vskip2mm
\noindent
In the step 3, we prove the following:
\begin{proposition} Under the assumptions of Proposition {\rm  4.1},
the formula {\rm (4.20)} in Proposition {\rm 4.1}  does not occur either.
\end{proposition}

\noindent
In this case, we  have $\sum_{i,j,k,p}(\bar h^{p^{\ast}}_{ijk})^2=0$, $\bar \lambda=0$ and $\bar \lambda_1\bar \lambda_2\neq 0$.

\vskip2mm
\noindent
Since $\bar H=\bar \lambda_1+\bar \lambda_2$ and S=$\bar \lambda_1^2+3\bar \lambda_2^2$,
we get
 $$
\bar \lambda_1=\dfrac{3\bar H\pm \sqrt{4S-3\bar H^2}}4, \ \quad
\bar \lambda_2=\dfrac{\bar H\mp \sqrt{4S-3\bar H^2}}4.
$$

\begin{lemma}
Under the assumptions of Proposition { \rm 4.1}, if
\begin{equation}
\begin{cases}
\begin{aligned}
& \sum_{i,j,k,p}(\bar h^{p^{\ast}}_{ijk})^2=0, \  \  \bar \lambda=0,    \ \ \bar\lambda_1\bar\lambda_2\neq 0, \\
& S<\sup H^2=\bar H^2 \leq 3S-2 \ \ \text{\rm and } \ \  S<\sup H^2<\dfrac43S
\end{aligned}
\end{cases}
\end{equation}
is satisfied, then
$$
\bar \lambda_1=\dfrac{3\bar H+ \sqrt{4S-3\bar H^2}}4, \ \quad
\bar \lambda_2=\dfrac{\bar H- \sqrt{4S-3\bar H^2}}4.
$$ do  not occur.
\end{lemma}
\begin{proof}
If
$$
\bar \lambda_1=\dfrac{3\bar H+ \sqrt{4S-3\bar H^2}}4, \ \quad
\bar \lambda_2=\dfrac{\bar H-\sqrt{4S-3\bar H^2}}4
$$
hold,
we have
\begin{equation}
\bar \lambda_1^2+\bar \lambda_2^2=\dfrac{\bar H^2+2S+ \sqrt{(4S-3\bar H^2)\bar H^2}}4.
\end{equation}
Due to $\bar \lambda_1\neq 3\bar \lambda_2$, we know $\bar H^2<\dfrac43S
$
and
$\dfrac{4S}3 \leq 3S-2$ if and only if $S\geq \dfrac65$.
According to (2.25) and  $\sum_{i,j,k,p}(\bar h^{p^{\ast}}_{ijk})^2=0$, we have
\begin{equation}
\begin{aligned}
S(1-\dfrac12S)-(S-\bar H^2)^2+\dfrac12(\bar  H^2)^2-\bar H^2(\bar \lambda_1^2+\bar\lambda_2^2)=0.
\end{aligned}
\end{equation}
We get from (4.29)
\begin{equation}
\begin{aligned}
S(1-\dfrac32S)+\dfrac32S\bar H^2-\dfrac34\bar  H^4-\dfrac{\bar H^2\sqrt{(4S-3\bar H^2)\bar H^2}}4=0.
\end{aligned}
\end{equation}
We consider function
$$
f(x)=S(1-\dfrac32S)+\dfrac32Sx-\dfrac34x^2-\dfrac{x\sqrt{(4S-3x)x}}4
$$
for $S<x< \dfrac43S$.
We know that
$$
f(S)=S(1-S)<0
$$
since $1<S<2$,
\begin{equation}
\begin{aligned}
f^{\prime}(x)=\dfrac{df(x)}{dx}&=\dfrac32S-\dfrac32x-\dfrac{\sqrt{(4S-3x)x}}4-\dfrac {x(2S-3x)}{4\sqrt{(4S-3x)x}}\\
&=\dfrac32(S-x)-\dfrac {3x(S-x)}{2\sqrt{(4S-3x)x}}\\
&=\dfrac32(S-x)\bigl(1-\dfrac {x}{\sqrt{(4S-3x)x}}\bigl)>0
\end{aligned}
\end{equation}
since
$S<x$  and  $x>\sqrt{(4S-3x)x}$.
Thus,  $f(x)$ is an increasing function of $x$.
\newline
If $S\geq \dfrac65$, then $\dfrac43S\leq 3S-2$. Hence, we have $S<\bar H^2< \dfrac43S$. \newline
Since
$$
f(\dfrac43S)=S(1-\dfrac32S)+\dfrac32S\dfrac43S-\dfrac34(\dfrac43S)^2=S(1-\dfrac56S) \leq 0,
$$
we conclude  $f(x)<0$ for any $x \in (S, \dfrac43 S)$,
which  is in contradiction to  (4.31).
Thus, we must have $S<\dfrac65$. In this case,
$\dfrac{4S}3 >3S-2$ and
\begin{equation*}
\begin{aligned}
f(3S-2)
&=S(1-\dfrac32S)+\dfrac32S(3S-2)-\dfrac34(3S-2)^2\\
&\ \ -\dfrac{(3S-2)\sqrt{(4S-3(3S-2))(3S-2)}}4\\
&=(3S-2)\biggl(\dfrac32-\dfrac{5}4S-\dfrac{\sqrt{(6-5S)(3S-2)}}4\biggl)\\
&=(3S-2)\dfrac{\sqrt{6-5S}}4\bigl(\sqrt{6-5S}-\sqrt{3S-2}\bigl)<0.
\end{aligned}
\end{equation*}
Therefore, it is also impossible. It finishes the proof of Lemma 4.2.

\end{proof}

\begin{lemma}
Under the assumptions of Proposition {\rm  4.1}, if \begin{equation}
\begin{cases}
\begin{aligned}
& \sum_{i,j,k,p}(\bar h^{p^{\ast}}_{ijk})^2=0, \ \  \bar\lambda=0,   \ \ \bar\lambda_1\bar\lambda_2\neq 0, \\
& S<\sup H^2=\bar H^2 \leq 3S-2 \ \ \text{\rm and } \ \  S<\sup H^2<\dfrac43S,
\end{aligned}
\end{cases}
\end{equation}
is  satisfied, then we have
$$
\bar \lambda_1=\dfrac{3\bar H- \sqrt{4S-3\bar H^2}}4, \ \quad
\bar \lambda_2=\dfrac{\bar H+\sqrt{4S-3\bar H^2}}4.
$$
and  $S\geq \dfrac65$.
\end{lemma}
\begin{proof}
According to Lemma 4.2, we must have
$$
\bar \lambda_1=\dfrac{3\bar H-\sqrt{4S-3\bar H^2}}4, \ \quad
\bar \lambda_2=\dfrac{\bar H+\sqrt{4S-3\bar H^2}}4.
$$
Thus,
\begin{equation}
\bar \lambda_1^2+\bar \lambda_2^2=\dfrac{\bar H^2+2S- \sqrt{(4S-3\bar H^2)\bar H^2}}4.
\end{equation}
If $S< \dfrac65$ holds,  then we get
$\bar \lambda_1\neq 3\bar \lambda_2$ and
$\dfrac{4S}3 > 3S-2$.
According to (2.25), we have
\begin{equation*}
\begin{aligned}
S(1-\dfrac12S)-(S-\bar H^2)^2+\dfrac12(\bar  H^2)^2-\bar H^2(\bar \lambda_1^2+\bar\lambda_2^2)=0.
\end{aligned}
\end{equation*}
we obtain from (4.34)
\begin{equation}
\begin{aligned}
S(1-\dfrac32S)+\dfrac32S\bar H^2-\dfrac34\bar  H^4+\dfrac{\bar H^2\sqrt{(4S-3\bar H^2)\bar H^2}}4=0.
\end{aligned}
\end{equation}
Now we consider function
$$
f_1(x)=S(1-\dfrac32S)+\dfrac32Sx-\dfrac34x^2+\dfrac{x\sqrt{(4S-3x)x}}4
$$
for $S<x\leq 3S-2$.
Since
%$$
%f_1(S)=S(1- \dfrac S2)>0
%$$
%for  $1<S<2$ and
\begin{equation*}
\begin{aligned}
f_1^{\prime}(x)=\dfrac{df_1(x)}{dx}&=\dfrac32S-\dfrac32x+\dfrac{\sqrt{(4S-3x)x}}4+\dfrac {x(2S-3x)}{4\sqrt{(4S-3x)x}}\\
&=\dfrac32(S-x)+\dfrac {3x(S-x)}{2\sqrt{(4S-3x)x}}\\
&=\dfrac32(S-x)\bigl(1+\dfrac {x}{\sqrt{(4S-3x)x}}\bigl)<0
\end{aligned}
\end{equation*}
for
$S<x$,
 $f_1(x)$ is a decreasing function of $x$  on $(S,3S-2)$.
\begin{equation*}
\begin{aligned}
f_1(3S-2)
&=S(1-\dfrac32S)+\dfrac32S(3S-2)-\dfrac34(3S-2)^2\\
&\ \ +\dfrac{(3S-2)\sqrt{(4S-3(3S-2))(3S-2)}}4\\
&=(3S-2)\biggl(\dfrac32-\dfrac{5}4S+\dfrac{\sqrt{(6-5S)(3S-2)}}4\biggl)\\
&=(3S-2)\dfrac{\sqrt{6-5S}}4\bigl(\sqrt{6-5S}+\sqrt{3S-2}\bigl)>0
\end{aligned}
\end{equation*}
since $S< \dfrac65$.
Thus $f_1(x)>0$ for any $x \in (S, 3S-2]$,
which is in contradiction to (4.35).
\end{proof}
\begin{lemma}
Under the assumptions of Proposition {\rm  4.1}, if
\begin{equation}
\begin{cases}
\begin{aligned}
& \sum_{i,j,k,p}(\bar h^{p^{\ast}}_{ijk})^2=0, \ \  \bar \lambda=0,    \ \ \bar\lambda_1\bar\lambda_2\neq 0, \\
& S<\sup H^2=\bar H^2 \leq 3S-2 \ \ \text{\rm and } \ \  S<\sup H^2<\dfrac43S,
\end{aligned}
\end{cases}
\end{equation}
are satisfied, then we have
  $1.89\leq S<2$.
\end{lemma}
\begin{proof}
According to Lemma 4.2 and Lemma 4.3, we
know  $2>S\geq \dfrac65$ and
$$
\bar \lambda_1=\dfrac{3\bar H- \sqrt{4S-3\bar H^2}}4, \ \quad
\bar \lambda_2=\dfrac{\bar H+\sqrt{4S-3\bar H^2}}4.
$$
In this case, $\dfrac{4S}3\leq 3S-2$.
Hence, we  have
$$
 \dfrac65 \leq S<2, \ \ S<\bar H^2<\dfrac43S,  \  \  \sum_{i,j,k,p^{\ast}}(\bar h^{p^{\ast}}_{ijk})^2=0
$$
and
$$
\bar \lambda_1=\dfrac{3\bar H-\sqrt{4S-3\bar H^2}}4, \ \quad
\bar \lambda_2=\dfrac{\bar H+\sqrt{4S-3\bar H^2}}4.
$$
From  (2.35) of Lemma 2.3  in Section 2, we get
\begin{equation*}
\begin{aligned}
&\dfrac12\lim_{m\to \infty}\mathcal L\sum_{i,j,k,p}(h^{p^{\ast}}_{ijk})^2(p_m)\\
&=(\bar H^2-2S)\bar H^2+(\dfrac12\bar H^2+\dfrac S2+2)\bar H^2(\bar \lambda_1^2+\bar \lambda_2^2)\\
&\ \ -\dfrac{\bar H^2-S}2(\bar H^4+H^3\bar \lambda_1)-\bar H^2(\bar \lambda_1^2+\bar \lambda_2^2)^2
-\bar H^3(\bar \lambda_1^3+\bar \lambda_2^3).
\end{aligned}
\end{equation*}
Since
$$
H\lambda_1=\frac{H^2-S}2+\lambda_1^2+\lambda_2^2,  \  \  \lambda_1^3
+ \lambda_2^3=H(\lambda_1^2-\lambda_1\lambda_2+\lambda_2^2),
$$
we have
\begin{equation*}
\begin{aligned}
&\dfrac12\lim_{m\to \infty}\mathcal L\sum_{i,j,k,p} (h^{p^{\ast}}_{ijk})^2(p_m)\\
&=\bar H^2\bigl[\bar H^2-2S-\dfrac{(\bar H^2-S)^2}4-\dfrac{\bar H^2(\bar H^2-S)}2+\dfrac {\bar H^4}2 \bigl ]\\
&\ \ +\bar H^2\bigl(S-\dfrac{3\bar H^2}2
+2-\bar \lambda_1^2-\bar\lambda_2^2\bigl)(\bar \lambda_1^2+\bar \lambda_2^2).
\end{aligned}
\end{equation*}
By making use of
$$
\bar \lambda_1^2+\bar \lambda_2^2=\dfrac{\bar H^2+2S- \sqrt{(4S-3\bar H^2)\bar H^2}}4,
$$
we obtain
\begin{equation}
\begin{aligned}
&\dfrac12\lim_{m\to \infty}\mathcal L\sum_{i,j,k,p} ( h^{p^{\ast}}_{ijk})^2(p_m)\\
&=\bar H^2\biggl[-\dfrac{\bar H^4}2+\dfrac{3\bar H^2}2 -S
+\dfrac{(\bar H^2-1)\sqrt{(4S-3\bar H^2)\bar H^2}}2 \biggl ].\\
\end{aligned}
\end{equation}
On the other hand,  from (2.34) and (2.35)
\begin{equation*}
\begin{aligned}
&\dfrac12\lim_{m\to \infty}\mathcal L\sum_{i,j,k,p} (h^{p^{\ast}}_{ijk})^2(p_m)
=\lim_{m\to \infty}\sum_{i,j,k,l,p}(\bar h^{p^{\ast}}_{ijkl})^2(p_m)\\
&\geq2(\bar h_{1122}^{1^{\ast}}-\bar h_{2211}^{1^{\ast}})^{2}+2(\bar h_{1122}^{2^{\ast}}-\bar h_{2211}^{2^{\ast}})^{2}
+\frac{1}{2}(\bar h_{2222}^{1^{\ast}}-\bar h_{2221}^{2^{\ast}})^{2}. \\
\end{aligned}
\end{equation*}
From Gauss equation and Ricci identities, we have
\begin{equation*}
\aligned
h_{2222}^{1^{\ast}}-h_{2221}^{2^{\ast}}
=&h_{2212}^{2^{\ast}}-h_{2221}^{2^{\ast}}\\
=&\sum_mh_{m2}^{2^{\ast}}R_{m212}+\sum_mh_{2m}^{2^{\ast}}R_{m212}+\sum_mh_{22}^{m^{\ast}}R_{m212}\\
=&(h_{12}^{2^{\ast}}+h_{21}^{2^{\ast}}+h_{22}^{1^{\ast}})R_{1212}\\
=&3\lambda_{2}K(\delta_{11}\delta_{22}-\delta_{12}\delta_{21})\\
=&3\lambda_{2}K,
\endaligned
\end{equation*}
\begin{equation*}
\begin{aligned}
h_{1122}^{1^{\ast}}-h_{2211}^{1^{\ast}}
=&h_{1212}^{1^{\ast}}-h_{1221}^{1^{\ast}}\\
=&\sum_mh_{m2}^{1^{\ast}}R_{m112}+\sum_mh_{1m}^{1^{\ast}}R_{m212}+\sum_mh_{12}^{m^{\ast}}R_{m112}\\
=&\lambda_{2}R_{2112}+\lambda_{1}R_{1212}+\lambda_{2}R_{2112}\\
=&(\lambda_{1}-2\lambda_{2})K.
\end{aligned}
\end{equation*}
From the above equations, we obtain
\begin{equation*}
\begin{aligned}
&2(h_{1122}^{1^{\ast}}-h_{2211}^{1^{\ast}})^{2}+2(h_{1122}^{2^{\ast}}-h_{2211}^{2^{\ast}})^{2}+\frac{1}{2}(h_{2222}^{1^{\ast}}-h_{2221}^{2^{\ast}})^{2} \\
=&2(\lambda_{1}-2\lambda_{2})^{2}K^{2}+\frac{1}{2}(3\lambda_{2}K)^{2}\\
=&K^{2}[2\lambda_{1}^{2}-8\lambda_{1}\lambda_{2}+\frac{25}{2}\lambda_{2}^{2}]\\
=&K^{2}[2S-4(H^{2}-\lambda_{1}^{2}-\lambda_{2}^{2})+\frac{13}{2}\lambda_{2}^{2}]\\
=&K^{2}(6S-4H^{2}-\frac{3}{2}\lambda_{2}^{2}).
\end{aligned}
\end{equation*}
Thus, we have
\begin{equation*}
\begin{aligned}
&\dfrac12\lim_{m\to \infty}\mathcal L\sum_{i,j,k,p^{\ast}} (h^{p^{\ast}}_{ijk})^2(p_m)
=\lim_{m\to \infty}\sum_{i,j,k,l,p^{\ast}}(\bar h^{p^{\ast}}_{ijkl})^2(p_m)\\
&\geq2(\bar h_{1122}^{1^{\ast}}-\bar h_{2211}^{1^{\ast}})^{2}+2(\bar h_{1122}^{2^{\ast}}-\bar h_{2211}^{2^{\ast}})^{2}
+\frac{1}{2}(\bar h_{2222}^{1^{\ast}}-\bar h_{2221}^{2^{\ast}})^{2} \\
&\geq \dfrac{(\bar H^2-S)^2}4\bigl(6S-4\bar H^2-\dfrac{3\bar \lambda_2^2}2 \bigl )\\
&=\dfrac{(\bar H^2-S)^2}4\bigl[(6-\dfrac38)S-(4-\dfrac3{16})\bar H^2-\dfrac{3\sqrt{(4S-3\bar H^2)\bar H^2}}{16} \bigl ]
\end{aligned}
\end{equation*}
Hence, we obtain,  in view of (4.37) and (4.38)
\begin{equation*}
\begin{aligned}
&\bar H^2\bigl[-\dfrac{\bar H^4}2+\dfrac{3\bar H^2}2 -S
+\dfrac{(\bar H^2-1)\sqrt{(4S-3\bar H^2)\bar H^2}}2 \bigl ]\\
&\geq \dfrac{(\bar H^2-S)^2}4\bigl[(6-\dfrac38)S-(4-\dfrac3{16})\bar H^2
-\dfrac{3\sqrt{(4S-3\bar H^2)\bar H^2}}{16} \bigl ].
\end{aligned}
\end{equation*}
From (2.25) and $\sum_{i,j,k,p}(\bar h^{p^{\ast}}_{ijk})^2=0$, we know
$$
\dfrac{\bar H^2\sqrt{(4S-3\bar H^2)\bar H^2}}4=-S(1-\dfrac32S)-\dfrac32S\bar H^2+\dfrac34\bar  H^4.
$$
Therefore, we conclude
\begin{equation}
\begin{aligned}
&\bar H^6-3\bar H^4S
+3\bar H^2S^2-3S^2+2S\\
&\geq  \dfrac{(\bar H^2-S)^2}4\bigl[(6+\dfrac34)S-(4+\dfrac3{8})\bar H^2
-\dfrac{3S(3S-2)}{8\bar H^2} \bigl ].
\end{aligned}
\end{equation}
Since
$-\dfrac3{4}x-\dfrac{3S(3S-2)}{8x} $ is a decreasing function of $x$,  for $S<x<\dfrac{4S}3$,
we have
$$
-\dfrac3{4}\bar H^2-\dfrac{3S(3S-2)}{8\bar H^2} > -S-\dfrac{9(3S-2)}{32}.
$$
Hence, we get
\begin{equation}
\begin{aligned}
&\bar H^6-3\bar H^4S
+3\bar H^2S^2-3S^2+2S\\
&\geq  \dfrac{(\bar H^2-S)^2}4\bigl[(5-\dfrac{3}{32})S-(4-\dfrac3{8})\bar H^2+\dfrac9{16}\bigl ].
\end{aligned}
\end{equation}
We consider a function $g=g(x)$ of $x$ defined by
$$
g(x)=x^3-3x^2S
+3xS^2-3S^2+2S- \dfrac{(x-S)^2}4\bigl[(5-\dfrac{3}{32})S-(4-\dfrac3{8})x+\dfrac9{16}
 \bigl ].
$$

\begin{equation*}
\begin{aligned}
g^{\prime}(x)&=3x^2-6xS
+3S^2+\dfrac{(x-S)^2}4(4-\dfrac3{8})\\
&- \dfrac{(x-S)}2\bigl[(5-\dfrac{3}{32})S-(4-\dfrac3{8})x+\dfrac9{16}\bigl]\\
&=(x-S)\bigl[(6-\dfrac{9}{32})x-(6+\dfrac{23}{64})S-\dfrac9{32}\bigl].
\end{aligned}
\end{equation*}
Hence, $g(x)$ attains its minimum at $(6-\dfrac{9}{32})x-(6+\dfrac{23}{64})S-\dfrac9{32}=0$.
$$
g(S)=S(S-1)(S-2)<0, \ \ g(\dfrac{4S}3)= (1+\dfrac{121}{36\cdot96})S^3-\dfrac{193}{64}S^2+2S<0
$$
if $\dfrac65\leq S<1.89$. We have $g(x)<0$ for $S<x<\dfrac{4S}3$, which is in contradiction to (4.39).
Hence, $S$ satisfies
$$
1.89\leq S<2.
$$

\end{proof}

\vskip3mm
\noindent
\begin{lemma}
Under the assumptions of Proposition {\rm  4.1}, if
 \begin{equation*}
\begin{cases}
\begin{aligned}
& \sum_{i,j,k,p}(\bar h^{p^{\ast}}_{ijk})^2=0, \ \  \bar \lambda=0,    \ \ \bar\lambda_1\bar\lambda_2\neq 0, \\
& S<\sup H^2=\bar H^2 \leq 3S-2 \ \ \text{\rm and } \ \  S<\sup H^2<\dfrac43S,
\end{aligned}
\end{cases}
\end{equation*}
is satisfied, then we have
$$
S<\bar H^2\leq S+\dfrac15S.
$$
\end{lemma}
\begin{proof}
Since $\bar H^2=\sup H^2$, if $S+\dfrac S5<\bar H^2<S+\dfrac S3$,
we consider a function $f_2=f_2(x)$ of $x$ defined by
$$
f_2(x)=S(1-\dfrac32S)+\dfrac32Sx-\dfrac34x^2+\dfrac{x\sqrt{(4S-3x)x}}4
$$
for $\dfrac65S<x\leq \dfrac43S$.
We know that
$$
f_2(\dfrac{6S}5)=S(1- \dfrac {39-6\sqrt3}{50}S)<0
$$
since $1.89<S<2$,
\begin{equation*}
\begin{aligned}
f_2^{\prime}(x)=\dfrac{df(x)}{dx}&=\dfrac32S-\dfrac32x+\dfrac{\sqrt{(4S-3x)x}}4+\dfrac {x(2S-3x)}{4\sqrt{(4S-3x)x}}\\
&=\dfrac32(S-x)+\dfrac {3x(S-x)}{2\sqrt{(4S-3x)x}}\\
&=\dfrac32(S-x)\bigl(1+\dfrac {x}{\sqrt{(4S-3x)x}}\bigl)<0.
\end{aligned}
\end{equation*}
$f_2(x)$ is a decreasing function of $x$ and we can not have
\begin{equation*}
\begin{aligned}
S(1-\dfrac32S)+\dfrac32S\bar H^2-\dfrac34\bar  H^4+\dfrac{\bar H^2\sqrt{(4S-3\bar H^2)\bar H^2}}4=0.
\end{aligned}
\end{equation*}
Hence, we must have
$$
S<\bar H^2<\dfrac65S.
$$
\end{proof}

\vskip2mm
\noindent
{\it Proof of Proposition 4.3}. According to Lemma 4.4 and Lemma 4.5, we have
\begin{equation*}
\begin{cases}
\begin{aligned}
& \sum_{i,j,k,p}(\bar h^{p^{\ast}}_{ijk})^2=0, \ \  \bar \lambda=0,  \ \ \bar\lambda_1\bar\lambda_2\neq 0, \\
&  S<\sup H^2<\dfrac65S, \ \  1.89<S<2.
\end{aligned}
\end{cases}
\end{equation*}
We obtain from (4.38)
\begin{equation*}
\begin{aligned}
&\bar H^6-3\bar H^4S
+3\bar H^2S^2-3S^2+2S\\
&\geq  \dfrac{(\bar H^2-S)^2}4\bigl[(6+\dfrac34)S-(4+\dfrac3{8})\bar H^2
-\dfrac{3S(3S-2)}{8\bar H^2} \bigl ].
\end{aligned}
\end{equation*}
Since
$-\dfrac3{4}\bar H^2-\dfrac{3S(3S-2)}{8\bar H^2} $ is a decreasing function of $\bar H^2$,  for $S<\bar H^2<\dfrac{6S}5$,
we have
$$
-\dfrac3{4}\bar H^2-\dfrac{3S(3S-2)}{8\bar H^2} > -\dfrac{9}{10}S-\dfrac{5(3S-2)}{16}.
$$
\begin{equation}
\begin{aligned}
&\bar H^6-3\bar H^4S
+3\bar H^2S^2-3S^2+2S\\
&\geq  \dfrac{(\bar H^2-S)^2}4\bigl[(5-\dfrac{7}{80})S-(4-\dfrac3{8})\bar H^2+\dfrac5{8}
 \bigl ].
\end{aligned}
\end{equation}
We consider function
$$
f_3(x)=x^3-3x^2S
+3xS^2-3S^2+2S- \dfrac{(x-S)^2}4\bigl[(5-\dfrac{7}{80})S-(4-\dfrac3{8})x+\dfrac5{8}
 \bigl ].
$$

\begin{equation*}
\begin{aligned}
f_3^{\prime}(x)&=3x^2-6xS
+3S^2+\dfrac{(x-S)^2}4(4-\dfrac3{8})\\
&- \dfrac{(x-S)}2\bigl[(5-\dfrac{7}{80})S-(4-\dfrac3{8})x+\dfrac5{8}\bigl]\\
&=(x-S)\bigl[(6-\dfrac{9}{32})x-(6+\dfrac{29}{80})S-\dfrac5{16}\bigl].
\end{aligned}
\end{equation*}
Hence, $f_3(x)$ attains its minimum at $(6-\dfrac{9}{32})x-(6+\dfrac{29}{80})S-\dfrac5{16}=0$.
$$
f_3(S)=S(S-1)(S-2)<0, \ \  f_3(\dfrac{6S}5)= (1+\dfrac1{125}-\dfrac{9}{1600})S^3-(3+\dfrac{1}{160})S^2+2S<0
$$
if $\dfrac65\leq S<2$.  This is in contradiction to (4.40).
Therefore, we conclude that the formula (4.18) in Proposition 4.1 does not occur either.

\begin{flushright}
$\square$
\end{flushright}

\vskip2mm
\noindent
{\it Proof of Theorem 4.2}.  According to Propositions 4.1, Propositions 4.2 and Propositions 4.3, we know that there are no 2-dimensional
complete Lagrangian self-shrinkers $X: M^2\to \mathbb{R}^{4}$ with constant squared norm $S$
of the  second fundamental form  and $1<S<2$.

\begin{flushright}
$\square$
\end{flushright}

\begin{theorem}
 Let $X: M^2\to \mathbb{R}^{4}$ be a
2-dimensional   Lagrangian self-shrinker in $\mathbb R^4$.
If   $S\equiv 2$  or $S\equiv 1$, then  the mean curvature  $H$ satisfies $H\neq 0$ on  $ M^2$.
\end{theorem}
\begin{proof} If there exists a point $p \in M^2$ such that $H=0$ at $p$, then we know
$H^{1^*}=H^{2^*}=0$.
Thus, at $p$, we have
$$
H=0, \ \ H^{1^*}=\lambda_1+\lambda_2=0, \ \  \lambda=h_{12}^{1^*}=-h_{22}^{2^*}.
$$
From
$$
H^{p^*}_{,i}=\sum_{k}h_{ik}^{p^*}\langle X, e_k\rangle, \ \ \text{\rm for} \ \ i, p=1, 2,
$$
we have
\begin{equation*}
\begin{aligned}
&h_{111}^{1^*}+h_{221}^{1^*}=H^{1^*}_{,1}=\lambda_1\langle X, e_1\rangle+\lambda\langle X, e_2\rangle,\\
&h_{112}^{1^*}+h_{222}^{1^*}=H^{1^*}_{,2}=\lambda\langle X, e_1\rangle-\lambda_1\langle X, e_2\rangle,\\
&h_{112}^{2^*}+h_{222}^{2^*}=H^{2^*}_{,2}=-\lambda_1\langle X, e_1\rangle-\lambda\langle X, e_2\rangle
\end{aligned}
\end{equation*}
and
$$
0=\dfrac12\nabla_i S=\lambda_1(h_{11i}^{1^*}-3h_{22i}^{1^*})+\lambda(3h_{11i}^{2^*}-h^{2^{\ast}}_{22i}), \ \ \text{for } \ i=1, 2,
$$
it means that,
$$
\lambda_1(h_{11i}^{1^*}-3h_{22i}^{1^*})+\lambda(3h_{11i}^{2^*}-h^{2^{\ast}}_{22i})=0, \ \ \text{for } \ i=1, 2
$$
since $S$ is constant.
Thus, we get a system of linear equations
\begin{equation}
\begin{pmatrix}
&- \lambda & -3 \lambda_1 &3 \lambda&  \lambda_1 &0\\
&0&-\lambda & -3\bar \lambda_1 &3 \lambda& \lambda_1 \\
&0&0&1&0&1\\
&0&1&0&1&0\\
&1&0&1&0&0
\end{pmatrix}
\begin{pmatrix}
h_{222}^{2^*}\\
h_{222}^{1^*}\\
h_{122}^{1^*}\\
h_{112}^{1^*}\\
h_{111}^{1^*}
\end{pmatrix}
=
\begin{pmatrix}
0\\
0\\
\lambda_1\langle X, e_1\rangle+\lambda\langle X, e_2\rangle\\
\lambda\langle X, e_1\rangle-\lambda_1\langle X, e_2\rangle\\
-\lambda_1\langle X, e_1\rangle-\lambda\langle X, e_2\rangle
\end{pmatrix}.
\end{equation}
From $S=4\lambda^2+4\lambda_1^2$, we know  $\lambda_1=-\lambda_2$ and $2\lambda_1^2+2\lambda^2=\dfrac12S\neq 0$.
Hence, by solving the above  system, we get
\begin{equation*}
\begin{aligned}
h_{222}^{2^*}=&\frac{-5\lambda_1}4\langle X, e_1\rangle-\frac{3\lambda}4\langle X, e_2\rangle, \\
h_{222}^{1^*}=&\frac{3\lambda}4\langle X, e_1\rangle-\frac{\lambda_1}4\langle X, e_2\rangle, \\
h_{221}^{1^*}=&\frac{\lambda_1}4\langle X, e_1\rangle-\frac{\lambda}4\langle X, e_2\rangle, \\
h_{211}^{1^*}=&\frac{\lambda}4\langle X, e_1\rangle-\frac{3\lambda_1}4\langle X, e_2\rangle, \\
h_{111}^{1^*}=&\frac{3\lambda_1}4\langle X, e_1\rangle+\frac{5\lambda}4\langle X, e_2\rangle. \\
\end{aligned}
\end{equation*}

\noindent  With a direct calculation, we obtain
\begin{equation*}
\begin{aligned}
|\nabla^{\perp} \vec H|^2&=\sum_{i,p}(H^{p^*}_{,i})^2\\
&=\sum_i(h_{11i}^{1^*}+h_{22i}^{1^*})^2+\sum_i(h_{11i}^{2^*}+h_{22i}^{2^*})^2\\
&=2(\lambda^2+\lambda_1^2)\langle X, e_1\rangle^2+2(\lambda^2+\lambda_1^2)\langle X, e_2\rangle^2
=2(\lambda^2+\lambda_1^2) |X|^2\\
\end{aligned}
\end{equation*}
and
\begin{equation}
\begin{aligned}
&\sum_{i,j,k,p}(h^{p^*}_{ijk})^2\\
&=\sum_p(h_{111}^{p^*})^2+\sum_p(h_{222}^{p^*})^2+3\sum_p(h_{221}^{p^*})^2
+3\sum_p(h_{112}^{p^*})^2\\
&=\sum_{p}(h_{11p}^{1^*})^2+(h_{222}^{1^*})^2+(h_{222}^{2^*})^2
+3\sum_p(h_{22p}^{1^*})^2+3(h_{112}^{1^*})^2+3(h_{112}^{2^*})^2\\
&=\dfrac{5}{2}\bigl\{(\lambda_1^2+\lambda^2)\langle X, e_1\rangle^2+(\lambda_1^2+\lambda^2)\langle X, e_2\rangle^2\bigl\}\\
&=\dfrac{5}{2}(\lambda_1^2+\lambda^2)|X|^2=\dfrac54|\nabla^{\perp} \vec H|^2.
\end{aligned}
\end{equation}
From the Ricci identity (2.8), we have
$$
h^{2^*}_{1122}-h^{2^*}_{2211}=3\lambda K,
$$
$$
h^{1^*}_{1112}-h^{1^*}_{1121}=-3\lambda K,
$$
$$
h^{2^*}_{1112}-h^{2^*}_{1121}=3\lambda_{1}K,
$$
$$
h^{2^*}_{2212}-h^{2^*}_{2221}=-3\lambda_{1}K.
$$
Thus, we obtain
\begin{equation}
\begin{aligned}
&\sum_{i,j,k,l,p}(h_{ijkl}^{p^{\ast}})^{2}\\
=&4(h_{1122}^{1^{\ast}})^{2}+6(h_{2211}^{1^{\ast}})^{2}+6(h_{1122}^{2^{\ast}})^{2}+4(h_{2211}^{2^{\ast}})^{2}
+4(h_{2222}^{1^{\ast}})^{2}+4(h_{1111}^{2^{\ast}})^{2}\\
&+(h_{1111}^{1^{\ast}})^{2}+(h_{2222}^{2^{\ast}})^{2}+(h_{1112}^{1^{\ast}})^{2}+(h_{2221}^{2^{\ast}})^{2}\\
=&2(h_{1122}^{1^{\ast}}-h_{2211}^{1^{\ast}})^{2}+2(h_{1122}^{1^{\ast}}+h_{2211}^{1^{\ast}})^{2}\\
&+2(h_{1122}^{2^{\ast}}-h_{2211}^{2^{\ast}})^{2}+2(h_{1122}^{2^{\ast}}+h_{2211}^{2^{\ast}})^{2}\\
&+\frac{1}{2}(h_{2222}^{1^{\ast}}-h_{2221}^{2^{\ast}})^{2}+\frac{1}{2}(h_{2222}^{1^{\ast}}+h_{2221}^{2^{\ast}})^{2}\\
&+\dfrac12(h_{1121}^{1^{\ast}}+h_{1112}^{1^{\ast}})^{2}+\dfrac12(h_{1121}^{1^{\ast}}-h_{1112}^{1^{\ast}})^{2}\\
&+\dfrac12(h_{1111}^{1^{\ast}}+h_{2211}^{1^{\ast}})^{2}+\dfrac12(h_{1111}^{1^{\ast}}-h_{2211}^{1^{\ast}})^{2}\\
&+\frac{1}{2}(h_{2222}^{2^{\ast}}-h_{1122}^{2^{\ast}})^{2}+\frac{1}{2}(h_{2222}^{2^{\ast}}+h_{1122}^{2^{\ast}})^{2}\\
&+\dfrac12(h_{2211}^{1^{\ast}}-h_{2222}^{1^{\ast}})^{2}+\dfrac12(h_{2211}^{1^{\ast}}+h_{2222}^{1^{\ast}})^{2}\\
&+(h_{1122}^{2^{\ast}})^2+3(h_{1111}^{2^{\ast}})^2+2(h_{2222}^{1^{\ast}})^{2} \\
&\geq  18\lambda_1^2K^2+\dfrac92\lambda_1^2K^2+18\lambda^2K^2+\dfrac92\lambda^2K^2\\
&= \dfrac{45}2(\lambda_1^2+\lambda^2)K^2\\
&=\dfrac{45}{32}S^3
\end{aligned}
\end{equation}
because  of $S=4(\lambda_1^2+\lambda^2)$ and $K=-\dfrac12S$. \newline
On the other hand, since $S$ is constant and $H=0$ at  $p$, from (2.35) in Lemma 2.3 and (4.42), we obtain, at  $p$,
\begin{equation}
\aligned
\frac{1}{2}\mathcal{L}\sum_{i,j,k,p}(h_{ijk}^{p^{\ast}})^{2}
=-\dfrac{3S}2|\nabla^{\perp} {\vec H}|^{2}=-\dfrac{3}{5}S^2(3S-2).
\endaligned
\end{equation}
According to  $h_{11k}^{p^*}+h_{22k}^{p^*}=H^{p^*}_{,k}$ and $H^{p^*}=0$ for $p, k=1, 2$, by a direct calculation, we have
 $$
 \sum_{i,j,k,l,p,q}h_{il}^{p^{\ast}}h_{ijk}^{p^{\ast}}h_{jl}^{q^{\ast}}H_{,k}^{q^{\ast}}=\dfrac{S}4|\nabla^{\perp} {\vec H}|^{2}.
 $$
From  (2.34) in Lemma 2.3, we get
\begin{equation*}
\aligned
&\frac{1}{2}\mathcal{L}\sum_{i,j,k,p}(h_{ijk}^{p^{\ast}})^{2}\\
&=\sum_{i,j,k,l,p}(h_{ijkl}^{p^{\ast}})^{2}-(5S-2)\sum_{i,j,k,p}(h_{ijk}^{p^{\ast}})^{2}+\dfrac{5S}2|\nabla^{\perp} {\vec H}|^{2}\\
 &\ \  -\sum_{i,j,k,l,p,q}h_{il}^{p^{\ast}}h_{ijk}^{p^{\ast}}h_{jl}^{q^{\ast}}H_{,k}^{q^{\ast}}\\
&=\sum_{i,j,k,l,p}(h_{ijkl}^{p^{\ast}})^{2}-(5S-2)\sum_{i,j,k,p}(h_{ijk}^{p^{\ast}})^{2}
+\dfrac{9S}{4}|\nabla^{\perp} {\vec H}|^{2}\\
&=\sum_{i,j,k,l,p}(h_{ijkl}^{p^{\ast}})^{2}-(5S-2)\dfrac{S(3S-2)}2+\dfrac9{10}S^2(3S-2).\\
\endaligned
\end{equation*}
Thus, we have from (4.44)
\begin{equation*}
\aligned
\sum_{i,j,k,l,p}(h_{ijkl}^{p^{\ast}})^{2}-(5S-2)\dfrac{S(3S-2)}2+\dfrac9{10}S^2(3S-2)=-\dfrac{3}{5}S^2(3S-2),
\endaligned
\end{equation*}
namely,
\begin{equation*}
\aligned
\sum_{i,j,k,l,p}(h_{ijkl}^{p^{\ast}})^{2}=S(S-1)(3S-2),
\endaligned
\end{equation*}
which is in contradiction to (4.43) for  $S\equiv 2$ or  $S\equiv 1$. Hence, we conclude  that  $H\neq 0$ on $M^2$.
\end{proof}

\begin{proposition}
Let $X: M^2\to \mathbb{R}^{4}$ be a 2-dimensional complete  Lagrangian self-shrinker  in $\mathbb R^4$.
If  the squared norm $S$ of the second fundamental form satisfies $S\equiv 1$ or $S\equiv 2$, then
$\sup H^2=S$.
\end{proposition}
\begin{proof}
\noindent
In terms of Lemma 4.1,  there exists a sequence $\{p_m\}$ in $M^2$ such that
\begin{equation*}
\lim_{m\rightarrow\infty} H^2(p_m)=\sup H^2, \ \
\lim_{m\rightarrow\infty} |\nabla H^2(p_m)|=0, \ \
\limsup_{m\rightarrow\infty}\mathcal{L} H^2(p_m)\leq 0
\end{equation*}
\vskip2mm
\noindent
and
$$
\bar \lambda=0, \  \   \bar h^{1^{*}}_{ij}=\bar \lambda_i\delta_{ij}.
$$
\newline
(1) Case for $S\equiv 2$.
By making use of the same assertion as in the proof of Proposition 4.1, we have
 for $ k=1, 2$,
 \begin{equation}
\begin{cases}
&\bar h^{1^*}_{11k}+\bar h^{1^*}_{22k}=0,\\
&\bar h^{2^*}_{11k}+\bar h^{2^*}_{22k}=0,\\
&(\bar \lambda_1-3\bar \lambda_2)\bar h^{1^*}_{11k}=0
\end{cases}
\end{equation}
with $\bar\lambda_1\bar\lambda_2\neq 0$.
\vskip1mm
\noindent
If $\bar \lambda_1=3\bar \lambda_2$, we get
$$
\lim_{m\rightarrow\infty}H^2(p_m)=\bar H^2=(\bar \lambda_1+\bar \lambda_2)^2=16\bar \lambda_2^2=\dfrac{4S}3.
$$
By making use of the same assertion as in the proof of Proposition 4.2, we can know that  this is impossible.

\noindent
Thus, we get   $\bar \lambda_1\neq 3\bar \lambda_2$. In this case, we obtain
$\bar h^{p^{\ast}}_{ijk}=0$  for any $i, j, k, p$ from (4.45).  Hence, we have
from (2.25) in Lemma 2.1,
\begin{equation}
\begin{aligned}
&0=S(1-\dfrac32S)+2\bar H^2S-\dfrac12\bar H^4-\bar H^2(\bar \lambda_1^2+\bar\lambda_2^2)\\
&=-(S-\bar H^2)^2-\dfrac12\bar H^2(\bar \lambda_1-\bar\lambda_2)^2.
\end{aligned}
\end{equation}
We conclude
$$
\sup H^2=\bar H^2=S=2.
$$
\vskip2mm
\noindent
(2) Case for $S\equiv 1$.
Since $S=1$, we have $\sup H^2>0$. From $\lim_{m\rightarrow\infty} |\nabla H^2(p_m)|=0$ and $|\nabla H^2|^2=4\sum_i(\sum_{p^*}H^{p^{\ast}}H^{p^{\ast}}_{,i})^2$, we get
$$
\bar H^{1^*}_{,i}=0.
$$
From (4.7) and (4.10), we have
$$
\bar \lambda_i\lim_{m\rightarrow\infty} \langle X,e_i \rangle(p_m)=0.
$$
Next, we  take  the following three cases into consideration.

\vskip2mm
\noindent
(a)  If  $\bar \lambda_1=0$,
in this case, $\bar \lambda_2\neq0$, $3\bar H^2=S=1$. Since $\bar H^{1^*}_{,i}=0$ and $S=1$, we get
\begin{equation*}
\bar h_{11k}^{1^{*}}+\bar h_{22k}^{1^{*}}=0,\ \ \bar h_{22k}^{1^{*}}=0,\ \ k=1,2.
\end{equation*}
Therefore,
\begin{equation*}
\bar h_{111}^{1^{*}}=\bar h_{112}^{1^{*}}=\bar h_{122}^{1^{*}}=\bar h_{222}^{1^{*}}=0
\end{equation*}
and
\begin{equation*}
|\nabla^{\perp} \overrightarrow{\bar H}|^2=\sum_{i,j,k,p}(\bar h_{ijk}^{p^{*}})^2.
\end{equation*}
From $\limsup_{m\rightarrow\infty}\mathcal{L} |H|^2(p_m)\leq 0$ and (4.8), we obtain
\begin{equation*}
\frac{1}{2}(\bar H^2-1)^2\leq0,
\end{equation*}
it means that, $\bar H^2=1$. It is a contradiction.
\vskip2mm
\noindent
(b)  If $\bar \lambda_2=0$, in this case, $\bar \lambda_1\neq0$, $\sup H^2=\bar H^2=S=1$.
\vskip2mm
\noindent
(c) If $\bar \lambda_1\bar \lambda_2\neq0$, in this case,
 for $ k=1, 2$,
 \begin{equation*}
\begin{cases}
&\bar h^{1^*}_{11k}+\bar h^{1^*}_{22k}=0,\\
&\bar h^{2^*}_{11k}+\bar h^{2^*}_{22k}=0,\\
&(\bar \lambda_1-3\bar \lambda_2)\bar h^{1^*}_{11k}=0.
\end{cases}
\end{equation*}
\vskip2mm
\noindent
If  $\bar \lambda_1\neq3\bar \lambda_2$,  from the above equations, we know
\begin{equation*}
\sum_{i,j,k,p}(\bar h_{ijk}^{p^{*}})^2=0,\ \ i, j, k, p=1, 2.
\end{equation*}
From (4.8), we get
\begin{equation*}
0\geq\frac{1}{2}(\bar H^2-1)^2.
\end{equation*}
Hence,  we have
\begin{equation*}
\sup \bar H^2=1=S.
\end{equation*}
\vskip2mm
\noindent
If  $\bar \lambda_1=3\bar \lambda_2$, we have  $\bar H^2=\frac{4}{3}S=\frac{4}{3}$
and $1=S=\bar \lambda_1^2+3\bar \lambda_2^2=12\bar \lambda_2^2$. From (2.25), we get
\begin{equation*}
\sum_{i,j,k,p}(\bar h_{ijk}^{p^{*}})^2=-\frac{1}{6}<0.
\end{equation*}
It is impossible.
From the above arguments, we conclude that, for $S=2$ or $S=1$,
\begin{equation*}
\sup H^2=S.
\end{equation*}
\end{proof}
\vskip2mm
\noindent
\begin{theorem}
Let $X: M^2\to \mathbb{R}^{4}$ be a 2-dimensional complete  Lagrangian self-shrinker  in $\mathbb R^4$.
If  the squared norm $S$ of the second fundamental form satisfies $S\equiv 1$ or $S\equiv 2$, then
$H^2=S$ is constant.
\end{theorem}
\begin{proof}
We can apply the generalized maximum principle for $\mathcal L$-operator
to the function $-H^2$. Thus, there exists a sequence $\{p_m\}$ in $M^2$ such that
\begin{equation*}
\lim_{m\rightarrow\infty} H^2(p_m)=\inf H^2, \ \
\lim_{m\rightarrow\infty} |\nabla H^2(p_m)|=0, \ \
\liminf_{m\rightarrow\infty}\mathcal{L} H^2(p_m)\geq 0.
\end{equation*}
By making use of the similar  assertion as in the proof of Lemma 4.1, we have
\begin{equation}
\begin{cases}
\begin{aligned}
&\lim_{m\rightarrow\infty} H^2(p_m)=\inf H^2=\bar H^2,\quad
\lim_{m\rightarrow\infty} |\nabla H^2(p_m)|=0,\\
&
\lim_{m\rightarrow\infty} |\nabla ^{\perp}\vec H|^2(p_m)-\sum_{i,j,k,p}(\bar h^{p^{\ast}}_{ijk})^2+\dfrac12(\bar H^2-S)(\bar H^2-3S+2)\geq 0.
\end{aligned}
\end{cases}
\end{equation}
\vskip2mm
\noindent
By taking  the   limit and making use of the same assertion as in Theorem  4.3, we  can prove
$\inf H^2\neq 0$.
Hence, without loss of the generality,
at each point $p_m$, we choose $e_1$, $e_2$  such that
$$
\vec H=H^{1^{\ast}}e_{1^{\ast}}
$$
and we can assume
$$
\lim_{m\rightarrow\infty}h^{p^{\ast}}_{ijl}(p_m)=\bar h^{p^{\ast}}_{ijl}, \quad \lim_{m\rightarrow\infty}h^{p^{\ast}}_{ij}(p_m)=\bar h^{p^{\ast}}_{ij},
\quad \lim_{m\rightarrow\infty}h^{p^{\ast}}_{ijkl}(p_m)=\bar h^{p^{\ast}}_{ijkl},
$$
for $i, j, k, l,p=1, 2$.
From $\lim_{k\rightarrow\infty} |\nabla H^2(p_m)|=0$ and $|\nabla H^2|^2=4\sum_i(\sum_{p^*}H^{p^{\ast}}H^{p^{\ast}}_{,i})^2$, we have
\begin{equation}
\bar H^{1^*}_{,k}=0.
\end{equation}
From (4.7) and (4.48), we obtain
\begin{equation}
\begin{cases}
\begin{aligned}
&\bar \lambda_1\lim_{m\rightarrow\infty} \langle X,e_1\rangle(p_m)+\bar \lambda \lim_{m\rightarrow\infty} \langle X,e_2 \rangle(p_m)=0,\\
&\bar \lambda\lim_{m\rightarrow\infty} \langle X,e_1\rangle(p_m)
+\bar \lambda_2 \lim_{m\rightarrow\infty} \langle X,e_2 \rangle(p_m)=0.
\end{aligned}
\end{cases}
\end{equation}
If $\bar \lambda_1\bar \lambda_2\neq \bar \lambda^2$ and $\bar \lambda\neq0$,  we get
$$
\lim_{m\rightarrow\infty} \langle X,e_1\rangle(p_m)=
\lim_{m\rightarrow\infty} \langle X,e_2 \rangle(p_m)=0.
$$
Thus, we know,  for  $k=1, 2$,
\begin{equation}
\begin{cases}
\begin{aligned}
&\bar h^{1^*}_{11k}+\bar h^{1^*}_{22k}=0, \\
&\bar h^{2^*}_{11k}+\bar h^{2^*}_{22k}=0,\\
&\bar \lambda_1\bar h^{1^{\ast}}_{11k}+3\bar \lambda \bar h^{1^{\ast}}_{12k}
+3\bar \lambda_2 \bar h^{2^{\ast}}_{12k}-\bar \lambda\bar h^{2^{\ast}}_{22k}=0.
\end{aligned}
\end{cases}
\end{equation}
We  conclude,  for any $i, j, k, p$,
$$
\bar h^{p^*}_{ijk}=0.
$$

\noindent From  (4.47) and (2.25) in Lemma 2.1, we have

\begin{equation}
\begin{aligned}
&S\geq \inf H^2= \bar H^2,\\
&S(1-\dfrac12S)-(S-\bar H^2)^2+\dfrac12\bar  H^4-\bar H^2(\bar \lambda_1^2+\bar\lambda_2^2+2\bar \lambda^2)=0.
\end{aligned}
\end{equation}

\noindent From Lemma 2.5 and taking limit,
\begin{equation}
\aligned
0&\leq\sum_{i,j,k,l,p}(\bar h_{ijkl}^{p^{\ast}})^{2}
=\frac{1}{2}\lim_{m\rightarrow\infty} \mathcal{L}\sum_{i,j,k,p}( h_{ijk}^{p^{\ast}})^{2}(p_m)\\
&=\bar H^{2}\bigl[\bar H^{2}-2S+\frac{1}{2}\bar H^{4}-\bar K\bar H^{2}-\bar K^{2}\bigl]-\bar \lambda^2\bar H^{4}\\
&\ \ +\bar H^{2}(S+2-\frac{3}{2}\bar H^{2}-\bar \lambda_{1}^{2}
-\bar \lambda_{2}^{2}-2\bar \lambda^2)(\bar \lambda_{1}^{2}+\bar \lambda_{2}^{2}+2\bar\lambda^2)\\
&<\bar H^{2}\bigl[\bar H^{2}-2S+\frac{1}{2}\bar H^{4}-\bar K\bar H^{2}-\bar K^{2}\bigl]\\
&\ \ +\bar H^{2}(S+2-\frac{3}{2}\bar H^{2}-\bar \lambda_{1}^{2}
-\bar \lambda_{2}^{2}-2\bar \lambda^2)(\bar \lambda_{1}^{2}+\bar \lambda_{2}^{2}+2\bar\lambda^2).
\endaligned
\end{equation}
According to (4.51), we have
\begin{equation*}
\aligned
&\bar H^{2}\bigl[\bar H^{2}-2S+\frac{1}{2}\bar H^{4}-\bar K\bar H^{2}-\bar K^{2}\bigl]\\
&\ \ + \biggl(S+2-\frac{3}{2}\bar H^{2}-\dfrac{1}{\bar  H^2}\bigl(S(1-\dfrac12S)-(S-\bar H^2)^2+\dfrac12\bar  H^4\bigl)\biggl)\\
&\ \ \times\biggl(S(1-\dfrac12S)-(S-\bar H^2)^2+\dfrac12\bar  H^4\biggl)\\
&=\dfrac{1}{4\bar  H^2}\biggl(\bar  H^8-2S\bar H^6-6S(S-1)\bar  H^4+2S(2-3S)^2\bar  H^2-(2-3S)^2S^2\biggl).
\endaligned
\end{equation*}
We consider a function $f(t)$ defined by
\begin{equation}
\aligned
&f(t)=t^4-2St^3-6S(S-1)t^2+2S(2-3S)^2t-(2-3S)^2S^2,
\endaligned
\end{equation}
for $0<t\leq S$. Thus, we get
\begin{equation}
f^{'}(t)=4t^3-6St^2-12S(S-1)t+2S(2-3S)^2,\ \
f^{''}(t)=12(t^2-St-S(S-1)),
\end{equation}
$f^{''}(t)<0$ for  $t\in (0, S)$.  Hence, $f^{'}(t)$ is a decreasing function
for  $t\in (0, S)$. Since  $f^{'}(S)=4S(S-1)(S-2)=0$,
  $f(t)$ is a increasing function
for $t\in (0, S)$. According to
\begin{equation}
f(S)=2(S-1)(S-2)S^2=0,
\end{equation}
we conclude $f(t)<0$ for $t\in (0, S)$. This  is a contradiction.

\noindent Hence, we have
$\bar \lambda_1\bar \lambda_2\neq 0$ and $\bar \lambda=0$.
In this case, we get
 for $ k=1, 2$,
 \begin{equation}
\begin{cases}
&\bar h^{1^*}_{11k}+\bar h^{1^*}_{22k}=0\\
&\bar h^{2^*}_{11k}+\bar h^{2^*}_{22k}=0\\
&(\bar \lambda_1-3\bar \lambda_2)\bar h^{1^*}_{11k}=0.
\end{cases}
\end{equation}
\vskip1mm
\noindent
If $\bar \lambda_1=3\bar \lambda_2$, we obtain
$$
\inf H^2=\lim_{m\rightarrow\infty}H^2(p_m)=(\bar \lambda_1+\bar \lambda_2)^2=16\bar \lambda_2^2=\dfrac{4S}3,
$$
which  is impossible from Proposition 4.4.
Thus, we get   $\bar \lambda_1\neq 3\bar \lambda_2$. In this case, we have
$$
\bar h^{p^{\ast}}_{ijk}=0
$$
for any $i, j, k, p$ from (4.56).
\vskip1mm
\noindent
From   (2.19), we know
\begin{equation*}
\bar H_{,ij}^{p^{*}}=\bar h_{ij}^{p^{*}}-\sum_k\bar h_{ik}^{p^{*}}\bar h_{jk}^{1^{*}}\bar H
\end{equation*}
because of $H^{1^*}=H$ and $H^{2^*}=0$.  Thus, we  get
\begin{equation}\label{eq:1-31-1}
\begin{aligned}
&\bar h_{1111}^{1^{*}}+\bar h_{2211}^{1^{*}}=\bar \lambda_1-\bar \lambda_1^2\bar H,\\
&\bar h_{1122}^{1^{*}}+\bar h_{2222}^{1^{*}}=\bar \lambda_2-\bar \lambda_2^2\bar H,\\
&\bar h_{1112}^{1^{*}}+\bar h_{2212}^{1^{*}}=\bar h_{1121}^{1^{*}}+\bar h_{2221}^{1^{*}}=0,\\
&\bar h_{1122}^{2^{*}}+\bar h_{2222}^{2^{*}}=0,\\
&\bar h_{1121}^{2^{*}}+\bar h_{2221}^{2^{*}}=\bar \lambda_2-\bar \lambda_1\bar \lambda_2\bar H.
\end{aligned}
\end{equation}
From Ricci identities (2.8), we obtain
\begin{equation}\label{eq:1-31-7}
\begin{aligned}
&\bar h_{1122}^{2^{*}}=\bar h_{2211}^{2^{*}}, \ \bar h_{1112}^{1^{*}}=\bar h_{1121}^{1^{*}},\\
&\bar h_{1112}^{2^{*}}-\bar h_{1121}^{2^{*}}=(\bar \lambda_1-2\bar \lambda_2)\bar K, \\
& \bar h_{2212}^{2^{*}}-\bar h_{2221}^{2^{*}}=3\bar \lambda_2\bar K.
\end{aligned}
\end{equation}
On the other hand, since  $S$ is constant, we know, for $k,l=1,2$,
\begin{equation}\label{eq:1-31-2}
0=-\sum_{i,j,p}\bar h_{ijl}^{p^{*}}\bar h_{ijk}^{p^{*}}=\sum_{i,j,p}\bar h_{ij}^{p^{*}}\bar h_{ijkl}^{p^{*}}
=\bar h_{11}^{1^{*}}\bar h_{11kl}^{1^{*}}+3\bar h_{12}^{2^{*}}\bar h_{12kl}^{2^{*}}=\bar\lambda_1\bar h_{11kl}^{1^{*}}+3\bar \lambda_2\bar h_{22kl}^{1^{*}}.
\end{equation}
From \eqref{eq:1-31-1} and \eqref{eq:1-31-2}, we have
\begin{equation}\label{eq:1-31-6}
\begin{aligned}
&(\bar \lambda_1-3\bar \lambda_2)\bar h_{2211}^{1^{*}}=\bar \lambda_1^2-\bar \lambda_1^3\bar H,\\
&(\bar \lambda_1-3\bar \lambda_2)\bar h_{2222}^{1^{*}}=\bar \lambda_1\bar \lambda_2-\bar \lambda_1\bar \lambda_2^2\bar H,\\
&(\bar \lambda_1-3\bar \lambda_2)\bar h_{1122}^{1^{*}}=-3\bar \lambda_2^2+3\bar \lambda_2^3\bar H\\
&(\bar \lambda_1-3\bar \lambda_2)\bar h_{2212}^{1^{*}}=0.
\end{aligned}
\end{equation}
Hence, we conclude from \eqref{eq:1-31-1} and \eqref{eq:1-31-6}
\begin{equation*}
\bar h_{2212}^{1^{*}}=\bar h_{1112}^{1^{*}}=0
\end{equation*}
because of  $\bar \lambda_1\neq3\bar \lambda_2$.
According to   \eqref{eq:1-31-6}, we obtain
\begin{equation}
(\bar \lambda_1-3\bar \lambda_2)(\bar h_{1112}^{2^{*}}-\bar h_{1121}^{2^{*}})
=-S+(3\bar \lambda_2^3+\bar \lambda_1^3)\bar H
\end{equation}
because of $S=\bar \lambda_1^2+3\bar \lambda_2^2$. \newline
For the case $S\equiv 1$, from \eqref{eq:1-31-7}, we know
\begin{equation}
\begin{aligned}
&(\bar \lambda_1-3\bar \lambda_2)(\bar h_{1112}^{2^{*}}-\bar h_{1121}^{2^{*}})\\
&=(\bar \lambda_1-3\bar \lambda_2)(\bar \lambda_1-2\bar \lambda_2)\bar K\\
&=(1+3\bar \lambda_2^2-5\bar \lambda_1\bar \lambda_2)(\bar \lambda_1\bar \lambda_2-\bar \lambda_2^2)\\
&=\bar \lambda_1\bar \lambda_2-\bar \lambda_2^2+8\bar \lambda_1\bar \lambda_2^3
-3\bar \lambda_2^4-5\bar \lambda_1^2\bar \lambda_2^2\\
&=\bar \lambda_1\bar \lambda_2-2\bar \lambda_2^2+8\bar \lambda_1\bar \lambda_2^3
-4\bar \lambda_1^2\bar \lambda_2^2.\\
\end{aligned}
\end{equation}
By a direct calculation and by  using  $S=\bar \lambda_1^2+3\bar \lambda_2^2=1$, we get
\begin{equation*}
-1+(3\bar \lambda_2^3+\bar \lambda_1^3)\bar H-\bigl\{\bar \lambda_1\bar \lambda_2-\bar \lambda_2^2
+8\bar \lambda_1\bar \lambda_2^3
-3\bar \lambda_2^4-5\bar \lambda_1^2\bar \lambda_2^2\bigl\}
=-8\bar \lambda_1\bar \lambda_2^3\neq 0.
\end{equation*}
From (4.61) and (4.62), it is impossible.\newline
For  the case $S\equiv 2$,  we have from (2.25) in Lemma 2.1,
\begin{equation}
\begin{aligned}
&0=S(1-\dfrac32S)+2\bar H^2S-\dfrac12\bar H^4-\bar H^2(\bar \lambda_1^2+\bar\lambda_2^2)\\
&=-(S-\bar H^2)^2-\dfrac12\bar H^2(\bar \lambda_1-\bar\lambda_2)^2.
\end{aligned}
\end{equation}
We conclude from Proposition 4.4
$$
\inf H^2=\bar H^2=S=\sup H^2.
$$
Thus, we know that $H^2=S$ is constant. \newline
From now on, we consider the case $\bar \lambda_1\bar \lambda_2=\bar \lambda^2$.
In this case, we have
$$S=\bar \lambda_1^2+3\bar \lambda_2^2+4\bar \lambda^2
=(\bar \lambda_1+\bar \lambda_2)(\bar \lambda_1+3\bar \lambda_2)=\bar H(\bar \lambda_1+3\bar \lambda_2).
$$
If $S\equiv 1$, from (2.25), we get
\begin{equation*}
\sum_{i,j,k,p}(\bar h^{p^{\ast}}_{ijk})^2=\dfrac{1}{2}(\bar H^2-1)(3\bar H^2-1)\geq 0.
\end{equation*}
Hence, either $\bar H^2\geq 1$, or $\bar H^2\leq \dfrac{1}{3}$.
If $\bar H^2\geq 1$, then we have $ H^2\equiv 1=S$ since  $\inf H^2=\bar H^2\leq \sup H^2=1$ in view of Proposition 4.4.
According to $S=\bar \lambda_1^2+3\bar \lambda_2^2+4\bar \lambda^2$ and  $H^2=\bar \lambda_1^2+\bar \lambda_2^2+2\bar \lambda^2$, we know $\bar \lambda=0$ and $\bar \lambda_2=0$. \newline
If $\bar H^2\leq \dfrac{1}{3}$, from $S=\bar H(\bar \lambda_1+3\bar \lambda_2)=1$,
we obtain  $(\bar \lambda_1+3\bar \lambda_2)^2\geq 3$, which  implies
$\bar \lambda_1=\bar \lambda=0$ because of  $(\bar \lambda_1+3\bar \lambda_2)^2=
\bar \lambda_1^2+9\bar \lambda_2^2+6\bar \lambda^2\leq 3\bar \lambda_1^2+9\bar \lambda_2^2+12\bar \lambda^2=3S=3$.
Hence, we have $\inf H^2=\bar\lambda_2^2=\dfrac{S}3\neq 0$,
$$
\bar H^{1^*}_{,k}=0,  \ \ \bar H^{2^*}_{,1}=0
$$
because of  $H^{p^{\ast}}_{,i}=\sum_k h^{p^{\ast}}_{ik} \langle X,e_k \rangle$.
Hence, we have, by using the same calculations as in (4.6),
\begin{equation}
\begin{cases}
&\bar h^{1^*}_{11k}+\bar h^{1^*}_{22k}=0\\
&\bar h^{2^*}_{111}+\bar h^{2^*}_{221}=0\\
&3\bar \lambda_2\bar h^{1^*}_{22k}=0.
\end{cases}
\end{equation}
Hence, we get
$$
\bar h^{p^{\ast}}_{ijk}=0, \ \  \text{\rm except} \ \ i=j=k=p^{\ast}=2.
$$
If  $\bar h^{2^{\ast}}_{222}\neq 0$,  since $\bar \lambda_2\neq0$, $3\bar H^2=3\bar \lambda_2^2=S=1$, we have
\begin{equation*}
0=\sum_{i,j,k,p}(\bar h_{ijk}^{p^{*}})^2+S(1-\frac{3}{2}S)+2\bar H^{2}S-\frac{1}{2}\bar H^{4}
-\bar H^{2}\bar \lambda_2^2=\sum_{i,j,k,p}(\bar h_{ijk}^{p^{*}})^2>0.
\end{equation*}
It is impossible.
Hence, we know
$$
\bar h^{p^{\ast}}_{ijk}=0,
$$
for any $ i, j, k, p$.
From (2.19), we get
\begin{equation*}
\bar H_{,ij}^{p^{*}}=\bar h_{ij}^{p^{*}}-\sum_k\bar h_{ik}^{p^{*}}\bar h_{jk}^{1^{*}}\bar H.
\end{equation*}
We obtain
\begin{equation*}
\bar h_{1121}^{2^{*}}+\bar h_{2221}^{2^{*}}=\bar \lambda_2\neq 0.
\end{equation*}
From (2.34) of Lemma 2.3, we have
\begin{equation}
\begin{aligned}
&\frac{1}{2}\mathcal{L}\sum_{i,j,k,p}(\bar h_{ijk}^{p^{\ast}})^{2}=\sum_{i,j,k,l,p}(\bar h_{ijkl}^{p^{\ast}})^{2}>0.\\
\end{aligned}
\end{equation}
From (2.35) of Lemma 2.3, we get
\begin{equation*}
\aligned
&\frac{1}{2}\mathcal{L}\sum_{i,j,k,p}(\bar h_{ijk}^{p^{\ast}})^{2}\\
=&(\bar H^{2}-2S)\bar H^{2}+(3\bar K+2-\bar H^{2}+2S)\sum_{i,j}\bar H^2\bar h_{ij}^{1^{\ast}} \bar h_{ij}^{1^{\ast}}\\
&-\bar K(\bar H^{4}+\bar H^3\bar h_{11}^{1^{\ast}})
-\sum_{i,j,k,l,p}\bar H^2\bar h_{jk}^{1^{\ast}}\bar h_{jk}^{p^{\ast}}\bar h_{il}^{p^{\ast}}\bar h_{il}^{1^{\ast}}
-\sum_{i,j,k}\bar H^3 \bar h_{ik}^{1^{\ast}}\bar h_{ji}^{1^{\ast}}\bar h_{jk}^{1^{\ast}}\\
=&(\bar H^{2}-2S)\bar H^{2}+(3\bar K+2-\bar H^{2}+2S)\sum_{i,j}\bar H^2\bar \lambda_2^2
-\bar K\bar H^{4}-\bar H^2\bar \lambda_2^4-\bar H^3 \bar \lambda_2^3 \\
=&\bar H^6-3\bar H^4=-\frac{8}{27},
\endaligned
\end{equation*}
which is in contradiction to  (4.65).
Hence, we get  $\inf H^2=S$, that is,  $H^2=S$ is constant from Proposition 4.4.  \newline
For the case $S\equiv 2$,  first of all, we will prove $\bar\lambda=0$. If not, we have $S=2$ and $\bar \lambda_1\bar \lambda_2=\bar \lambda^2\neq0$.
By making use of the same assertion as in the proof of Theorem 4.3,
we have
\begin{equation}
\begin{pmatrix}
&- \bar \lambda & 3 \bar \lambda_2 &3\bar  \lambda&  \bar \lambda_1 &0\\
&0&-\bar \lambda & 3\bar \lambda_2 &3 \bar \lambda& \bar \lambda_1 \\
&0&0&1&0&1\\
&0&1&0&1&0\\
&1&0&1&0&0
\end{pmatrix}
\begin{pmatrix}
\bar h_{222}^{2^*}\\
\bar h_{222}^{1^*}\\
\bar h_{122}^{1^*}\\
\bar h_{112}^{1^*}\\
\bar h_{111}^{1^*}
\end{pmatrix}
=
\begin{pmatrix}
0\\
0\\
0\\
0\\
A
\end{pmatrix},
\end{equation}
where
\begin{equation}
\begin{aligned}
A=&\bar \lambda_2 \lim_{m\rightarrow\infty} \langle X,e_1 \rangle(p_m)
-\bar \lambda \lim_{m\rightarrow\infty} \langle X,e_2 \rangle(p_m)\\
=&\bar H\lim_{m\rightarrow\infty} \langle X,e_1 \rangle(p_m)\\
=&-\dfrac{H\bar \lambda}{\bar \lambda_1}\lim_{m\rightarrow\infty} \langle X,e_2 \rangle(p_m).
\end{aligned}
\end{equation}
Solving this system of linear equations,
we have
\begin{equation}
\begin{aligned}
\mu \bar h_{222}^{2^*}=&\bigl\{12\bar\lambda^2+(\bar \lambda_1-3\bar \lambda_2)^2\bigl\}A , \\
\bar h_{222}^{1^*}=&-\bar h_{211}^{1^*},\\
\bar h_{221}^{1^*}=&-\bar h_{111}^{1^*}, \\
\mu\bar h_{211}^{1^*}=&\bar \lambda(\bar \lambda_1-3\bar \lambda_2)A, \\
\mu\bar h_{111}^{1^*}=&-4\bar \lambda^2A \\
\end{aligned}
\end{equation}
with $\mu=16\bar\lambda^2+(\bar \lambda_1-3\bar \lambda_2)^2$.
\begin{equation}
\begin{aligned}
\lim_{m\rightarrow\infty}|\nabla^{\perp} \vec H|^2&=\sum_{i,p}(\bar H^{p^*}_{,i})^2=(\bar h_{112}^{2^*}+\bar h_{222}^{2^*})^2=A^2
\end{aligned}
\end{equation}
\begin{equation}
\begin{aligned}
&\sum_{i,j,k,p}(\bar h^{p^*}_{ijk})^2=\sum_p(\bar h_{111}^{p^*})^2+\sum_p(\bar h_{222}^{p^*})^2+3\sum_p(\bar h_{221}^{p^*})^2
+3\sum_p(\bar h_{112}^{p^*})^2\\
&=\sum_{p}(\bar h_{11p}^{1^*})^2+(\bar h_{222}^{1^*})^2+(\bar h_{222}^{2^*})^2
+3\sum_p(\bar h_{22p}^{1^*})^2+3(\bar h_{112}^{1^*})^2+3(\bar h_{112}^{2^*})^2\\
&=\lim_{m\rightarrow\infty}|\nabla^{\perp} \vec H|^2.
\end{aligned}
\end{equation}
Since  $S= 2$ and $\bar \lambda_1\bar \lambda_2=\bar \lambda^2\neq0$, we obtain
$$
S=\bar \lambda_1^2+3\bar \lambda_2^2+4\bar \lambda_3^2=\bar H(\bar \lambda_1+3\bar \lambda_2)=2,
$$
From (2.25), we have
\begin{equation}
\sum_{i,j,k,p}(\bar h^{p^{\ast}}_{ijk})^2=4-4\bar H^2+\frac{3}{2}\bar H^4=(2-H^2)^2+\dfrac12\bar H^4>0.
\end{equation}
Since  $S=2$ is constant, we get, for $k,l=1,2$,
\begin{equation*}
-\sum_{i,j,p}\bar h_{ijl}^{p^{*}}\bar h_{ijk}^{p^{*}}=\sum_{i,j,p}\bar h_{ij}^{p^{*}}\bar h_{ijkl}^{p^{*}}
=\bar\lambda_1\bar h_{11kl}^{1^{*}}
+3\bar\lambda\bar h_{12kl}^{1^{*}}+3\bar \lambda_2\bar h_{22kl}^{1^{*}}-\bar \lambda\bar h_{22kl}^{2^{*}},
\end{equation*}
namely,
\begin{equation}
\begin{aligned}
&\bar \lambda_1 \bar h_{1122}^{1^{*}}+3\bar \lambda \bar h_{1222}^{1^{*}}+3\bar \lambda_2 \bar h_{2222}^{1^{*}}
-\bar \lambda \bar h_{2222}^{2^{*}}\\
&=-(\bar h_{112}^{1^{*}})^2-3(\bar h_{112}^{2^{*}})^2-3(\bar h_{212}^{2^{*}})^2-(\bar h_{222}^{2^{*}})^2\\
&=-\dfrac{12\bar\lambda^2+(\bar \lambda_1-3\bar \lambda_2)^2}{\mu}A^2,\\
&\bar \lambda_1\bar h_{1112}^{1^{*}}+3\bar \lambda\bar h_{1212}^{1^{*}}
+3\bar \lambda_2 \bar h_{2212}^{1^{*}}-\bar \lambda \bar h_{2212}^{2^{*}}\\
&=-\bar h_{111}^{1^{*}}\bar h_{112}^{1^{*}}-3\bar h_{111}^{2^{*}}\bar h_{112}^{2^{*}}
-3\bar h_{121}^{2^{*}}\bar h_{122}^{2^{*}}-\bar h_{221}^{2^{*}}\bar h_{222}^{2^{*}}\\
&=\dfrac{\bar\lambda(\bar\lambda_1-3\bar\lambda_2)}{\mu}A^2.\\
\end{aligned}
\end{equation}
From (2.19) and taking limit,  we know, for $\ i, j, p=1, 2$
\begin{equation*}
\bar h_{11ij}^{p^{*}}+\bar h_{22ij}^{p^{*}}=\sum_k\bar h_{ikj}^{p^{*}}\lim_{m\rightarrow\infty} \langle X,e_k \rangle(p_m)
+\bar h_{ij}^{p^{*}}-\sum_k\bar h_{ik}^{p^{*}}\bar h_{jk}^{1^{*}}\bar H,
\end{equation*}
it means that,
\begin{equation*}
\begin{aligned}
& \bar h_{1122}^{1^{*}}+\bar h_{2222}^{1^{*}}=
\sum_k\bar h_{2k2}^{1^{*}}\lim_{m\rightarrow\infty} \langle X,e_k \rangle(p_m) +
\bar \lambda_2 -(\bar \lambda_2^2+\bar \lambda^2) \bar H, \\
&\bar h_{1112}^{1^{*}}+\bar h_{2212}^{1^{*}}=
\sum_k\bar h_{1k2}^{1^{*}}\lim_{m\rightarrow\infty} \langle X,e_k \rangle(p_m)+\bar \lambda -\bar \lambda\bar H^2, \\
& \bar h_{1122}^{2^{*}}+\bar h_{2222}^{2^{*}}=\sum_k\bar h_{2k2}^{2^{*}}\lim_{m\rightarrow\infty} \langle X,e_k \rangle(p_m)-\bar \lambda, \\
\end{aligned}
\end{equation*}
Since, from (4.49) and (4.67) and $S=\bar H(\bar \lambda_1+3\bar \lambda_2)=2$,
$$
\sum_k\bar h_{2k2}^{1^{*}}\lim_{m\rightarrow\infty} \langle X,e_k \rangle(p_m)=\dfrac{\bar\lambda_1A^2}{\mu},
$$
$$
\sum_k\bar h_{2k2}^{2^{*}}\lim_{m\rightarrow\infty} \langle X,e_k \rangle(p_m)=
\dfrac{3\bar \lambda}{\mu }A^2-
\dfrac{\bar \lambda_1^2+3\bar\lambda^2}{2 \bar\lambda}A^2,
$$
$$
\sum_k\bar h_{1k2}^{1^{*}}\lim_{m\rightarrow\infty} \langle X,e_k \rangle(p_m)
=-\dfrac{3\bar\lambda}{\mu}A^2,
$$
we obtain
\begin{equation}
\begin{cases}
\begin{aligned}
&\bar \lambda_1 \bar h_{1122}^{1^{*}}+3\bar \lambda \bar h_{1222}^{1^{*}}+3\bar \lambda_2 \bar h_{2222}^{1^{*}}
-\bar \lambda \bar h_{2222}^{2^{*}}=-\dfrac{12\bar\lambda^2+(\bar \lambda_1-3\bar \lambda_2)^2}{\mu}A^2,\\
&\bar \lambda_1\bar h_{1112}^{1^{*}}+3\bar \lambda\bar h_{1212}^{1^{*}}
+3\bar \lambda_2 \bar h_{2212}^{1^{*}}-\bar \lambda \bar h_{2212}^{2^{*}}
=\dfrac{\bar\lambda(\bar\lambda_1-3\bar\lambda_2)}{\mu}A^2.\\
& \bar h_{1122}^{1^{*}}+\bar h_{2222}^{1^{*}}=
\dfrac{\bar\lambda_1A^2}{\mu}+
\bar \lambda_2 -(\bar \lambda_2^2+\bar \lambda^2) \bar H, \\
&\bar h_{1112}^{1^{*}}+\bar h_{2212}^{1^{*}}=-
\dfrac{3\bar\lambda}{\mu}A^2+\bar \lambda -\bar \lambda\bar H^2, \\
& \bar h_{1122}^{2^{*}}+\bar h_{2222}^{2^{*}}=\dfrac{3\bar \lambda}{\mu }A^2-
\dfrac{\bar \lambda_1^2+3\bar\lambda^2}{2 \bar\lambda}A^2-\bar \lambda. \\
\end{aligned}
\end{cases}
\end{equation}
In view of rom (4.73), we know
\begin{equation*}
\begin{aligned}
& 4\bar \lambda\bar h_{2222}^{2^{*}}+(\bar \lambda_1-3\bar \lambda_2)\bar h_{2222}^{1^{*}}\\
=&\bigl\{\dfrac{15\bar\lambda^2+2\bar \lambda_1^2+9\bar \lambda_2^2}{\mu}
-\dfrac{3(\bar \lambda_1^2+3\bar\lambda^2)}{2}\bigl\}A^2
-2\bar \lambda^2-\bar \lambda^2\bar H^2\\
&  (\bar \lambda_1-3\bar \lambda_2)\bar h_{2222}^{2^{*}}-4\bar \lambda\bar h_{2222}^{1^{*}}\\
=&\bigl\{\dfrac{4\bar\lambda(\bar\lambda_1-3\bar\lambda_2)}{\mu}
-\dfrac{(\bar \lambda_1+3\bar\lambda_2)(\bar \lambda_1-3\bar\lambda_2)\bar\lambda_1}{2\bar \lambda}\bigl\}A^2
+2\bar \lambda\bar\lambda_2.\\
\end{aligned}
\end{equation*}
Thus, we have
\begin{equation*}
\begin{aligned}
&\bigl\{ -16\bar \lambda^2+(\bar \lambda_1-3\bar \lambda_2)^2\bigl\}\bar h_{2222}^{2^{*}}-8\bar\lambda(\bar \lambda_1-3\bar \lambda_2)\bar h_{2222}^{1^{*}}\\
=&-4\bar\lambda\bigl\{\dfrac{15\bar\lambda^2+2\bar \lambda_1^2+9\bar \lambda_2^2}{\mu}
-\dfrac{3(\bar \lambda_1^2+3\bar\lambda^2)}{2}\bigl\}A^2
+8\bar \lambda^3+4\bar \lambda^3\bar H^2\\
+&\bigl(\dfrac{4\bar\lambda}{\mu}
-\dfrac{(\bar \lambda_1+3\bar\lambda_2)\bar\lambda_1}{2\bar \lambda}\bigl)(\bar \lambda_1-3\bar\lambda_2)^2A^2
+2\bar \lambda\bar\lambda_2(\bar\lambda_1-3\bar\lambda_2)\\
&=\bigl\{-4\bar\lambda\dfrac{21\bar\lambda^2+\bar \lambda_1^2}{\mu}
+\dfrac{(\bar \lambda_1^2+3\bar\lambda^2)(18\bar\lambda^2-\bar\lambda_1^2-9\bar\lambda_2^2)}{2\bar\lambda}\bigl\}A^2
+10\bar \lambda^3-6\bar \lambda\bar \lambda_2^2+4\bar \lambda^3\bar H^2\\
\end{aligned}
\end{equation*}
Taking covariant differentiation of  (2.25) and using (4.47) and  (4.48), we obtain
\begin{equation*}
\begin{aligned}
0=&\bar h_{111}^{1^{*}} \bar h_{1112}^{1^{*}}+4\bar h_{111}^{2^{*}} \bar h_{1122}^{1^{*}}+6\bar h_{122}^{1^{*}}\bar h_{1222}^{1^{*}}
+4\bar h_{222}^{1^{*}}\bar h_{2222}^{1^{*}}+\bar h_{222}^{2^{*}}\bar h_{2222}^{2^{*}}\\
 &-\bar H^2
\bigl\{\bar \lambda (\bar h_{112}^{2^{*}}+\bar h_{222}^{2^{*}})
+\bar \lambda_1 \bar h_{111}^{2^{*}}+2\bar \lambda\bar h_{122}^{1^{*}}+\bar \lambda_2\bar h_{222}^{1^{*}}\bigl\}\\
&
\end{aligned}
\end{equation*}
 Since
 \begin{equation*}
\begin{aligned}
&\bar H^2 \bigl\{\bar \lambda (\bar h_{112}^{2^{*}}+\bar h_{222}^{2^{*}})
+\bar \lambda_1 \bar h_{111}^{2^{*}}+2\bar \lambda\bar h_{122}^{1^{*}}+\bar \lambda_2\bar h_{222}^{1^{*}}\bigl\}=\bar H^2 \dfrac{\bar \lambda A}{\mu}(2+\mu),
\end{aligned}
\end{equation*}
\begin{equation*}
 \begin{aligned}
&\bar h_{111}^{1^{*}} \bar h_{1112}^{1^{*}}+4\bar h_{111}^{2^{*}} \bar h_{1122}^{1^{*}}+6\bar h_{122}^{1^{*}}\bar h_{1222}^{1^{*}}
+4\bar h_{222}^{1^{*}}\bar h_{2222}^{1^{*}}+\bar h_{222}^{2^{*}}\bar h_{2222}^{2^{*}}\\
 &=\dfrac{A}{\mu}\biggl(\bigl\{ -16\bar \lambda^2+(\bar \lambda_1-3\bar \lambda_2)^2\bigl\}\bar h_{2222}^{2^{*}}-8\bar\lambda(\bar \lambda_1-3\bar \lambda_2)\bar h_{2222}^{1^{*}}\biggl)\\
 &+\dfrac{\bar\lambda A^3}{\mu}\bigl\{\dfrac{84\bar\lambda^2+4\lambda_1^2}{\mu}-14(\bar\lambda_1^2+3\bar\lambda^2)\bigl\}
 +\dfrac{4\bar\lambda A}{\mu}(-7\bar\lambda^2-3\lambda_2^2+3\bar\lambda_2^2\bar H^2),\\
&
\end{aligned}
\end{equation*}
we have
\begin{equation*}
 \begin{aligned}
&\bar H^2 \dfrac{\bar \lambda A}{\mu}(2+\mu)\\
 &=\dfrac{A}{\mu}\biggl(\bigl\{-4\bar\lambda\dfrac{21\bar\lambda^2+\bar \lambda_1^2}{\mu}
+\dfrac{(\bar \lambda_1^2+3\bar\lambda^2)(18\bar\lambda^2-\bar\lambda_1^2-9\bar\lambda_2^2)}{2\bar\lambda}\bigl\}A^2\\
&+10\bar \lambda^3-6\bar \lambda\bar \lambda_2^2+4\bar \lambda^3\bar H^2\biggl)\\
 &+\dfrac{\bar\lambda A^3}{\mu}\bigl\{\dfrac{84\bar\lambda^2+4\lambda_1^2}{\mu}-14(\bar\lambda_1^2+3\bar\lambda^2)\bigl\}
 +\dfrac{4\bar\lambda A}{\mu}(-7\bar\lambda^2-3\lambda_2^2+3\bar\lambda_2^2\bar H^2)\\
 &=\dfrac{ A^3}{\mu}\dfrac{(\bar \lambda_1^2+3\bar\lambda^2)(-10\bar\lambda^2-\bar\lambda_1^2-9\bar\lambda_2^2)}{2\bar\lambda}
 +\dfrac{\bar\lambda A}{\mu}(-18\bar\lambda^2-18\lambda_2^2+4(\bar\lambda^2+3\bar\lambda_2^2)\bar H^2),
\end{aligned}
\end{equation*}
which  is  impossible because of $2+\mu=2\bar\lambda_1^2+12\bar\lambda_2^2+14\bar\lambda^2$.
Hence, we have $\bar\lambda=0$, that is,   $\bar\lambda_1 \bar\lambda_2=0$. \newline
If $\bar \lambda_2=0$, we get $\inf H^2=S=\sup H^2$ from Proposition 4.4.  Namely, $H^2=S$ is constant.
\vskip1mm
\noindent
If $\bar \lambda_1=0$,  we have
$$
\bar \lambda_2\neq 0,  \ \ S=3\inf H^2, \ \
\bar H^{1*}_{,k}=0, \ k=1,2.
$$
Hence, we have, by using the same calculations as in (4.6),
\begin{equation}
\begin{cases}
&\bar h^{1^*}_{11k}+\bar h^{1^*}_{22k}=0\\
&\bar h^{2^*}_{111}+\bar h^{2^*}_{221}=0\\
&3\bar \lambda_2\bar h^{1^*}_{22k}=0.
\end{cases}
\end{equation}
Hence, we have
$$
\bar h^{p^{\ast}}_{ijk}=0, \ \ \text{\rm except} \ \ i=j=k=p^{\ast}=2.
$$
If $\bar h^{2^{\ast}}_{222}=0$, we get
$$
\bar h^{p^{\ast}}_{ijk}=0,
$$
for any $i, j, k, p$.  According to Lemma 2.1, we have
\begin{equation*}
0=\sum_{i,j,k,p} (h_{ijk}^{p^{\ast}})^{2}+S(1-\frac{3}{2}S)+2H^{2}S-\frac{1}{2}H^{4}
-\sum_{j,k,p,q}H^{p^{\ast}}h_{jk}^{p^{\ast}}H^{q^{\ast}}h_{jk}^{q^{\ast}}
=-2.
\end{equation*}
This is impossible.
\vskip1mm
\noindent
If  $\bar h^{2^{\ast}}_{222}\neq0$,  from Lemma 2.1, we obtain
$$
|\nabla^{\perp} \vec H|^2=\sum_{i,j,k,p}(\bar h^{p^{\ast}}_{ijk})^2=(\bar h^{2^{\ast}}_{222})^2=2.
$$
Since $S=\sum_{i,j,p}(h_{ij}^{p^{*}})^2$ is constant, we have
\begin{equation*}
\sum_{i,j,p}h_{ij}^{p^{*}}h_{ijk}^{p^{*}}=0, \ k=1, 2
\end{equation*}
and
\begin{equation*}
\sum_{i,j,p}h_{ijl}^{p^{*}}h_{ijk}^{p^{*}}+\sum_{i,j,p}h_{ij}^{p^{*}}h_{ijkl}^{p^{*}}=0, \ k, l=1, 2.
\end{equation*}
Then, for $ k, l=1, 2$,  we get
\begin{equation*}
\sum_{i,j,p}\bar h_{ijl}^{p^{*}}\bar h_{ijk}^{p^{*}}=-\sum_{i,j,p}\bar h_{ij}^{p^{*}}\bar h_{ijkl}^{p^{*}}
=-\bar h_{22}^{1^{*}}\bar h_{22kl}^{1^{*}}-2\bar h_{12}^{2^{*}}\bar h_{12kl}^{2^{*}}=-3\bar \lambda_2\bar h_{22kl}^{1^{*}}.
\end{equation*}
If $k=l=1$, we have
\begin{equation}\label{eq:1}
\bar h_{2211}^{1^{*}}=0.
\end{equation}
From (2.19), we know
\begin{equation*}
\bar H_{,ij}^{p^{*}}=\bar h_{ij}^{p^{*}}-\sum_k\bar h_{ik}^{p^{*}}\bar h_{jk}^{1^{*}}\bar H.
\end{equation*}
Let $p=i=2, j=1$, we get
\begin{equation*}
\bar h_{1121}^{2^{*}}+\bar h_{2221}^{2^{*}}=\bar \lambda_2.
\end{equation*}
From \eqref{eq:1}, we obtain
\begin{equation*}
\bar h_{2221}^{2^{*}}=\bar \lambda_2\neq0.
\end{equation*}
On the other hand, from Lemma 2.1, we have
\begin{equation*}
2\sum_{i,j,k,p}\bar h_{ijk}^{p^{*}}\bar h_{ijk1}^{p^{*}}=(\sum_{j,k,p,q}\bar H^{p^{\ast}}\bar h_{jk}^{p^{\ast}}\bar H^{q^{\ast}}\bar h_{jk}^{q^{\ast}})_{,1}=0
\end{equation*}
because  $\bar H_{,i}^{1^{\ast}}=0, h_{ij1}^{q^{\ast}}=0$.
Since
\begin{equation*}
2\sum_{i,j,k,p}\bar h_{ijk}^{p^{*}}\bar h_{ijk1}^{p^{*}}=2\bar h_{222}^{2^{*}}\bar h_{2221}^{2^{*}}=2\bar \lambda_2\bar h_{222}^{2^{*}} \neq0,
\end{equation*}
it is a contradiction. Thus, we know that $H^2=S$ is constant.
\end{proof}

\vskip3mm
\noindent
{\it Proof of Theorem 1.1}. From  Theorem 4.1 and Theorem 4.2, we know $S=0$, $S=1$ or $S=2$.
According to  the result of Cheng and Peng \cite{CP}, we only consider
the case $S\equiv 2$ and $S\equiv 1$.
Therefore, the mean curvature $H^2=S$ is constant from Theorem 4.4. \newline
If $H^2=S=2$, from (2.25) in Lemma 2.1, we have
$$
\sum_{i,j,k,p}(h^{p^{\ast}}_{ijk})^2\equiv 0,  \ \  \lambda_1=\lambda_2\neq 0.
$$
According to
$$
H^{p^*}_{,i}=\sum_{k}h_{ik}^{p^*}\langle X, e_k\rangle, \ \ \text{\rm for} \ \ i, p=1, 2,
$$
we know, at any point,
$$
0=h_{11i}^{1^*}+h_{22i}^{1^*}=H^{1^*}_{,i}=\lambda_i\langle X, e_i\rangle.
$$
Hence, we get $\langle X, e_i\rangle=0$ for $i=1, 2$ at any point. Thus,
$|X|^2$ is constant.  According to
$$
\dfrac12\mathcal{L} |X|^2=2-|X|^2
$$
we obtain
$$
|X|^2\equiv 2,
$$
it means that, $X: M^2\to \mathbb{R}^{4}$  becomes  a complete surface  in the sphere $S^3(\sqrt2)$.
Because  $S=2$,
it is easy to prove that  $X: M^2\to S^3(\sqrt2)$ is minimal and its Gaussian curvature is zero.
Thus, we conclude that
$X: M^2\to \mathbb{R}^{4}$ is the Clifford torus $S^1(1)\times S^1(1)$. \newline
If $H^2=S=1$, from (2.25) in Lemma 2.1, we have
$$
\sum_{i,j,k,p}(h^{p^{\ast}}_{ijk})^2\equiv 0,  \ \  \lambda=0, \ \ \lambda_2= 0.
$$
From the results of Yau in \cite{Y}, we know that $X:M^2\rightarrow \mathbb{R}^4$ is $S^1(1)\times \mathbb{R}^1$.
It completes the proof of Theorem 1.1.
\begin{flushright}
$\square$
\end{flushright}

\end{document}